\documentclass[11pt,a4paper,final]{amsart}
\usepackage[latin1]{inputenc}
\usepackage{amsmath}
\usepackage{amsfonts}
\usepackage{amssymb}
\usepackage{amsthm,mathtools,stmaryrd}
\usepackage{mathrsfs}
\usepackage{graphicx}
\usepackage{a4wide}
\usepackage{bbm}
\usepackage{todo}
\usepackage{latexsym}
\usepackage{cite}
\usepackage{mathrsfs}
\usepackage{bbm}
\usepackage{subfigure}
\usepackage{float}
\usepackage{epsfig}
\usepackage{epstopdf}
\usepackage{mathtools}
\usepackage{listings}
\usepackage{pgf}
\usepackage[T1]{fontenc}
\usepackage[latin1]{inputenc}
\usepackage[english]{babel}
\usepackage{pifont,tabularx}
\usepackage{stmaryrd}
\usepackage{dsfont}
\usepackage{theoremref}
\usepackage{nicefrac}
\usepackage{enumerate}
\usepackage{natbib}
\usepackage{subfigure}

\usepackage{tikz}
\usepackage{pgfplots}
\usepackage{pgfplotstable}
\pgfplotsset{compat=1.7}
\usepackage{caption}
\usepackage{hyperref}
\hypersetup{
	colorlinks,
	citecolor=blue,
	linkcolor=blue,
	urlcolor=teal}

\numberwithin{equation}{section}

\theoremstyle{plain} 
\newtheorem{thm}{Theorem}[section]
\newtheorem{prop}[thm]{Proposition}
\newtheorem{cor}[thm]{Corollary}
\newtheorem{lem}[thm]{Lemma}

\newtheorem{obs}[thm]{Remark}

\theoremstyle{remark}

\usepackage{graphicx} 

\title[Phase transition for collisions on combs]{On the phase transition for the number of collisions on comb graphs}

\author{Umberto de Ambroggio}
\address{Department of Mathematics, National University of Singapore}
\curraddr{10 Lower Kent Ridge Road, National University of Singapore}
\email{umberto@nus.edu.sg}
\thanks{}

\author{Jenson Ng}
\address{Department of Mathematics, The Hong Kong University}
\curraddr{Pokfulam Road, Hong Kong}
\email{u3014692@connect.hku.hk}
\thanks{}

\author{Maximilian Nitzschner}
\address{Department of Mathematics, The Hong Kong University of Science and Technology}
\curraddr{Clear Water Bay, Kowloon, Hong Kong}
\email{mnitzschner@ust.hk}
\thanks{}

\author{Carlo Scali}
\address{School of Computation, Information and Technology, Technische Universit{\"a}t M{\"u}nchen}
\curraddr{Boltzmannstra{\ss}e 3, 85748, Garching bei M{\"u}nchen}
\email{carlo.scali@tum.de}
\thanks{}

\begin{document}
\begin{abstract}
We consider collisions of simple random walks on comb graphs $\mathrm{Comb}(\mathbb{Z},H)$, which are obtained by attaching vertical segments of the form $[0,H_x] \cap \mathbb{Z}$ to any point $x$ of the integer axis. For $\mathrm{Comb}(\mathbb{Z},H)$ with profile $H_x(x) = |x| \log^\gamma(|x| \vee 1)$, we show that two independent simple random walks starting from the same site collide infinitely often almost surely if $\gamma \leq 2$. If the tooth profile is taken as a typical realization of i.i.d.~heavy-tailed random variables with $\textbf{P}(H_x > z) \sim Cz^{-\gamma}$ (with some $C > 0$) as $z$ tends to infinity, we show that infinitely many collisions occur almost surely for two independent random walks if $\gamma > 1/3$, whereas finitely many collisions occur almost surely if $\gamma \in (0,1/3)$, and for any $\gamma \in (0,1]$, three independent random walks only collide finitely many times, almost surely. 
\bigskip


\end{abstract}
\maketitle

\section{Introduction}

In this article, we study the number of collisions of independent, discrete-time simple random walks on comb graphs defined over the integer line $\mathbb{Z}$ as a base graph. Given a sequence $H\coloneqq (H_x)_{x \in \mathbb{Z}} \in \mathbb{R}_{\geq 0}^{\mathbb{Z}}$, a comb graph $\mathrm{Comb}(\mathbb{Z},H)$ is constructed by attaching to each point $x$ of $\mathbb{Z}$ a vertical segment (called `tooth') of the form $\{0,1,...,\lfloor H_x\rfloor\}$ (with $\lfloor \cdot \rfloor$ denoting the integer part).
Our interest is both in comb graphs with deterministic, regularly growing tooth profiles, and in tooth profiles 
that appear as a realization of i.i.d.~heavy-tailed random variables. \medskip

The problem of collisions of random walks on comb graphs goes back at least to~\cite{KP}. There, it is proved that on a comb graph obtained by attaching \textit{infinite} teeth to each vertex of $\mathbb{Z}$, two independent random walks starting from the same vertex collide only finitely many times almost surely. Since then, several criteria on finite length tooth profiles $(H_x)_{x \in \mathbb{Z}}$ have been developed to distinguish between the (almost sure) occurrence or absence of infinitely many collisions of two or three random walks on $\mathrm{Comb}(\mathbb{Z},H)$, see~\cite{BPS,Chenchen,CroydonDeAmbroggio,CWZ08,Koops}. Collisions of two or multiple walks have furthermore been investigated on comb graphs with a radially symmetric tooth profile over other planar graphs including $\mathbb{Z}^2$ (see~\cite{de2025collisions}), as well as on several classes of random recurrent graphs (see~\cite{astoquillca2026collisions,BPS,ChenChen10,croydon2026collision,HP,Wat23}), or in certain random or directionally biased media (see~\cite{MR3542625,devulder2025infinitely,DGP1,DGP2,HH}). In a slightly different direction, the spatial distribution of the collisions has been studied extensively in~\cite{Nguyen,noda2025convergence,noda2026} for $\mathbb{Z}$ as well as certain random graphs. The graph $\mathrm{Comb}(\mathbb{Z},H)$ in which the tooth lengths $(H_x)_{x \in \mathbb{Z}}$ are a realization of a random i.i.d.~(bi-infinite) sequence is of special significance, since the horizontal coordinates of each simple random walk on the comb may be viewed as a random walk on $\mathbb{Z}$ with `transparent traps'. To that effect, the random walk on $\mathrm{Comb}(\mathbb{Z},H)$ has also received attention in the context of scaling limits of trapped walks, see, e.g.,~\cite{BCCR, Bertacchi}. Finally, we mention that the problem of collisions on comb graphs has appeared in the physics literature (see, e.g.,~\cite{agliari2016two,peng2019first} and references therein), as a natural model to study diffusion-limited reactions. The aim of this article is to prove the existence of several phase transitions for the number of collisions of random walks for both regularly growing combs, and combs with random tooth lengths, that manifestly go \textit{beyond} the criteria in~\cite{BPS,Chenchen}. In particular, our results answer a question in~\cite[Section 6]{BPS} concerning comb graphs with random tooth profiles. As a byproduct, we obtain precise heat kernel bounds for random walks on $\mathrm{Comb}(\mathbb{Z},H)$ in this case, which may be of independent interest. \medskip

We now state our set-up and our results in a more detailed way. Let $G = (V,E)$ be a connected, locally finite graph, and consider $j \in \mathbb{N}$ independent discrete-time simple random walks $(X^j_n)_{n \geq 0}$ starting from $x \in V$ under the measure $P_{x,...,x}^G$ (we refer to Section~\ref{sec:Notation} for further details on the notation). We say that the graph $G$ possesses the \textit{infinite collision property} resp.~\textit{infinite triple collision property}, if the set of collision times $\{j \in \mathbb{N} \, : \, X^1_j = X^2_j \}$ is infinite $P_{x,x}^G$-almost surely, resp.~if the set of triple collision times $\{j \in \mathbb{N} \, : \, X^1_j = X^2_j = X^3_j \}$ is infinite $P_{x,x,x}^G$-almost surely, for some $x \in V$. If the corresponding sets are finite $P_{x,x}^G$- (resp.~$P_{x,x,x}^G$-)almost surely for some $x \in V$, we say that $G$ has the \textit{finite collision property}, resp.~\textit{finite triple collision property}. One can show a $0$-$1$-law asserting that any graph has either the finite or the infinite collision property, and this does not depend on the choice of the point $x \in V$ (see \cite[Proposition 2.1]{BPS}). Moreover, we note in passing that while we have chosen to study the collision properties of discrete-time simple random walks, it was established very recently in~\cite{astoquillca2026-2} that for bounded degree graphs, the corresponding notion of infinitely many collisions of two continuous-time simple random walks is equivalent to the infinite collision property stated above. \smallskip 

The graphs studied in the present article are comb graphs $\mathrm{Comb}(\mathbb{Z},H)$, which are defined by specifying a (possibly random) profile function $H : \mathbb{Z} \rightarrow \mathbb{R}_{\geq 0}$, and have vertex set 
\begin{equation}
\label{eq:Vertex-set-comb}
    V = \{ (x,\ell) \in \mathbb{Z} \times  \mathbb{N}_0 \, : \, 0 \leq \ell \leq H_x \}
\end{equation}
and edge set
\begin{equation}
\label{eq:Edge-set-comb}
\begin{split}
    E & = \Big\{\{(x,0),(y,0) \} \, : \,  |x-y| = 1 \Big\}  \cup \Big\{\{(x,\ell),(x,k) \} \, : \, x \in \mathbb{Z}, |\ell-k| = 1 \Big\}.
    \end{split}
\end{equation}
Whether $\mathrm{Comb}(\mathbb{Z},H)$ has the infinite or finite (triple) collision property depends in a subtle way on the tooth profile $(H_x)_{x \in \mathbb{Z}}$. Although, at least intuitively, profiles with `shorter teeth' favor the infinite (triple) collision property, there is no simple general monotonicity property for the latter (see also~\cite[Remark 1.7]{BPS}). In~\cite{CWZ08}, it is established that $\mathrm{Comb}(\mathbb{Z},H)$ possesses the infinite collision property if $H_x \leq |x|^{1/5}$ holds for all $x \in \mathbb{Z}$. In~\cite{BPS}, a general criterion (valid for general recurrent, locally finite, connected graphs) based on effective resistances is established, which in particular implies that $\mathrm{Comb}(\mathbb{Z},H)$ has the infinite collision property for the profile $H_x = |x|^\alpha$, $x \in \mathbb{Z}$, when $\alpha \leq 1$. Moreover, it is proved in~\cite{BPS} that the finite collision property holds for the same comb graph with profile $H_x = |x|^\alpha$, $x \in \mathbb{Z}$, when $\alpha > 1$. In the context of comb graphs $\mathrm{Comb}(\mathbb{Z},H)$, a further improvement appeared in~\cite{Chenchen}, in which it is proved that the infinite collision property holds if
\begin{equation}
\label{eq:C-C-Criterion}
    \sum_{n = 1}^\infty \frac{1}{1 \vee \max_{-n \leq x \leq n} H_x} = \infty.
\end{equation}
The latter criterion is useful to investigate the growth rate of the tooth profile that leads to the phase transition between the infinite and finite collision property. Specifically, it is shown in~\cite{Chenchen} that
\begin{equation}
\label{eq:Infinite-collisions-Chen-Chen}
    \begin{minipage}{0.8\textwidth}for $H^{(\gamma)}_x = |x| \log^\gamma(|x| \vee 1)$, $\mathrm{Comb}(\mathbb{Z},H^{(\gamma)})$ has the infinite collision property if $\gamma \leq 1$, and the finite collision property if $\gamma > 2$. \end{minipage}
\end{equation}
Our first main result addresses the range $\gamma \in (1,2]$ for $H^{(\gamma)}_x = |x| \log^\gamma(|x| \vee 1)$. We show in Theorem~\ref{theo:Reg} that
\begin{equation}
\label{eq:Main-regular-comb-intro}
     \begin{minipage}{0.8\textwidth}for $H^{(\gamma)}_x = |x| \log^\gamma(|x| \vee 1)$, $\mathrm{Comb}(\mathbb{Z},H^{(\gamma)})$ has the infinite collision property if $\gamma \in (1,2]$. \end{minipage}
\end{equation}
This provides an example showing that the criterion~\eqref{eq:C-C-Criterion} obtained in~\cite{Chenchen} $-$ perhaps surprisingly $-$ does \textit{not} capture the precise growth rate of the tooth profile that governs the phase transition from infinitely many to finitely many collisions of two random walks on comb graphs. We further comment on this observation in Remark~\ref{rem:Sec-3}. \medskip 

We now turn to our results concerning comb graphs with a random tooth profile. Let $(H_x)_{x \in \mathbb{Z}}$ be i.i.d.~random variables defined on some probability space $(\Omega,\mathcal{G},\textbf{P})$, with values in $\mathbb{R}_{\geq 0}$. We denote the expectation under $\textbf{P}$ by $\textbf{E}$. It is proved in~\cite[Theorem 1.4]{Chenchen} that 
\begin{equation}
\label{eq:Chen-Chen-triple-coll}
    \begin{minipage}{0.8\textwidth}
 if $\textbf{E}[H_0] < \infty$, $\mathrm{Comb}(\mathbb{Z},H^\omega)$ has the infinite triple collision property for $\textbf{P}$-almost every (a.e.) realization $\omega \in \Omega$  
    \end{minipage}
\end{equation}
(thus, under the same condition, the infinite collision property holds for $\textbf{P}$-a.e.~$\omega \in \Omega$ as well).
Therefore, one can only expect a phase transition for the (triple) collision property in the case where the random variables $H_x$, $x \in \mathbb{Z}$, are sufficiently heavy-tailed. A convenient choice for our purposes is to assume the tail asymptotics as $z$ tends to infinity, 
\begin{equation}
\label{eq:Heavy-tailed-def}
    \textbf{P}(H_0 > z) \sim Cz^{-\gamma}, \qquad \gamma \in (0,1],
\end{equation}
for some $C > 0$ (with the notation $f(z) \sim g(z)$ for two functions $f,g : \mathbb{R} \to (0,\infty)$ meaning that $\lim_{n \rightarrow\infty} \frac{f(z)}{g(z)}=1$). 

In Theorem~\ref{theo:MainRandomComb}, we obtain the critical decay exponent in~\eqref{eq:Heavy-tailed-def} governing the phase transition between the infinite and finite collision property, and show that 
 \begin{equation}
 \label{eq:Main-Result-random-intro}
\begin{minipage}{0.8\textwidth}
    for $\textbf{P}$-a.e.~$\omega \in \Omega$, $\mathrm{Comb}(\mathbb{Z},H^\omega)$ has the infinite collision property if $\gamma > 1/3 $ and the finite collision property if $\gamma < 1/3$.
     \end{minipage}
 \end{equation}
The proof of the infinite collision property for $\gamma \in (1/3, 1]$ in Section~\ref{sec:Double-collisions-random} forms the core of the present article. To establish this result, we perform a precise analysis of the geometry of the random comb and its consequences on heat-kernel estimates. In the range $\gamma \in (0,1/3)$, the finite collision property in~\eqref{eq:Main-Result-random-intro} was proved in a more general context in~\cite{Koops}, requiring only a lower bound (and, for technical reasons, an upper bound) on the tail probability in~\eqref{eq:Heavy-tailed-def}. For completeness, we also provide a short proof of the finite collision property in this regime, assuming only a lower bound on the tail probabilities and removing the technical assumption in~\cite{Koops}, see Theorem~\ref{theo:FinitePrecise} for a precise statement in that regime.

In view of~\eqref{eq:Chen-Chen-triple-coll}, we also examine the phase transition for the infinite triple collision property for i.i.d.~random tooth lengths $(H_x)_{x \in \mathbb{Z}}$ fulfilling~\eqref{eq:Heavy-tailed-def}, and prove in Corollary~\ref{cor:Triple-coll-gamma-smaller-1} and Theorem~\ref{theo:MainRandomCombTrip} that
\begin{equation}
\label{eq:Result-Triple-collisions-intro}
    \begin{minipage}{0.8\textwidth}
    for $\textbf{P}$-a.e.~$\omega \in \Omega$, $\mathrm{Comb}(\mathbb{Z},H^\omega)$ has the finite triple collision property if $\gamma \in (0,1] $.
    \end{minipage}
\end{equation}
The latter result reveals a somewhat different mechanism for the phase transition of the infinite triple collision property of $\mathrm{Comb}(\mathbb{Z},H^\omega)$ compared to $\mathrm{Comb}(\mathbb{Z},J^{(\beta)})$ with the growing profile $J^{(\beta)}(x) = \log^\beta(|x| \vee 1)$, $x \in \mathbb{Z}$, for $\beta  \in (0,\infty)$, studied in~\cite{CroydonDeAmbroggio}. In the latter context, infinitely many triple collisions occur almost surely if $\beta \leq 1$, whereas finitely many triple collisions occur almost surely if $\beta > 1$. 
Another comparison in this direction is with~\cite{croydon2026collision}, where the infinite triple collision property is established for the trace of a simple random walk in dimension four. Our result highlights how the infinite triple collision property is sensitive not only to the heat kernel, but also to the \textit{volume distribution} on the corresponding graph. We refer to Section~\ref{sec:Triple-collisions-random} for more a more detailed comparison between these results. \medskip

We now comment on the strategies underpinning the proofs of our main results~\eqref{eq:Main-regular-comb-intro},~\eqref{eq:Main-Result-random-intro}, and~\eqref{eq:Result-Triple-collisions-intro}. Our first result~\eqref{eq:Main-regular-comb-intro} concerns two independent walks on the regularly growing comb $\mathrm{Comb}(\mathbb{Z},H^{(\gamma)})$ with $H^{(\gamma)}_x = |x| \log^\gamma(|x| \vee 1)$ and $\gamma \in (1,2]$. For that set-up, we utilize a notion of `strikes' introduced in~\cite{Chenchen}. The latter informally correspond to the times in which one of the random walks enters the base site of the tooth in which the other walk is located. The principal aim is to obtain precise controls on the expected number of strikes that occur before one of the walks exists a horizontal part of the comb (see~\eqref{eq:Number-of-strikes} for the definition of this quantity). The lower and upper bound on the number are derived in Lemma~\ref{lem:LBMomentReg} and Lemma~\ref{lem:UBMomentReg}, respectively. A key aspect (which is implicit in these bounds) is to establish that when a strike occurs, the position of the random walk inside the tooth is nearly uniformly distributed (we refer to Remark~\ref{rem:Sec-3} for details). 

Next, we turn to~\eqref{eq:Main-Result-random-intro}, which addresses the random comb $\mathrm{Comb}(\mathbb{Z},H^\omega)$ for a typical realization of $(H^\omega_x)_{x \in \mathbb{Z}}$. The main difficulty lies in obtaining quantitative bounds on the heat kernel or the killed Green's function on the comb graph that are uniform over large rectangular sets. These rectangular sets are defined in~\eqref{eq:Sets-on-Comb} and informally act as a substitute for a `metric/resistance ball' on $\mathrm{Comb}(\mathbb{Z},H^\omega)$, and we obtain controls on their random geometry in Subsection~\ref{sect:EnvEstimates}. Furthermore, in Subsection~\ref{sect:LocalTime}, we show that the local time of the walk in the set of points in $\mathbb{Z}$ which have a suitably long tooth attached is large. The corresponding control is contained in Proposition~\ref{prop:KeyLTEstimate}. This key estimate is then used to bound the probability of exiting horizontal segments quickly in Subsection~\ref{sect:Exit}. These preparations are pivotal to produce precise heat-kernel bounds for a typical realization of the random comb graph in Subsection~\ref{sect:HKRandom}, which enter a second moment method for the number of collisions in Subsection~\ref{sect:SecondMomentRandom}. In Subsection~\ref{subsec:gamma-greater-1/3}, we then conclude the proof of~\eqref{eq:Main-Result-random-intro} for $\gamma \in (1/3,1)$, and consider the remaining cases in the subsequent Subsections.

Finally, we briefly discuss the proof of~\eqref{eq:Result-Triple-collisions-intro} concerning triple collisions on $\mathrm{Comb}(\mathbb{Z},H^\omega)$ for a typical realization of $(H_x^\omega)_{x \in \mathbb{Z}}$. For $\gamma \in (0,1)$, the result readily follows from the upper bounds on the heat kernel obtained in Lemma~\ref{lem:OnDiagonalHKUB} in combination with a result from~\cite{CroydonDeAmbroggio}. On the other hand, the case $\gamma = 1$ is more challenging. First, we obtain an upper bound on the heat kernel which coincides with similar bounds in~\cite{CroydonDeAmbroggio, croydon2026collision}. However, the volume enclosed in the segment $[-N, N] \cap \mathbb{Z}$ up to height $\ell$ (with $\ell,N \geq 1$) is approximately $N \log(\ell)$ and this differs dramatically from the corresponding volume of the regularly growing comb graphs in~\cite{CroydonDeAmbroggio}. 
The counting of collisions is then split into two families of (convenient) space-time sets, and proceeds via a slightly enhanced first moment method. \medskip

We will now describe the organization of the present article. In Section~\ref{sec:Notation}, we collect further notation and useful results concerning comb graphs and random walks, as well as electric network techniques. Section~\ref{sec:Regularly-growing-comb} deals with regularly growing comb graphs, and includes the proof of~\eqref{eq:Main-regular-comb-intro} in Theorem~\ref{theo:Reg}. In Section~\ref{sec:Double-collisions-random} we prove the main result~\eqref{eq:Main-Result-random-intro} concerning the phase transition between the infinite and finite collision property for random combs in Theorem~\ref{theo:MainRandomComb}. In Section~\ref{sec:Triple-collisions-random}, we prove the finite triple collision property~\eqref{eq:Result-Triple-collisions-intro} for random combs with a tooth distribution having a sufficiently heavy tail in Corollary~\ref{cor:Triple-coll-gamma-smaller-1} and Theorem~\ref{theo:MainRandomCombTrip}. \medskip

Finally, we state our convention on constants that is used throughout the article. We denote by $C, c, c', \dots$ positive constants that change from place to place. Numbered constants such as $c_1,c_2, \dots$ refer to the value assigned at the first appearance in the text. 

\section{Notation and useful results}
\label{sec:Notation}

In this Section we introduce further notation and collect useful facts concerning random walks, heat kernels, Green's functions, and effective resistances. \medskip

We write $\mathbb{N} \coloneqq \{1,2,...\}$ for the set of natural numbers and $\mathbb{N}_0 \coloneqq \mathbb{N} \cup \{0\}$. When $a,b \in \mathbb{R}$, we denote the maximum between $a$ and $b$ by $a \vee b$, the minimum between $a$ and $b$ by $a \wedge b$, and set $\llbracket a,b \rrbracket \coloneqq [a,b] \cap \mathbb{Z}$. We use the notation $\lfloor a \rfloor \coloneqq \max\{k \in \mathbb{Z} \, : \, k \leq a \}$ and $\lceil a \rceil \coloneqq \min\{k \in \mathbb{Z} \, : \, k \geq a \} $ for the lower and upper integer part of $a$. The cardinality of any set $A$ will be denoted by $|A|$. For a connected, locally finite graph $G = (V,E)$, and $x \in V$ we write $\mathrm{deg}(x) \coloneqq |\{e \in E \, : \, x \in e \}|$ for the degree of $x$. For two vertices $x,y \in V$, we denote the graph distance, i.e.~the minimal number $n$ such that $\{x,x_1\},\{x_1,x_2\},...,\{x_{n-1},y\}\in E$ holds for some vertices $x_1,...,x_{n-1} \in V$, by $d(x,y)$. For $x \in V, n \in \mathbb{N}$, let $B(x, n) \coloneqq \{y \in V \colon d(x, y) \le n\}$ be the graph-distance ball of radius $n$ centered at $x$. For $A \subseteq V$, we denote by $\partial A \coloneqq \{z \in V \setminus A \, : \, \{z,w\} \in E \text{ for some }w\in A \}$ the outer (vertex) boundary of $A$. If $X$ and $Y$ are real-valued random variables (not necessarily defined on the same probability space), we write $X  \preceq Y $ if the law of $Y$ stochastically dominates the law of $X$. \medskip

Next, we introduce some useful set notation concerning the geometry of comb graphs. Recall the notation $\mathrm{Comb}(\mathbb{Z},H)$ for the comb over $\mathbb{Z}$ with profile $H  = (H_x)_{x \in \mathbb{Z}} \in \mathbb{N}_0^{\mathbb{Z}}$, i.e.~the graph $\mathrm{Comb}(\mathbb{Z},H) = (V,E)$ with $V$ and $E$ as in~\eqref{eq:Vertex-set-comb} and~\eqref{eq:Edge-set-comb}. We sometimes identify by slight abuse of notation the graph $\mathrm{Comb}(\mathbb{Z},H)$ with its vertex set and, when there is no risk of confusion, abbreviate the site $(0, 0)$ in a comb graph as $0$. For $z \in \mathbb{Z}$, $N, M \in \mathbb{R}_{\geq 0}$ and a fixed tooth profile $(H_x)_{x \in \mathbb{Z}}$, we introduce the following subsets of $\mathrm{Comb}(\mathbb{Z},H)$:
\begin{equation}
\label{eq:Sets-on-Comb}
    \begin{split}
    D^z_N & \coloneqq \{v = (x,\ell) \in \mathrm{Comb}(\mathbb{Z},H) \, : \, |z-x| \leq N \}, \\
\mathcal{R}_N^z(M) & \coloneqq\{v = (x,\ell) \in \mathrm{Comb}(\mathbb{Z},H) \, : \, |z-x| \leq N, \ell \leq M  \}, \\
\mathbb{V}_N^z(M) & \coloneqq \{v = (x,\ell) \in \mathrm{Comb}(\mathbb{Z},H) \, : \, |z-x| \leq N, \ell \geq M \}. 
        \end{split}
\end{equation}
We use the abbreviations $D_N = D^0_N$, $\mathcal{R}_N(M) = \mathcal{R}^0_N(M)$, and $\mathbb{V}_N(M) = \mathbb{V}_N^0(M)$. The set $D^z_N$ can be viewed as the \emph{disk} of radius $N$ around $(z,0)$ for the pseudo-metric $\overline{d}((x,\ell),(y,k)) \coloneqq |x-y|$, for $(x,\ell), (y,k) \in \mathrm{Comb}(\mathbb{Z},H)$, i.e.~the entire part of the comb graph over the interval $\llbracket  z-N,z+N\rrbracket$. For later use, we also introduce the sets
\begin{equation}
\label{eq:Comb-sets-part-2}
    \begin{split}
        \partial_s \mathcal{R}^z_N(M) & \coloneqq \{(z-\lfloor N\rfloor -1,0),(z+\lfloor N\rfloor+1,0) \}, \\
        \partial_u \mathcal{R}^z_N(M) & \coloneqq \{(x,\ell) \in \mathrm{Comb}(\mathbb{Z},H) \, : \, |x-z| \leq N, \ell = \lfloor M\rfloor+1 \},
    \end{split}
\end{equation}
which correspond to the `lateral' and `upper' boundary of the rectangle $\mathcal{R}^z_N(M)$ (note that the latter may be empty), and we again use the abbreviations $\partial_s \mathcal{R}_N(M)$ and $\partial_u \mathcal{R}_N(M)$ for the corresponding sets with $z= 0$. We refer to Figure~\ref{fig:Comb} for a visualization of the sets defined above.

\begin{figure}[H]
    \centering
    \includegraphics[width=0.9\linewidth]{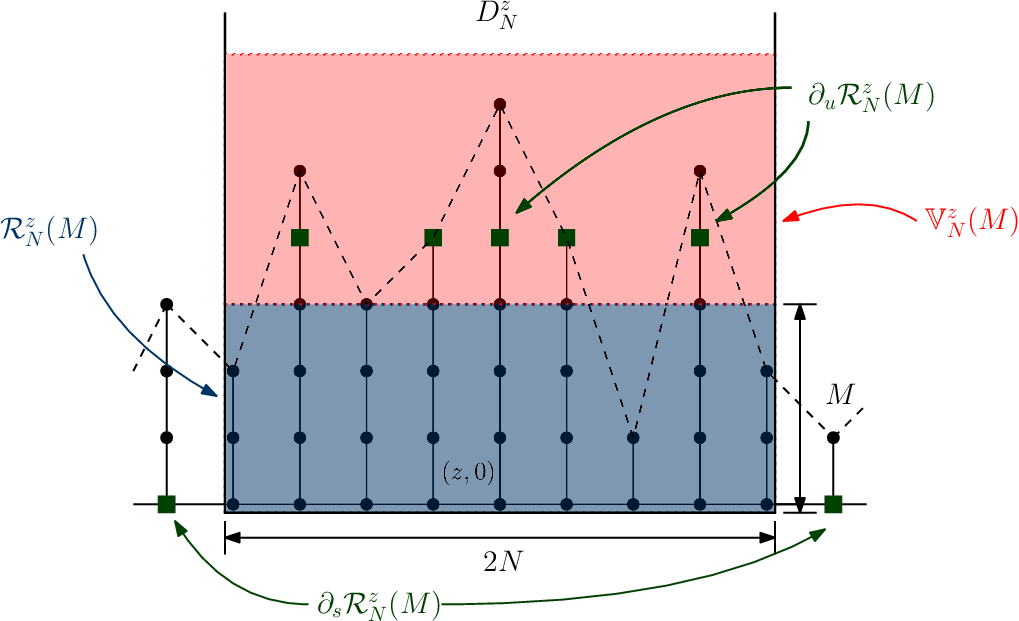}
    \caption{A visualization of the sets defined in~\eqref{eq:Sets-on-Comb} and~\eqref{eq:Comb-sets-part-2}. The dashed line corresponds to the graph of the profile $\{(x,H_x) \, : \, x \in \mathbb{Z}\}$, i.e.~the uppermost points in the comb.}
    \label{fig:Comb}
\end{figure}

We now turn to the random walk on (comb) graphs. 
 We fix a locally finite and connected graph $G = (V,E)$. For $j \in \mathbb{N}$, we denote the canonical process on $(V ^j)^{\mathbb{N}_0}$ by $(X^1_n,...,X^j_n)_{n \in \mathbb{N}_0}$, and abbreviate $(X_n^1)_{n \in \mathbb{N}_0}$ by  $(X_n)_{n \in \mathbb{N}_0}$. We write $\mathcal{F}^{j}_k \coloneqq \sigma(X^i_0,...,X_k^i\, : \, 1  \leq i \leq j)$ for the canonical filtration generated by all $j$ coordinate projections up to time $k$ and $\mathcal{F} \coloneqq \sigma(X_k \colon k \in \mathbb{N}_0)$. The family of canonical shift operators on the space of trajectories $(V^j)^{\mathbb{N}_0}$ is denoted by $(\vartheta^j_n)_{n \geq 0}$, with $\vartheta_n(w^1,...,w^j)(k) = (w^1(n+k),...,w^j(n+k))$ for $n,k \in \mathbb{N}_0$ and $w^1,...,w^j\in V^{\mathbb{N}_0}$ (and $\vartheta_n = \vartheta^1_n)$. For $x \in V$, we denote by $P^G_x$ the law of a (discrete-time) simple random walk starting from $x \in V$ on the canonical space $V^{\mathbb{N}_0}$, namely the discrete-time Markov chain with initial distribution $\delta_x$ (the Dirac measure supported in $x \in V$) and transition probabilities
 \begin{equation}
p_{x,y} = \frac{1}{\mathrm{deg}(x)}\mathds{1}_{\{\{x,y\} \in E \}}, \qquad x,y  \in V.
\end{equation}
We often abbreviate $P^G_x$ as $P_x$ if we work on a fixed graph. The expectation under $P^G_x$ will be written as $E^G_x$ (abbreviated by $E_x$ if we work on a fixed graph). For $x_1,...,x_j \in V$, we introduce a shorthand notation for the product measure $P^G_{x_1,..., x_j} = P^G_{x_1} \otimes ... \otimes P^G_{x_j}$ (again dropping $G$ from the notation if there is no risk of confusion). 

For any $A \subseteq V$, we write $\tau_A \coloneqq \inf\{n \geq 0 \, : \, X_n \in A \}$ for the entrance time of the walk in $A$ (we use the convention that $\inf \varnothing = \infty$). If $A = \{x\}$ for some $x \in V$, we abbreviate $\tau_{\{x\}}$ by $\tau_x$. The heat kernel of the walk $(X_n)_{n \geq 0}$ killed upon exiting $A \subseteq V$ is defined as
\begin{equation}
    p_n^A(x,y) \coloneqq \frac{1}{\mathrm{deg}(y)} P_x(X_n = y, \tau_{A^c} > n), \qquad x,y \in V, n \in \mathbb{N}_0,
\end{equation}
and we abbreviate $p_n(x,y) = p_n^V(x,y)$ for $x,y \in V$, $n \in \mathbb{N}_0$. 
We define the Green kernel of the walk killed upon exiting $A \subseteq V$ as
\begin{equation}
    g^A(x,y) \coloneqq \sum_{n = 0}^\infty p_n^A(x,y), \qquad x,y \in V.
\end{equation}
Both $p_n^A(\cdot,\cdot)$ and $g^A(\cdot,\cdot)$ are symmetric in their arguments, and for $G$ recurrent, $g^A(\cdot,\cdot)$ is finite for $A \neq V$. We now state the Green kernel criterion for the infinite collision property on a large class of graphs $G$ from~\cite[Theorem 3.1]{BPS}:
\begin{equation}
\label{eq:GKC}
    \begin{minipage}{0.8\linewidth}
        Suppose $G = (V,E)$ is a locally finite, connected, infinite, recurrent graph,  $o \in V$ is a fixed vertex and $(U_j)_{j \in \mathbb{N}}$ is an increasing sequence of finite subsets $U_j \subseteq V$ with $\bigcup_{j \in \mathbb{N}} U_j = V$. If for some $C_g > 0$,
        $$
        g^{U_j}(x,x) \leq C_gg^{U_j}(o,o) \qquad \text{for all }x\in U_j, \  j \in \mathbb{N},
        $$
        then $G$ has the infinite collision property.
    \end{minipage}
\end{equation}
Next, we record for a set $\varnothing \neq  A \subseteq V$, $x \in A$, and $n \in \mathbb{N}$ the useful elementary relation
\begin{equation}
\label{eq:Exit-time-estimate}
    P_x(\tau_A > n)^2 = \left(\sum_{y \in A} p_n^A(x,y) \mathrm{deg}(y)\right)^2 \leq \left(\max_{y \in A} \deg(y)\right) |A|p_{2n}^A(x,x),
\end{equation}
where we used the Cauchy-Schwarz inequality in the second step. We then collect another general bound from~\cite[Lemma 4.4]{BPS}, which states that for $\varnothing \neq  A \subseteq V$, one has
\begin{equation}
    \label{eq:Bound-HK-via-Green}
    p_n(x,x) \leq \frac{2g^A(x,x)}{nP_x(\tau_{A^c} \geq n)}, \qquad x\in A, n \in \mathbb{N}.
\end{equation}

We also need the notion of effective resistance, which we now briefly recall. For subsets $\varnothing \neq A, B \subseteq V$, we define define the effective resistance between $A$ and $B$ as 
\begin{equation}
\label{eq:Eff-res-def}
    \begin{split}R_{\mathrm{eff}}(A,B) & = \inf\{\mathcal{E}(h) \, : \, h\vert_A = 1, h\vert_B = 0, \mathcal{E}(h) < \infty \}^{-1}, \qquad \text{where} \\
    \mathcal{E}(h) & = \frac{1}{2}\sum_{x,y \in V \, : \, \{x,y\}\in E} (h(x)-h(y))^2 , \ h : V \rightarrow \mathbb{R}
    \end{split}
\end{equation}
(using again the convention that $\inf \varnothing = \infty$). If $A = \{a\}$, we abbreviate $R_{\mathrm{eff}}(\{a\},B)$ as $R_{\mathrm{eff}}(a,B)$, and use a similar notation if both $A$ and $B$ are singleton sets. It follows directly from the definition~\eqref{eq:Eff-res-def} that
\begin{equation}
\label{eq:Monotonicity}
R_{\mathrm{eff}}(A,B) \geq R_{\mathrm{eff}}(A',B), \qquad \text{for } \varnothing \neq A \subseteq A'  \subseteq V, \varnothing \neq B \subseteq V,
\end{equation}
and that $R_{\mathrm{eff}}(\cdot,\cdot)$ is symmetric in its arguments. Effective resistances can be calculated explicitly using network reduction techniques (see, for instance,~\cite[Chapter 2]{Barbook} or~\cite[Chapter 2]{LPBook}) which will be used at various stages throughout the article. We note that for any locally finite, connected graph $G$, we have
\begin{equation}
\label{eq:Effective-resistance-graph-distance-bound}
    R_{\mathrm{eff}}(x,y) \leq d(x,y), \qquad x,y \in V,
\end{equation}
and equality holds if the graph is a tree. 
The following bound for escape probabilities on  a recurrent, connected, locally finite graph $G$ will be useful: 
\begin{equation}
\label{eq:Escape-Prob-bound}
    P_x(\tau_A > \tau_{B^c} ) \leq \frac{R_{\mathrm{eff}}(x,A)}{R_{\mathrm{eff}}(x,B^c)}, \qquad \text{for }x \in B\setminus A, A \subseteq B \subseteq V,
\end{equation}
see, for instance,~\cite[Lemma 2.62]{Barbook} for a proof.
We also recall the commute time identity (see, for instance,~\cite[Proposition~10.7]{LPBook}), which states that for a finite connected graph $G = (V,E)$ and any $a,b\in V$, one has 
\begin{equation}
\label{eq:Commute-time}
    E_a[\tau_b] + E_b[\tau_a] = 2|E|R_{\mathrm{eff}}(a,b).
\end{equation}

\section{Phase transition for the number of collisions on regularly growing combs}
\label{sec:Regularly-growing-comb}

In this Section, we consider a regularly growing comb graph $\mathrm{Comb}(\mathbb{Z},H^{(\gamma)})$ with tooth profile
\begin{equation}\label{eqn:GrowthReg}
    H^{(\gamma)}_{x} \coloneqq |x| \log^\gamma(|x| \vee 1), \quad \text{for }x \in \mathbb{Z},
\end{equation}
where $\gamma \in (0, \infty)$. 
Our main result concerning these comb graphs is as follows.
\begin{thm}\label{theo:Reg}
The following statements hold:
	\begin{enumerate}
		\item If $\gamma \le 2$, then $\mathrm{Comb}(\mathbb{Z},H^{(\gamma)})$ has the infinite collision property.
        \item If $\gamma > 2$ then $\mathrm{Comb}(\mathbb{Z},H^{(\gamma)})$ has the finite collision property.
	\end{enumerate}
\end{thm}
 This result completes the picture for the phase transition from the infinite to finite collision property of $\mathrm{Comb}(\mathbb{Z},H^{(\gamma)})$ in~\cite{Chenchen}, see~\eqref{eq:Infinite-collisions-Chen-Chen}. There, the second statement and the first statement of Theorem~\ref{theo:Reg} restricted to the sub-regime $\gamma \le 1$ are obtained, using the criterion~\eqref{eq:Edge-set-comb} for the latter. Our Theorem~\ref{theo:Reg} addresses the regime $\gamma \in (1, 2]$ in which the criterion~\eqref{eq:Edge-set-comb} cannot be applied. Our strategy is essentially a variation of the approach in~\cite{Chenchen}, however more precise controls on the heat kernel are required, see Remark~\ref{rem:Sec-3} for a more thorough comparison to~\cite{Chenchen}. The remainder of this Section is organized as follows: In Subsection~\ref{subsec:HK-bounds}, we develop heat kernel bounds for the random walk on $\mathrm{Comb}(\mathbb{Z},H^{(\gamma)})$. In Subsection~\ref{subsec:Collisions-Strikes}, we introduce a notion of `strikes' following~\cite{Chenchen}, which is then used to organize the summation over the number of collisions. In Subsections~\ref{subsec:First-Moment-Lower} and~\ref{subsec:First-Moment-Upper}, lower and upper bounds on the number of strikes are proved. Finally, in the short Subsection~\ref{subsec:Denouement-Comment} we finish the proof of Theorem~\ref{theo:Reg} and provide some remarks on our approach. 

\subsection{Lower and upper bounds on the heat kernel}
\label{subsec:HK-bounds}
For $z\in \mathbb{Z}$ and $N \in \mathbb{R}_{\geq 0}$, we recall from~\eqref{eq:Sets-on-Comb} the notation $D^z_N$ for the `disk' of horizontal radius $N$ around the point $(z,0)$. We begin with some standard exit time estimates for random walks on comb graphs. Throughout, we use the abbreviation $\tau_{z,N} = \tau_{(D^{z}_N)^c}$ for the exit time of the disk $D_z^{N}$ with $z \in \mathbb{Z}, N \in \mathbb{R}_{\geq 0}$.

We now state two lemmas from~\cite{BPS}, the proof of which can be adjusted by changing the tooth profile to~\eqref{eqn:GrowthReg}. The following result corresponds to~\cite[Lemma~4.5]{BPS}.
\begin{lem}\label{lem:ExitLDPReg}
    There exist constants $C_1, c_1> 0 $ such that uniformly over all $x= (x_1, 0) \in \mathrm{Comb}(\mathbb{Z},H^{(\gamma)})$, $k \ge 1$ and $n \ge 1$
    \begin{equation}
        P_x \left(\tau_{x_1,k} \le n \right) \le C_1 \exp \left( - c_1 \left(\frac{k^3 \log^\gamma(k \vee 1)}{n}\right)^{1/3} \right).
    \end{equation}
\end{lem}
The following result is analogous to \cite[Lemma~4.9]{BPS}, again adjusting for the different tooth profile.
\begin{lem}\label{prop:UBHKReg}
    There exist constants $C_2, c_2 > 0$ such that uniformly over all $x = (x_1, x_2) \in \mathrm{Comb}(\mathbb{Z},H^{(\gamma)})$ and all $y = (y_1, y_2) \in \mathrm{Comb}(\mathbb{Z},H^{(\gamma)})$ we have for all $n \ge 1$ setting $n = s^3 \log^\gamma(s\vee1)$,
    \begin{equation}
    \label{eq:HK-UB-general}
        p_n(x, y) \le \begin{cases}
            C_2 s^{-2}\log^{-\gamma}(s\vee1) e^{-c_2 \frac{(x_2 \vee y_2)^2}{n}}, & s \ge |x_1 - y_1|,\\
            C_2 |x_1 - y_1|^{-2}\log^{-\gamma}(|x_1 - y_1|\vee1) e^{-c_2 \frac{(x_2 \vee y_2)^2}{n}}, & s < |x_1 - y_1|.
        \end{cases} 
    \end{equation}
\end{lem}

We will also need a lower bound on the heat kernel of the random walk killed upon exiting $D_{bN}$ for sufficiently large $b \geq 13$, and $N \in \mathbb{N}$.  Its proof may be compared to that~\cite[Lemma 3.2]{CroydonDeAmbroggio}, adapted to the present case.

\begin{prop}\label{prop:LBHKReg}
    There exist constants $c_3, c_4, c_5 > 0$, $b \geq 13$ such that for all $x = (x_1, x_2) \in D_N$ and all $y = (y_1, y_2) \in D_N$ with $n- d(x, y) \in 2\mathbb{N}_0$, we have     \begin{equation}
    \label{eq:Killed-main-bound-regular-case}
        p_n^{D_{bN}}(x, y) \ge \frac{c_3}{N^2 \log^{\gamma}(N)} \qquad \text{ for } n \in \llbracket c_4 N^3 \log^\gamma(N), c_5 N^3 \log^\gamma(N) \rrbracket.
    \end{equation}
\end{prop}
\begin{proof}
      By applying~\eqref{eq:Exit-time-estimate} to $A = D^{x_1}_{(b-1)N}$, we obtain the on-diagonal lower bound
      \begin{equation}
\label{eq:On-diagonal-bound}
        p_{2n}^{D^{x_1}_{bN}}(x, x) \ge \frac{P_x \left( \tau_{x_1,{(b-1)N}} > c_4 N^3 \log^\gamma(N) \right)^2}{3|D^{x_1}_{(b-1)N}|},
\end{equation} 
for $x = (x_1,x_2) \in D_N$, $N \in \mathbb{N}$, and $n \in \llbracket c_4 N^3 \log^\gamma(N), c_5 N^3 \log^\gamma(N) \rrbracket$.
    By Lemma~\ref{lem:ExitLDPReg}, the enumerator satisfies $P_x \left( \tau_{x_1,{(b-1)N}} > c_4 N^3 \log^\gamma(N) \right) \ge \frac{1}{2}$ by choosing $b$ large enough (depending on $C_1,c_1$, and $c_5$), showing that
    \begin{equation}
        \label{eq:Killed-main-bound-diagonal}
        p_{2n}^{D^{x_1}_{bN}}(x, x) \ge \frac{c}{N^2 \log^{\gamma}(N)}, \qquad \text{ for } 2n \in \llbracket c_4 N^3 \log^\gamma(N), c_5 N^3 \log^\gamma(N) \rrbracket,
    \end{equation}
    which corresponds to \eqref{eq:Killed-main-bound-regular-case} in the on-diagonal case $x = y \in D_N$.  We now extend the lower bound to the off-diagonal case $x \neq y \in D_N$, distinguishing into three (sub-)cases. 
    
    \textbf{Case 1:} $x = (x_1, 0) \in D_N$ and $y = (y_1, y_2) \in D_N$ with $y_2 \le N$. 
    By \cite[(3.11)]{CroydonDeAmbroggio} we have that for any set $B \subseteq \mathrm{Comb}(\mathbb{Z},H^{(\gamma)})$, any pair of vertices $v, w \in \mathrm{Comb}(\mathbb{Z},H^{(\gamma)})$, and $n \in \mathbb{N}$ such that $n- d(v, w) \in 2 \mathbb{N}_0$,
    \begin{equation}\label{eqn:TrickUmbiDave}
        p_{n}^{B}(v, w) \ge P_v\left( \tau_w \le n \wedge \tau_{B^c}\right)p_{2 \lfloor n/2\rfloor }^{B}(w,w).
    \end{equation}
     Applying this bound with $x = (x_1,0)$ and $y = (y_1,y_2)$ in place of $v$ and $w$ and $D_{bN}^{x_1}$ replacing $B$, we find that whenever $d(x, y) \in 2 \mathbb{N}_0$, we have
    \begin{equation}
    \label{eq:From-on-to-off-diagonal-Case-I}
        p_n^{D^{x_1}_{bN}}(x, y) \ge p_n^{D^{x_1}_{bN}}(y, y) P_x \left( \tau_{y} \le n \wedge \tau_{x_1,bN}\right).
    \end{equation}
    (the case $d(x,y) \in 2\mathbb{N}_0+1$ can be treated analogously, possibly by adjusting constants, and is omitted). We can bound $p_n^{D^{x_1}_{bN}}(y, y)$ from below using the on-diagonal estimate~\eqref{eq:Killed-main-bound-diagonal} and focus on the second member of the product on the right-hand side.  
    We use the decomposition
    \begin{equation}
    \label{eq:Decomposition-Proof-Case-I-b}
        P_x \left( \tau_{y} > n \wedge \tau_{x_1,bN}\right) \le P_x \left( \tau_{y} \wedge \tau_{x_1,bN} > n\right) + P_x \left( \tau_{y} > \tau_{x_1,bN}\right).
    \end{equation}
    Using the commute time identity~\eqref{eq:Commute-time},~\eqref{eq:Effective-resistance-graph-distance-bound}, and the Markov inequality, we obtain
    \begin{equation}
    \label{eq:Aux-1-Proof}
        P_x \left( \tau_{y} \wedge \tau_{x_1,bN} > n\right) \le \frac{3N N^{2}\log^\gamma(N)}{n} \le \frac{3}{c_4}.
    \end{equation}
        Moreover, since $D_{x_1,bN}$ is a tree, we can use~\eqref{eq:Escape-Prob-bound} and obtain
    \begin{equation}
        \label{eq:Aux-2-Proof}
        P_x \left( \tau_{y} > \tau_{x_1,bN}\right) \le \frac{3N}{(b-1)N}.
    \end{equation}
    By choosing $c_4 \geq 12$ and $b \geq 13$, 
    the expressions on the right-hand sides of~\eqref{eq:Aux-1-Proof} and~\eqref{eq:Aux-2-Proof} are each from bounded above by $\frac{1}{4}$. The claim then follows by inserting the bounds~\eqref{eq:Aux-1-Proof} and~\eqref{eq:Aux-2-Proof} into~\eqref{eq:Decomposition-Proof-Case-I-b}, and using the on-diagonal bound~\eqref{eq:Killed-main-bound-diagonal} in combination with~\eqref{eq:From-on-to-off-diagonal-Case-I}. \medskip

    \textbf{Case 2:} $x = (x_1, x_2) \in D_N$ and $y = (y_1,y_2) \in D_N$ with $x_2>0$ and $y_1 \leq N$, and let $n \in \mathbb{N}$ with $n - d(x,y) \in 2\mathbb{N}_0$. We bound $p_n^{D^{x_1}_{bN}}(x,y)$ from below by only retaining trajectories of length $n$ from $x$ to $y$ that enter $(x_1, 0)$ at some time in $\llbracket 1, n \rrbracket$. 
    Therefore, by decomposing according to the first entrance time, we see that \begin{equation}\label{eqn:HittingDecomp}
        p_n^{D^{x_1}_{bN}}(x, y) \ge \sum_{k = 1}^n P_x \left( \tau_{(x_1, 0)} = k\right) p_{n-k}^{D^{x_1}_{bN}}((x_1, 0), y).
    \end{equation}
     Since $H_{x_1}^{(\gamma)} \le N\log^\gamma(N)$, we obtain by the commute time identity~\eqref{eq:Commute-time} that 
     \begin{equation}
         E_x\left[\tau_{(x_1, 0)}\right] \le 2N^2\log^{2\gamma}(N).
     \end{equation}
    By the Markov inequality, we have 
    \begin{equation}\label{eqn:NegligibleHit}
        P_x \left( \tau_{(x_1, 0)} > N^2\log^{3\gamma}(N)\right) \le 2\log^{-\gamma}(N).
    \end{equation}
    Then (possibly after adjusting $c_4, c_5$) we obtain that
    \begin{equation}\label{eqn:Third}
        p_{n-k}^{D^{x_1}_{bN}}((x_1, 0), y) \ge \frac{c_3}{N^2 \log^{\gamma}(N)},
    \end{equation}
    holds uniformly over all $n \in \llbracket c_4 N^3 \log^\gamma(N), c_5 N^3 \log^\gamma(N) \rrbracket$ and $k \le N^2\log^{3\gamma}(N)$ if $d((x_1, 0), y) - (n-k) \in 2 \mathbb{N}_0$.

    \textbf{Case 3:} $x = (x_1, x_2) \in D_N$ and $y = (y_1,y_2) \in D_N$, with $y_1 > N$. We can decompose the trajectories according to their first hitting time of $(x_1, 0)$ as in \eqref{eqn:HittingDecomp} (but retaining the term for $k=0$). Using again \eqref{eqn:NegligibleHit} we can reduce the problem to proving a uniform lower bound on $p_n^{D^{x_1}_{bN}}((x_1, 0), y)$ and use the symmetry of $p_n^{D_{bN}^{x_1}}$, so that 
    \begin{equation}\label{eq:LowerCase3}
        p_n^{D^{x_1}_{bN}}((x_1, 0), y) = p_n^{D^{x_1}_{bN}}(y, (x_1, 0)).
    \end{equation}
    Finally the lower bound on the right-hand side is for a point $y \in D_N$ and a point $(x_1, 0)$, which falls under Case 2 and yields \eqref{eqn:Third}. We conclude by inserting the estimate \eqref{eqn:Third} into \eqref{eq:LowerCase3}.
\end{proof}

During the rest of the Section, we fix $b \geq 13$ so that the result of Proposition~\ref{prop:LBHKReg} holds.

\subsection{Collisions and strikes}
\label{subsec:Collisions-Strikes}

In this Subsection, we define a notion of a `strike', which informally corresponds to a time when one random walks visits the base of a tooth containing the other. To that end, we introduce further notation, following~\cite{Chenchen}. We let $X_n^1 = (U_n^1, V^1_n)$ and $X^2_n = (U_n^2, V_n^2)$, where we recall that $(X_n^1,X_n^2)_{n \geq 0}$ is the canonical process on $(V \times V)^{\mathbb{N}_0}$ (with $(U_n^j)_{n \geq 0}$ and $(V_n^j)_{n \geq 0}$, $j \in \{1,2\}$ referring to the horizontal and vertical coordinates of the walk $X^j$). We let $\tau^j_B$ denote the entrance times by $X^j$ into a set $B \subseteq \mathrm{Comb}(\mathbb{Z},H^{(\gamma)})$, and set for $N \in \mathbb{N}_0, b \geq 13$, and $x = (x_1,x_2) \in \mathrm{Comb}(\mathbb{Z},H^{(\gamma)})$,
\begin{equation}
    \theta_N \coloneqq \min\{\tau^1_{x_1,bN}, \tau^2_{x_1,bN}\}.
\end{equation}
Moreover, set $\sigma_0 = 0$ and for each $m \in \mathbb{N}_0$, define inductively
\begin{equation}
    \sigma_{m+1} \coloneqq \inf\left\{n > \sigma_m \colon U_n^1 = U_n^2, U_n^1 \neq U_{n-1}^1 \text{ or }U_n^2 \neq U_{n-1}^2 , |U_n^1|\le N\right\},
\end{equation}
where the index $m \ge 0$ counts the number of times one of the two walks is located in one of the teeth within $D_N$ and the other visits the unique vertex connecting the tooth to the backbone $\mathbb{Z}$. We call such an event a \textit{strike}. Note that this definition corresponds to a variant of a similar definition in~\cite[Section 2]{Chenchen} (our definition of $\sigma_{m+1}$ includes the additional restriction that $|U_n^1| \leq N$). Furthermore, we define the event
\begin{equation}
    \Psi_m \coloneqq \left\{X_n^1 = X_n^2 \text{ for some } \sigma_m \le n \le \inf\{k > \sigma_m \colon V^1_k=0 \text{ or } V^2_k = 0\}\right\}, m \in \mathbb{N}_0.
\end{equation}
Informally, $\Psi_m$ denotes the occurrence of at least one collision between $\sigma_m$ and the next time one of the walks hits height $0$ after $\sigma_m$. By~\cite[Lemma 2.2]{Chenchen}, for some $C,c > 0$,
\begin{equation}
\label{eq:Aux-Gambler-ruin-Sec-3}
    \frac{c}{h} \leq P_{(k,0),(k,h)}(\Psi_0) \leq \frac{C}{h}, \qquad k\in \mathbb{N}_0,h \in 2\mathbb{N} \text{ with }(k,h) \in \mathrm{Comb}(\mathbb{Z},H^{(\gamma)}).
\end{equation}
Finally, we are in a position to define
\begin{equation}
\label{eq:Number-of-strikes}
    \Theta_N \coloneqq \sum_{m = 1}^{\theta_{bN}} \mathds{1}_{\Psi_m}.
\end{equation}
The key result of this Section is the following lower bound on the probability of finding at least one collision after a strike, until either walk leaves $D_{bN}$. Its proof relies on two auxiliary results, Lemma~\ref{lem:LBMomentReg} and Lemma~\ref{lem:UBMomentReg}, which will be proved below.
\begin{prop}\label{prop:MainReg}
There exist constants $c_*, N_0 >0$ such that for all $N \geq N_0$ 
    \begin{equation}
    \label{eq:Prop-3.5-claim}
         \min_{y \in \partial D_N \cup \{0\}} \min_{\substack{x \in D_N \\
         d(x, y) \in 2 \mathbb{N}_0}}P_{x, y} \left( \Theta_N \ge 1 \right) \ge \begin{cases}
            c_{*} \frac{1}{\log(N)} & \gamma \in (1, 2),\\
            c_{*} \frac{1}{\log(N) \log\log(N)} & \gamma = 2.
        \end{cases}
    \end{equation}
\end{prop}
\begin{proof}
    For $x\in D_N$, $y \in \partial D_N \cup \{0\}$ with $d(x, y) \in 2 \mathbb{N}_0$, and $N \geq 1$, we have
    \begin{equation}
        P_{x, y} \left( \Theta_N \ge 1 \right) E_{x, y} \left[ \Theta_N \mid  \Theta_N \ge 1\right] = E_{x, y} \left[ \Theta_N\right],
    \end{equation}
    which implies, using the Markov property,
    \begin{equation}
        P_{x, y} \left( \Theta_N \ge 1 \right) \ge \frac{E_{x, y} \left[ \Theta_N\right]}{E_{x, y} \left[ \Theta_N \mid  \Theta_N \ge 1\right]} \ge \frac{E_{x, y} \left[ \Theta_N\right]}{1 + \max_{x \in D_N} E_{x, x} \left[ \Theta_N \right]}.
    \end{equation}
    The claim~\eqref{eq:Prop-3.5-claim} readily follows upon using Lemma~\ref{lem:LBMomentReg} (for $N \geq N_0$) and Lemma~\ref{lem:UBMomentReg}.
\end{proof}

The remainder of this Section is concerned with the proofs of the auxiliary results Lemma~\ref{lem:LBMomentReg} and Lemma~\ref{lem:UBMomentReg}. At the end of the Section we conclude the proof of Theorem~\ref{theo:Reg} from Proposition~\ref{prop:MainReg}.
\subsection{Lower bound on the first moment of $\Theta_N$}
\label{subsec:First-Moment-Lower} In the present Subsection, we prove a uniform lower bound on the expectation of $\Theta_N$, defined in~\eqref{eq:Number-of-strikes}. Our main result is the following.
\begin{lem}\label{lem:LBMomentReg}
    There exist constants $c_{**}, N_0>0$ such that for all $N \geq N_0$, 
    \begin{equation}
        \min_{y \in\partial D_N \cup \{0\}} {\min_{\substack{x \in D_N \\
         d(x, y)  \in 2 \mathbb{N}_0 }}} E_{x, y} \left[ \Theta_N\right] \ge c_{**} \log^{1 - \gamma}(N).
    \end{equation}
\end{lem}

\begin{proof}
    Let $y \in \partial D_N \cup \{0\}$, $x \in D_N$ with $d(x,y) \in 2 \mathbb{N}_0$, and $N \geq 10$. We write
    \begin{equation}
    \begin{split}
        E_{x, y} \left[ \Theta_N\right] \ge E_{x, y} \left[ \sum_{n = 0}^{\infty} \sum_{m = 1}^\infty \sum_{|k|\le N} \sum_{h = 1}^{\lfloor H_k \rfloor} \mathds{1}_{\Psi_m} \mathds{1}_{\{ \sigma_m = n\}} \mathds{1}_{\{\sigma_m < \theta_{bN}, X^1_{n} = (k, 0), X^2_{n} = (k, h)\}}\right].
    \end{split}
    \end{equation}
    We introduce the notation $f(N) = N^3 \log^\gamma(N)$, $N \in \mathbb{N}$, and $A_{m, n}(k, h) = \{ \sigma_m = n\}\cap \{X^1_{n} = (k, 0), X^2_{n} = (k, h)\}$, $m,n,k,h\in \mathbb{N}_0$. Note that $A_{m,n}(k,h)$ and $\{\theta_{bN} > n \}$ are $\mathcal{F}^2_n$-measurable, while on the event $\{ \sigma_m = n\}$, we have the identity $\mathds{1}_{\Psi_m} = \mathds{1}_{\Psi_0} \circ \vartheta^2_n$. Hence, by conditioning on $\mathcal{F}_n^2$ and applying the simple Markov property at time $n$ for the joint chain $(X^1_\ell,X^2_\ell)_{\ell \geq 0}$ in the second line, we obtain
    \begin{equation}
    \label{eq:First-Lower-Bound-Theta_N}
    \begin{split}
        E_{x, y} \left[ \Theta_N\right] &\ge \sum_{n = 0}^{\infty} \sum_{|k|\le N} \sum_{h = 1}^{\lfloor H_k \rfloor} \sum_{m = 1}^\infty E_{x, y} \left[  \mathds{1}_{ \Psi_m} \mathds{1}_{\{ \sigma_m = n\}} \mathds{1}_{\{\sigma_m < \theta_{bN}, X^1_{n} = (k, 0), X^2_{n} = (k, h)\}}\right]\\
        & = \sum_{n = 0}^{\infty} \sum_{|k|\le N} \sum_{h = 1}^{\lfloor H_k \rfloor} \sum_{m = 1}^\infty E_{x, y} \left[  P_{(k, 0), (k, h)}(\Psi_0)\mathds{1}_{\{ \sigma_m = n\}} \mathds{1}_{\{\sigma_m < \theta_{bN}, X^1_{n} = (k, 0), X^2_{n} = (k, h)\}}\right]\\
        &\stackrel{\eqref{eq:Aux-Gambler-ruin-Sec-3}}{\ge} \sum_{n = 0}^{\infty} \sum_{|k|\le N} \sum_{h = 1}^{\lfloor H_k \rfloor} \sum_{m = 1}^\infty \frac{c}{h}P_{x, y} \left( \sigma_m = n, \sigma_m < \theta_{bN}, X^1_{n} = (k, 0), X^2_{n} = (k, h)\right).
    \end{split}
    \end{equation}
    For the event under the probability in the last line of~\eqref{eq:First-Lower-Bound-Theta_N}, we observe that if the event $\{X^1_{n-2} = (k, 0), X^2_{n-2} = (k, h)\} \cap \{\theta_{bN}>n-2\}$ occurs, then with strictly positive probability (more precisely, with probability larger than $1/9$) $\{X^1_{n} = (k, 0), X^2_{n} = (k, h)\}$ is a strike. By restricting to the strikes with the property that at time $n-2$ the walks were in the same position we can bound $E_{x,y}[\Theta_N]$ from below by
\begin{equation}
\begin{split}
    E_{x, y} \left[ \Theta_N\right] &\ge \sum_{n = 2}^{\infty} \sum_{|k|\le N} \sum_{h = 1}^{\lfloor H_k \rfloor} \frac{c}{h}P_{x, y} (n-2 < \theta_{bN}, \\
    & \qquad \qquad \qquad \qquad  X^1_{n-2} = (k, 0), X^2_{n-2} = (k, h), X^1_{n} = (k, 0), X^2_{n} = (k, h) )\\
    & \ge \sum_{n = 2}^{\infty} \sum_{|k|\le N} \sum_{h = 1}^{\lfloor H_k \rfloor}  \frac{c}{9h} p_{n-2}^{D_{bN}}(x, (k, 0)) p_{n-2}^{D_{bN}}(y, (k, h))\\
    & \ge \sum_{\substack{n = \lfloor c_4 f(N) \rfloor+1\\ n \in 2\mathbb{N}}}^{\lfloor c_5f(N) \rfloor-2}\sum_{|k|\le N} \sum_{h = 1}^{\lfloor H_k \rfloor}  \frac{c}{9h} p_{n-2}^{D_{bN}}(x, (k, 0)) p_{n-2}^{D_{bN}}(y, (k, h))\\
    & \stackrel{\eqref{eq:Killed-main-bound-regular-case}}{\ge} \frac{c}{27} \sum_{n = \lfloor c_4 f(N) \rfloor+1}^{\lfloor c_5f(N) \rfloor-2} \sum_{N/2 \le |k|\le N}  \sum_{h = 1}^{\lfloor H_k \rfloor} \frac{1}{h} \frac{c_3^2}{N^4 \log^{2\gamma}(N)}\\
    & \ge \frac{c}{27} \sum_{n = \lfloor c_4 f(N) \rfloor+1}^{\lfloor c_5f(N) \rfloor-2} \sum_{N/2 \le |k|\le N}  \log(N/2) \frac{c_3^2}{N^4 \log^{2\gamma}(N)}\\
    & \ge \frac{c}{27}\sum_{n = \lfloor c_4 f(N) \rfloor+1}^{\lfloor c_5f(N) \rfloor-2} \frac{N}{2}  \log(N/2) \frac{c_3^2}{N^4 \log^{2\gamma}(N)} \ge c_{**} \log^{1-\gamma}(N),
\end{split}
\end{equation}
which concludes the proof.
 \end{proof}

\newpage

\subsection{Uniform upper bound on the first moment of $\Theta_N$}
\label{subsec:First-Moment-Upper}

\begin{lem}\label{lem:UBMomentReg}
    There exists a constant $C_{**}>0$ such that
    \begin{equation}
        \max_{x \in D_N} E_{x, x} \left[ \Theta_N \right] \le \begin{cases}
            C_{**} \log^{2 - \gamma}(N) & \gamma \in (1, 2),\\
            C_{**} \log\log(N) & \gamma = 2.
        \end{cases}
    \end{equation}
\end{lem}
\begin{proof}
    Let $x = (k, h) \in D_N$. Repeating the argument in~\eqref{eq:First-Lower-Bound-Theta_N}, using the upper bound in~\eqref{eq:Aux-Gambler-ruin-Sec-3}, and summing over $m$, we find
    \begin{equation}
        E_{x, x} \left[ \Theta_N\right] \le 2\sum_{n = 0}^{\infty} \sum_{|k'|\le N} \sum_{h' = 1}^{\lfloor H_{k'} \rfloor} p_{n}^{D_{bN}}(x, (k', 0)) p_{n}^{D_{bN}}(x, (k', h')) \frac{C}{h'},
    \end{equation}
    using for the last step the fact that for a strike to occur at time $n$, the event $\{X^1_n = (k', h'), X_n^2 = (k', 0)\} \cup \{X_n^1 = (k', 0), X_n^2 = (k', h')\}$ must necessarily occur. We split the summand on the right hand side into
    \begin{equation}
    \begin{split}
    \label{eqn:InnerUBReg} 
        E_{x, x} \left[ \Theta_N\right] &\le \underbrace{2\sum_{n = 0}^{\infty} \sum_{|k'|\le 2|k|} \sum_{h' = 1}^{\lfloor H_k \rfloor} p_{n}^{D_{bN}}(x, (k', 0)) p_{n}^{D_{bN}}(x, (k', h')) \frac{C}{h'}}_{\eqqcolon \, \eqref{eqn:InnerUBReg}\text{-(a)}} \\
        & +  \underbrace{2\sum_{n = 0}^{\infty} \sum_{|k'|=2|k|}^N \sum_{h' = 1}^{\lfloor H_{k'} \rfloor} p_{n}^{D_{bN}}(x, (k', 0)) p_{n}^{D_{bN}}(x, (k', h')) \frac{C}{h'}}_{\, \eqqcolon \eqref{eqn:InnerUBReg}\text{-(b)}}
            \end{split}
    \end{equation}
    We will begin by bounding the term \eqref{eqn:InnerUBReg}\text{-(b)}. We will apply Proposition~\ref{prop:UBHKReg} repeatedly. We define the function $F(s) = s^3 \log^\gamma(s \vee 1)$ for $s \in \mathbb{R}_{\geq 0}$, and further decompose \eqref{eqn:InnerUBReg}\text{-(b)} as
    \begin{equation}
    \label{Int1}
    \begin{split}
      \eqref{eqn:InnerUBReg}\text{-(b)}=\, & \underbrace{2\sum_{n = 0}^{\lfloor F(|k - k'| )\rfloor} \sum_{|k'|=2|k|}^N \sum_{h' = 1}^{\lfloor H_{k'} \rfloor} p_{n}^{D_{bN}}(x, (k', 0)) p_{n}^{D_{bN}}(x, (k', h')) \frac{c}{h'}}_{\eqqcolon \, \eqref{Int1}\text{-(a)}}\\ 
        &+ \underbrace{ 2\sum_{n = \lfloor F(|k - k'|) \rfloor+1}^{ \infty} \sum_{|k'|=2|k|}^N \sum_{h' = 1}^{\lfloor H_{k'} \rfloor} p_{n}^{D_{bN}}(x, (k', 0)) p_{n}^{D_{bN}}(x, (k', h')) \frac{c}{h'}}_{\eqqcolon \,\eqref{Int1}\text{-(b)}}.
            \end{split}
    \end{equation}
    The first term can be dealt with by integrating $h'$ and applying routine upper bounds on the heat kernel,~\eqref{eq:HK-UB-general}, giving
    \begin{equation}
        \eqref{Int1}\text{-(a)}\le 2C \sum_{|k'|=2|k|}^N \frac{\log(|k'|) |k - k'|^3 \log^\gamma(|k-k'|)}{|k - k'|^4 \log^{2\gamma}(|k-k'|)}.
    \end{equation}
    Using that $|k'| \le N$ and applying \eqref{Int:xlog(x)} gives the bound
    \begin{equation}
    \label{eq:First-Int-Part-Sum-1}
        \eqref{Int1}\text{-(a)} \le \begin{cases}
            C_{**} \log^{2 - \gamma}(N), & \gamma \in (1, 2),\\
            C_{**} \log\log(N), & \gamma = 2,
        \end{cases}
    \end{equation}
    as required. Moving to \eqref{Int1}\text{-(b)} we obtain (using in the application of Proposition~\ref{prop:UBHKReg} that $F(s) = n$ implies $F^{-1}(n) n^{-1} = (s^2\log^\gamma(s))^{-1}$)
    \begin{equation}
    \begin{split}
        \eqref{Int1}\text{-(b)} &\le 2\sum_{|k'|=2|k|}^N \sum_{n = \lfloor F(|k - k'|) \rfloor + 1}^{ \infty} \frac{\log(|k'|) (F^{-1}(n))^2}{n^2} \\
        & \le 4C  \sum_{|k'|=2|k|}^N \log(|k'|) \int_{|k - k'|}^\infty \frac{s^2}{s^6 \log^{2\gamma}(s)}\mathrm{d} s^3 \log^\gamma(s),
            \end{split}
    \end{equation}
    applying \eqref{Int:squuared} and \eqref{Int:xlog(x)} we obtain again
    \begin{equation}
    \label{eq:First-Int-Part-Sum-2}
        \eqref{Int1}\text{-(b)} \le \begin{cases}
            C_{**} \log^{2 - \gamma}(N), & \gamma \in (1, 2),\\
            C_{**} \log\log(N), & \gamma = 2.
        \end{cases}
    \end{equation}
    We now move to the term \eqref{eqn:InnerUBReg}\text{-(a)}. We split it similarly to the term \eqref{eqn:InnerUBReg}\text{-(b)} and write
    \begin{equation}
    \label{Int3}
    \begin{split}
       \eqref{eqn:InnerUBReg}\text{-(a)} = &\, \underbrace{2\sum_{n = 0}^{\lfloor F(|k'-k|) \rfloor} \sum_{|k'|\le 2|k|} \sum_{h' = 1}^{\lfloor H_{k'} \rfloor} p_{n}^{D_{bN}}(x, (k', 0)) p_{n}^{D_{bN}}(x, (k', h')) \frac{c}{h'}}_{ \eqqcolon\, \eqref{Int3}\text{-(a)}}\\ 
        &+ \underbrace{2\sum_{n = \lfloor F(|k'-k|) \rfloor +1}^{  \infty} \sum_{|k'|\le 2|k|} \sum_{h' = 1}^{\lfloor H_{k'} \rfloor} p_{n}^{D_{bN}}(x, (k', 0)) p_{n}^{D_{bN}}(x, (k', h')) \frac{c}{h'}}_{\eqqcolon\, \eqref{Int3}\text{-(b)}}.
        \end{split}
    \end{equation}
    We now proceed to bound \eqref{Int3}\text{-(a)}, obtaining
    \begin{equation}
    \begin{split}
         \eqref{Int3}\text{-(a)} &\le 2 C \sum_{|k'|\le 2|k|} \sum_{1 \leq h' \leq \sqrt{F(|k'-k|)}} \frac{1}{|k - k'|\log^{\gamma}(|k - k'| + 1)} \frac{c}{h'} \\
        &+ 2 C \sum_{|k'|\le 2|k|} \sum_{h' > \sqrt{F(|k'-k|)}} \frac{e^{-2\frac{(h')^2}{F(|k - k'|)}}}{|k - k'|\log^{\gamma}(|k - k'| + 1)} \frac{c}{\sqrt{F(|k'-k|)}}\\
        & \le 2 C \sum_{|k'|\le 2|k|} \frac{\log(|k - k'| + 1)}{|k - k'|\log^{\gamma}(|k - k'| + 1)} + 2 C \sum_{|k'|\le 2|k|} \frac{1}{|k - k'|\log^{\gamma}(|k - k'| + 1)}.
    \end{split}
    \end{equation}
    Here, the first term is the dominant term. Indeed, noticing that $|k - k'| \le 4|k|$ over the range of parameters considered, applying the change of variable $s = |k - k'|$, using that $|k| \le N$ and applying \eqref{Int:xlog(x)} we obtain the bound
    \begin{equation}
    \label{eq:Second-bound-part-1}
        \eqref{Int3}\text{-(a)} \le \begin{cases}
            C_{**} \log^{2 - \gamma}(N) & \gamma \in (1, 2),\\
            C_{**} \log\log(N) & \gamma = 2.
        \end{cases}
    \end{equation}
    We finally bound the remaining term \eqref{Int3}\text{-(b)}. Proceeding as in the last steps we obtain, using $n = F(s)$, that
    \begin{equation}
        \begin{split}
        \eqref{Int3}\text{-(b)} &\le 2 \sum_{|k'|\le 2|k|} \sum_{h' = 1}^{\lfloor H_{k'} \rfloor} \sum_{n = \lfloor F(|k'-k|) \rfloor+1}^{ \infty} \frac{1}{s^2\log^\gamma(s+1) s^2\log^\gamma(s+1)} \frac{c}{h'}\\
        &\le 2C \sum_{|k'|\le 2|k|} \sum_{h' = 1}^{\lfloor H_{2k} \rfloor}\frac{c}{h'} \int_{s = |k - k'|} \frac{1}{s^2\log^\gamma(s+1) s^2\log^\gamma(s+1)} \mathrm{d}s^3\log^\gamma(s).
        \end{split}
    \end{equation}
    Applying \eqref{Int:squuared} and using again the fact that $|k - k'|\le 4k$, we obtain
    \begin{equation}
    \begin{split}
        \eqref{Int3}\text{-(b)} &\le 2C \sum_{|k'|\le 2|k|} \sum_{h' = 1}^{\lfloor H_{2k} \rfloor}\frac{c}{h'} \frac{1}{|k-k'|\log^{\gamma}(|k - k'|+1)}\\
        & \le 2C \sum_{|u|\le 4|k|} \frac{1}{|u|\log^{\gamma - 1}(|u|+1)}.
        \end{split}
    \end{equation}
    Finally, we apply \eqref{Int:xlog(x)} and $k \le N$ to obtain the bound
    \begin{equation}
        \label{eq:Second-bound-part-2}
        \eqref{Int3}\text{-(b)} \le \begin{cases}
            C_{**} \log^{2 - \gamma}(N) & \gamma \in (1, 2),\\
            C_{**} \log\log(N) & \gamma = 2.
        \end{cases}
    \end{equation}
   Upon inserting~\eqref{eq:First-Int-Part-Sum-1} and~\eqref{eq:First-Int-Part-Sum-2} into~\eqref{Int1},~\eqref{eq:Second-bound-part-1} and~\eqref{eq:Second-bound-part-2} into~\eqref{Int3}, and combining~\eqref{Int1} and~\eqref{Int3}, the claim follows in view of~\eqref{eqn:InnerUBReg}.
\end{proof}

\subsection{Conclusion}
\label{subsec:Denouement-Comment}

\begin{proof}[Proof of Theorem~\ref{theo:Reg}]
    Finishing the proof is now routine, let us describe the argument briefly. We define the events 
\begin{equation}
    \mathrm{Col}_m \coloneqq \left\{ \text{there exists } n \in [\theta_{b^m}, \theta_{b^{m+1}}) \cap \mathbb{Z} \colon X_n^1 = X_n^2\right\}, \qquad m \in \mathbb{N}_0.
\end{equation}
By the strong Markov property and the uniformity over the starting points in Proposition~\ref{prop:MainReg},  
we have
\begin{equation}
    P_{0, 0} \left(\mathrm{Col}_m \mid \mathcal{F}^2_{\theta_{b^m}}\right) \ge \min_{y \in \partial D_{b^m} \cup \{0\}} \min_{\substack{x \in D_{b^m}  \\
         d(x, y) \in 2 \mathbb{N}_0 }}P_{x, y} \left(\Theta_{b^{m}} \ge 1\right) \ge \rho ( > 0), \ P_{0,0}\text{-a.s.,}
\end{equation}
where $\mathcal{F}^2_{\theta_{b^m}}$ is the sigma field generated by the two random walks up to the random time $\theta_{b^m}$. By a standard conditional Borel-Cantelli argument (see \cite[Corollary 7.20]{Kall}) this yields
\begin{equation}
    P_{0, 0} \left(\mathrm{Col}_m \text{ infinitely often}\right) = 1. \qedhere
\end{equation}
\end{proof}

\begin{obs}
\label{rem:Sec-3}
    \begin{enumerate}
        \item We compare our proof of the infinite collision property of $\mathrm{Comb}(\mathbb{Z},H^{(\gamma)})$ for $\gamma \in (1,2]$ to the proof for $\gamma > 2$ in~\cite{Chenchen}. We have set up a similar argument of defining strikes and estimating the number of collisions using them. However, our key improvement is to show that, when a strike takes place, the height of the walk which is not on the horizontal axis is essentially uniformly distributed over the tooth. This allows us to avoid the worst-case scenario used in \cite{Chenchen}, which uses a bound in which the second walk is at the top of the tooth when the other strikes. 
    \item We briefly comment on the proof of the finite collision property of $\mathrm{Comb}(\mathbb{Z},H^{(\gamma)})$ for $\gamma \leq 1$ in~\cite{Chenchen}. 
    The final display  in~\cite[p.~1355]{Chenchen} contains an `additional' logarithmic term without which the proof would extend to $\gamma \in (1,2]$.
    In essence, our proof shows that this logarithmic term is \textnormal{not} an artifact of the proof of the upper bound in the latter reference, but actually correctly quantifies the number of collisions.
    \end{enumerate}
\end{obs}

\section{Phase transition for the number of collisions on random combs}
\label{sec:Double-collisions-random}

In this Section, we consider the comb graph $\mathrm{Comb}(\mathbb{Z}, H^\omega)$ with an i.i.d.~random height assignment $(H^\omega_z)_{z \in \mathbb{Z}}$ defined on some probability space $(\Omega,\mathcal{G},\textbf{P})$, where the height distribution satisfies~\eqref{eq:Heavy-tailed-def} for some $\gamma \in (0, 1]$. Our main result is the following.

\begin{thm}\label{theo:MainRandomComb}
	The following statements hold for $\mathbf{P}$-a.e.~$\omega \in \Omega$.
	\begin{enumerate}
        \item If $\gamma > 1/3$ then  $\mathrm{Comb}(\mathbb{Z}, H^\omega)$ has the infinite collision property.
		\item If $\gamma < 1/3$ then $\mathrm{Comb}(\mathbb{Z}, H^\omega)$ has the finite collision property.
	\end{enumerate}
\end{thm}

The two statements in Theorem~\ref{theo:MainRandomComb} are established separately. 
Concerning Theorem~\ref{theo:MainRandomComb}-\emph{(2)}, we will state and prove a stronger result in Theorem~\ref{theo:FinitePrecise} implying the latter, which only requires a lower bound on the tail distribution of $H_0$ (see~\eqref{eq:Tail-LB-Assumption}; a similar statement with an additional technical condition appeared in~\cite{Koops}). 
The argument in this regime follows from a slightly modified first moment method.\medskip

The proof of Theorem~\ref{theo:MainRandomComb}-\emph{(1)} concerning the infinite collision property in the regime $\gamma \in (1/3,1]$ is more challenging, and constitutes the principal part of the present Section. We will set up a relatively standard second moment method counting collisions between certain Markovian times. However, both the lower bound for the first moment and the upper bound for the second moment will require \textit{uniform estimates} over large disks. This substantially increases the technical aspects of the proof and is the main reason behind the absence of the boundary case $\gamma = 1/3$ in our analysis, which we conjecture to fall into the regime of Theorem~\ref{theo:MainRandomComb}-\emph{(1)}. \medskip 

 We outline the `collision mechanism' on $\mathrm{Comb}(\mathbb{Z},H^\omega)$ effective in the regime $\gamma \in (1/3,1]$ that gives rise to the infinite collision property in Theorem~\ref{theo:MainRandomComb}-\emph{(1)}. To establish the infinite collision property we will restrict our attention to the collisions on increasingly large segments of the horizontal axis $\mathbb{Z}$ (viewed as a subset of $\mathrm{Comb}(\mathbb{Z},H^\omega)$). This is perhaps surprising as this ostensibly neglects a considerable amount of vertices that may potentially contribute to the total number of collisions. In fact, we show that the bulk of the local time of a simple random walk on $\mathrm{Comb}(\mathbb{Z},H^\omega)$ is spent `far away from the backbone', since such vertices constitute most of the volume of the graph (visible by the walk). However, collisions that occur far away from the backbone are highly correlated and hence not suitable for a second moment method. 
Informally, our estimates suggest that the number of collisions occurring at the \textit{typical} height are not the main contribution to the infinite collision property. 
The collision mechanism is therefore not that one random walk gets trapped in a long tooth and then `waits' for the other (as in \cite{DGP1, DGP2, devulder2025infinitely}), but that when both walks move between long excursions their paths typically intersect on the horizontal axis. This is associated to the fact that it is extremely unlikely for two walks to choose the same long tooth for a very long excursion. At the same time, long excursions are the driving force behind the slow movement of the random walk on $\mathrm{Comb}(\mathbb{Z},H^\omega)$. These observations are related to the scaling limit of the walk projected onto the horizontal axis derived in \cite{BCCR}. Indeed, the Fractional Kinetics limit observed there encodes the fact that traps are transparent (in the sense of~\cite[Section 3]{BCCR}) and their location is unpredictable (a non-transparent trapping mechanism would lead to a limiting process that `remembers' the environment, such as the FIN diffusion in~\cite{FIN}).

We briefly introduce further convention on notation. Throughout this Section, we use $P^\omega_x$ and $E^\omega_x$ as abbreviations for $P^{\mathrm{Comb}(\mathbb{Z},H^{\omega})}_x$ and $E^{\mathrm{Comb}(\mathbb{Z},H^{\omega})}_x$ for $\omega \in \Omega$, $x \in \mathrm{Comb}(\mathbb{Z},H^\omega)$. Similarly, we use $p^\omega_n(\cdot,\cdot)$ as a shorthand notation for $p^{\mathrm{Comb}(\mathbb{Z},H^{\omega})}_n(\cdot,\cdot)$ for $\omega \in \Omega$, $n \in \mathbb{N}_0$. Finally, for a subset $B \subseteq \mathrm{Comb}(\mathbb{Z},H^\omega)$, with $\omega \in \Omega$, we use $p^{\omega,B}_n(\cdot,\cdot)$ as a notation for the killed heat kernel $p^B_n(\cdot,\cdot)$, explicitly highlighting the dependence on $\omega$.

\subsection{Environment estimates}\label{sect:EnvEstimates}

In this Section we show several useful estimates that hold for a typical realization of the tooth profile $(H^\omega_z)_{z \in \mathbb{Z}}$. Since we will deal with the case $\gamma = 1$ separately in Subsection~\ref{Sect:Gamma1}, we assume (until then) that $\gamma \in (0, 1)$.

For $M,N \in \mathbb{R}_{\geq 0}$, $z \in \mathbb{Z}$, we recall the notation in~\eqref{eq:Sets-on-Comb} for the set of vertices $\mathbb{V}_{N}^z(M) = \left\{v = (x, \ell) \in \mathrm{Comb}(\mathbb{Z},H^\omega) \colon |z - x| \le N, \ell  \ge M \right\}$ and $D^z_N = \mathbb{V}^z_N(0)$ (see Figure~\ref{fig:Comb} for a visualization). We say that for $\ell \in \mathbb{N}_0$ a disk $D^z_{N}$ is \textit{$\ell$-good} if
\begin{equation}\label{eqn:GoodBox}
    \frac{|\{x \in \llbracket z-N, z+N \rrbracket \colon H_x \in [2^{\ell}, 2^{\ell+1})\} |}{\mathbf{E}\left[|\{x \in \llbracket z-N, z+N \rrbracket \colon H_x \in [2^{\ell}, 2^{\ell+1})\}|\right]} \in [1/4, 4],
\end{equation}
and 
\begin{equation}\label{eqn:GoodBox2}
    \frac{|\{x \in \llbracket z-N, z+N \rrbracket \colon H_x \ge 2^{\ell+1}\}|}{\mathbf{E}\left[|\{x \in \llbracket z-N, z+N \rrbracket \colon H_x \ge 2^{\ell+1}\}|\right]} \in [1/4, 4].
\end{equation}
We define the convenient abbreviation
\begin{equation}
\label{eq:convenient-abbrev}
    u(k)\coloneqq \lfloor \log_2(k^{2/(1 + \gamma)})\rfloor+1, \qquad k \in \mathbb{N},
\end{equation}
which enters the proofs below.
\begin{lem}\label{lem:GoodBoxes}
		For any $\delta_1>0$ and $\mathbf{P}$-a.e.\ $\omega \in \Omega$, there exists $N_0(\omega) \in \mathbb{N}$ such that, for all $N \ge N_0(\omega)$, all disks $D_k^z \subseteq D_N$ (i.e.~such that $\llbracket z - k, z+ k \rrbracket \subseteq \llbracket -N,N \rrbracket$), with $k\geq N^{\delta_1}$ 
        are $\ell$-good, for all $\ell \in \llbracket 0, u(k) \rrbracket$.
\end{lem}

\begin{proof}
	We observe that for a fixed box of width $k\ge 0$, we have
	\begin{equation}
		|\{x \in \llbracket -k, k\rrbracket \colon H_x \in [2^{\ell}, 2^{\ell+1})\}| \sim \mathrm{Binom}\left( 2k+1,C2^{-\gamma \ell }(1 -2^{-\gamma}) \right)
	\end{equation}
    (here and in the following, we denote by $\mathrm{Binom}(M,q)$ a generic binomially distributed random variable with parameters $\lfloor M \rfloor$, where $M \in \mathbb{R}_{\geq 0}$, and $q \in [0,1]$, on some auxiliary probability space $(S,\mathcal{S},\textbf{Q})$).
    We record the basic exponential bound
    \begin{equation}
        \mathbf{Q}\left(\mathrm{Binom}(m, q)>4 m q\right) \le e^{-2mq}, \qquad m \in \mathbb{N}, q \in [0,1].
    \end{equation}
    Hence, with $K(\gamma) = 2C(1 -2^{-\gamma})$, for any $k \in \mathbb{N}_0$ and $\ell \in \llbracket 0, u(k) \rrbracket$, we have
    \begin{equation}
    \label{eq:Environment-estimate-first-step}
    \begin{split}
    	\mathbf{P}\Big(|\{x \in \llbracket -k, k\rrbracket \colon H_x \in [2^{\ell}, 2^{\ell+1})\}| & > 4 \mathbf{E}\left[|\{x \in \llbracket -k, k\rrbracket \colon H_x \in [2^{\ell}, 2^{\ell+1})\}|\right] \Big)  \\
           & \qquad \qquad  \le e^{- (2k+1)K(\gamma) 2^{-\ell \gamma}},
            \end{split}
    \end{equation}
    with a similar bound for the lower deviation (with $1/4$ replacing $4$). Let $N \in \mathbb{N}$. For $k \geq N^{\delta_1}$, $\ell \leq u(k)$, we see upon using the definition~\eqref{eq:convenient-abbrev} that $ k 2^{-\ell \gamma} \geq 4^{-\gamma} k^{1 - \frac{2\gamma}{1+\gamma}} 
    \geq 4^{-\gamma} N^{\delta_1 \frac{1-\gamma}{1+\gamma} }$. 
    Therefore, we have 
    \begin{equation}
    \label{eq:Box-not-good-bound}
    \begin{split}
       \sup_{k \geq N^{\delta_1}} \sup_{\ell \in \llbracket 0, u(k) \rrbracket}	\mathbf{P}\Big(|\{x \in \llbracket -k, k \rrbracket  \colon H_x \in [2^{\ell}, 2^{\ell+1})]\}|  &  > 4 \mathbf{E}\left[|\{x \in \llbracket -k, k \rrbracket \colon H_x \in [2^{\ell}, 2^{\ell+1})\}|\right] \Big) \\ &   \qquad \qquad \stackrel{\eqref{eq:Environment-estimate-first-step}}{\le} e^{-\frac{2}{4^\gamma}K(\gamma) N^{\delta_1\frac{1-\gamma}{1+\gamma}}},
        \end{split}
    \end{equation}
    again with a similar bound for the lower deviation. By taking a union bound and using that the laws of $|\{x \in \llbracket z-k,z+k \rrbracket \colon H_x \in [2^\ell,2^{\ell+1}) \}|$ coincide for all $z \in \mathbb{Z}$, we find that
    \begin{equation}
    \begin{split}
        \textbf{P}\Big( \bigcup_{k = \lfloor N^{\delta_1} \rfloor + 1}^N \bigcup_{ \llbracket z -k , z+k \rrbracket \subseteq \llbracket -N,N\rrbracket }\bigcup_{\ell  = 0}^{u(k)} & \{ D^z_k \text{ does not fulfill~\eqref{eqn:GoodBox} for $\ell$} \} \Big) \\
        & \stackrel{\eqref{eq:Box-not-good-bound}}{\leq} 4N^{2-\delta_1}(u(N)+1) e^{-C(\gamma) N^{\delta_1\frac{1-\gamma}{1+\gamma}}}, 
            \end{split}
    \end{equation}
    for some $C(\gamma) > 0$.    
    By the first Borel-Cantelli lemma, we obtain that on a set $\Omega_0'$ of $\textbf{P}$-probability one, there exists $N_0'(\omega) \in \mathbb{N}$ such that all disks $D_k^z \subseteq D_N$ fulfill~\eqref{eqn:GoodBox} for all $k \geq N^{\delta_1}$, $\ell \in \llbracket 0,u(k)\rrbracket$, for all $N \geq N'_{0}(\omega)$. By repeating the argument for the random variables $|\{x \in \llbracket -k, k \rrbracket \colon H_x \ge 2^{\ell+1}\}|$, $\lfloor N^{\delta_1} \rfloor +1 \leq k \leq N $, $\ell \in \llbracket 0, u(k) \rrbracket$ we obtain the that on another set $\Omega_0''$ of $\textbf{P}$-probability one, there exists $N_0''(\omega) \in \mathbb{N}$ such that all disks $D_k^z \subseteq D_N$ fulfill~\eqref{eqn:GoodBox2} for all $k \geq N^{\delta_1}$, $\ell \in \llbracket 0,u(k) \rrbracket$, for all $N \geq N''_0(\omega)$. The statement of the Lemma follows by setting $N_0(\omega) = N_0'(\omega) \vee N_0''(\omega
    )$ for $ \omega \in \Omega_0' \cap \Omega_0''$.
\end{proof}
\begin{cor}\label{cor:GoodVolumes}
    For any $\delta_1>0$, there exists $c_{\mathrm{vol}} > 1$ such that for $\mathbf{P}$-a.e.\ $\omega \in \Omega$, with $N_0(\omega)$ as in the statement of Lemma~\ref{lem:GoodBoxes}, for all $N \ge N_0(\omega)$, $k \geq N^{\delta_1} $, $z \in \llbracket -N,N \rrbracket$ with $\llbracket z - k,z+k \rrbracket \subseteq \llbracket -N,N \rrbracket$ 
    and for all $\ell \in \llbracket 0, u(k) \rrbracket$, we have
    \begin{equation}
	 c_{\mathrm{vol}}^{-1}\leq \frac{\big|\mathbb{V}_{k}^z(2^{\ell}) \setminus \mathbb{V}_{k}^z(2^{\ell + 1})\big|}{\mathbf{E}\left[|\mathbb{V}_{k}^z(2^{\ell}) \setminus \mathbb{V}_{k}^z(2^{\ell + 1})|\right]} \leq c_{\mathrm{vol}}.
\end{equation}
\end{cor}
\begin{proof}
    We note that for all $z \in \mathbb{Z}$ with $\llbracket z - k,z+k \rrbracket \subseteq \llbracket -N,N \rrbracket$
    \begin{equation}
    \begin{split}
        2^{\ell} |\{x \in \llbracket z-k, z+k \rrbracket \colon H_x \ge 2^{\ell + 1}\}|  & \le \Big|\mathbb{V}_{k}^z(2^{\ell}) \setminus \mathbb{V}_{k}^z(2^{\ell+1})\Big|\\
        & \le 2^{\ell} |\{x \in \llbracket z-k, z+k \rrbracket \colon H_x \ge 2^{\ell}\}|.
            \end{split}
    \end{equation}
    The same inequality holds for the corresponding expectations. Hence, the result follows directly from Lemma~\ref{lem:GoodBoxes} and the fact that by~\eqref{eq:Heavy-tailed-def} the expected numbers of teeth of length at least $2^\ell$ and at least $2^{\ell+1}$ differ by a multiplicative factor $C2^{-\gamma }$, for any $\ell \geq 0$.
\end{proof}

We recall for $M,N \in \mathbb{R}_{\geq 0}$ and $z \in \mathbb{Z}$ the notion of a \textit{rectangle} $\mathcal{R}^z_{N}(M) = \{v = (x, y) \in \mathrm{Comb}(\mathbb{Z},H^\omega) \colon |z-x| \le N, y \le M\}$ defined in \eqref{eq:Sets-on-Comb}  (see Figure~\ref{fig:Comb} for a visualization). We will be particularly interested in the rectangles $\mathcal{R}^z_{k}(k^{2/(1+\gamma)})$, $k \geq 1$. Roughly speaking, these play the role of resistance balls in graphs with regular growth and are fundamentally linked to heat-kernel estimates. The following corollary shows that the effective resistance between a point $x \in \mathbb{Z}$ on the horizontal axis and the outside of a rectangle around $x$ is controlled by its expectation, uniformly over a large class of rectangles contained in $D_N$. 
\begin{cor}\label{cor:GoodRes}
    For any $\delta_1>0$ there exists $\widetilde{c}_{\mathrm{eff}}>1$ such that for $\mathbf{P}$-a.e.\ $\omega \in \Omega$, with $N_0(\omega)$ as in the statement of Lemma~\ref{lem:GoodBoxes}, for all $N \ge N_0(\omega)$, $k \geq N^{\delta_1} $, and $x \in \llbracket -N,N \rrbracket$ with $\llbracket x - k,x+k \rrbracket \subseteq \llbracket -N,N \rrbracket$, we have 
    \begin{equation}
    \label{eq:Resistance-concentration-eq}
	 \widetilde{c}_{\mathrm{eff}}^{-1} \leq \frac{R_{\mathrm{eff}}((x,0), \mathcal{R}^x_{k}(k^{2/(\gamma+1)})^c)}{\mathbf{E}\left[R_{\mathrm{eff}}((x,0), \mathcal{R}^x_{k}(k^{2/(\gamma+1)})^c)\right]} \leq \widetilde{c}_{\mathrm{eff}}.
\end{equation}
Moreover, there exists $c_{\mathrm{eff}}>1$ such that for all $x \in \mathbb{Z}$ and $k \in \mathbb{N}$,
\begin{equation}
\label{eq:Effective-res-exp}
    c_{\mathrm{eff}}^{-1} k \le \mathbf{E}\left[R_{\mathrm{eff}}((x,0), \mathcal{R}_{k}(k^{2/(\gamma+1)})^c)\right] \le c_{\mathrm{eff}} k.
\end{equation}
\end{cor}

\begin{proof}
We recall for $M,N \in \mathbb{R}_{\geq 0}$ the notation $\partial_s \mathcal{R}_N(M)$ and $\partial_u \mathcal{R}_N(M)$ from~\eqref{eq:Comb-sets-part-2}. We first establish the control~\eqref{eq:Effective-res-exp} on the expected effective resistance. Throughout, we fix $x \in \mathbb{Z}$.

The upper bound is immediate by~\eqref{eq:Effective-resistance-graph-distance-bound}, since  $d(x,\mathcal{R}^x_{k}(M)) \leq k$ for any $k \in \mathbb{N}$, $M \in \mathbb{R}_{\geq 0}$.  
We now turn to the lower bound, and let $k \in \mathbb{N}$. By the parallel network reduction formula \cite[Section~9.4]{LPBook}, we have  
\begin{equation}
\label{eq:After-network-reduction-A}
    R_{\mathrm{eff}}(x, \mathcal{R}^x_{k}(k^{2/(\gamma+1)})^c) \ge \frac{1}{2} \left(R_{\mathrm{eff}}(x, \partial_u \mathcal{R}^x_{k}(k^{2/(\gamma+1)})) \wedge R_{\mathrm{eff}}(x, \partial_s \mathcal{R}^x_{k}(k^{2/(\gamma+1)}))\right).
\end{equation}
 We obtain directly from the definition~\eqref{eq:Eff-res-def} of the effective resistance that
\begin{equation}
    R_{\mathrm{eff}}((x,0), \partial_s \mathcal{R}^x_{k}(k^{2/(\gamma+1)})) \ge \frac{k}{2}.
\end{equation}
One obtains a lower bound for $R_{\mathrm{eff}}((x,0), \partial_u \mathcal{R}^x_{k}(k^{2/(\gamma+1)}))$ by fusing all the teeth attached to vertices in $\llbracket x-k,x+k\rrbracket$ that exceed height $k^{2/(\gamma+1)}$, see Figure~\ref{fig:network-reduction} for a visual representation of this procedure. 

\begin{figure}[htbp]
\centering
\includegraphics[width=1\linewidth]{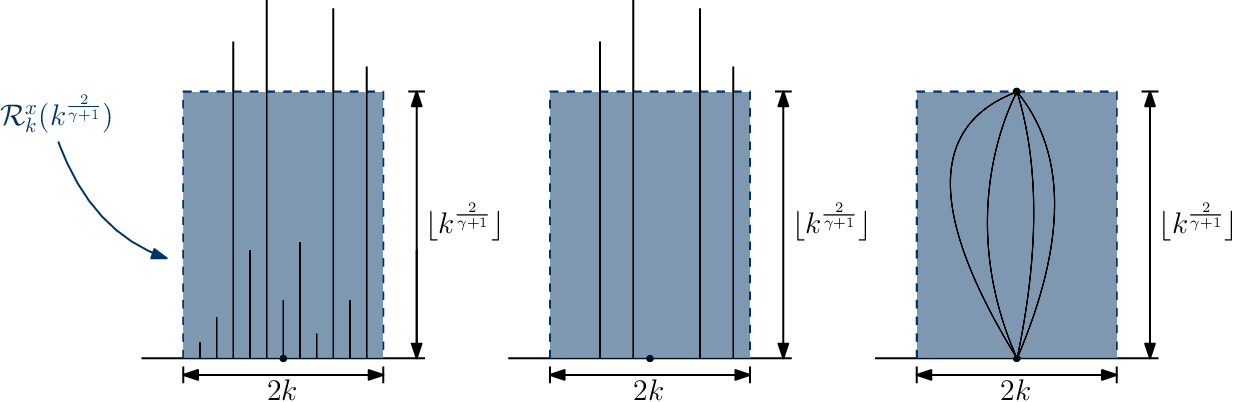}

\caption{A portion of the random comb, with the rectangle $\mathcal{R}_k^x(k^{2/(\gamma+1)})$ highlighted. The two steps of network reduction are indicated in the second and third panel: The short teeth do not affect the effective resistance and are erased in the first step. The remaining long teeth are put in parallel to bound the effective resistance from below.}
\label{fig:network-reduction}

\end{figure}

We note that the resistance along one tooth is the distance and hence bounded below by $k^{2/(\gamma+1)}$. Thus, defining $\mathcal{N}_{k,x} \coloneqq |\{y \in \llbracket x-k,x+k \rrbracket \colon H_y \geq k^{2/(\gamma+1)} \}|$, 
we obtain (again by the parallel law) that
\begin{equation}\label{eqn:LowerRes}
    R_{\mathrm{eff}}((x,0), \partial_u \mathcal{R}^x_{k}(k^{2/(\gamma+1)})) \ge \frac{1}{\mathcal{N}_{k,x}} k^{2/(\gamma+1)}.
\end{equation}
Hence, we see that
\begin{equation}
\label{eq:Expected-eff-res-LB}
   \mathbf{E}\left[ R_{\mathrm{eff}}((x,0), \mathcal{R}^x_{k}(k^{2/(\gamma+1)})^c) \right] \stackrel{\eqref{eq:After-network-reduction-A}}{\ge} \frac{1}{2}  \left( \mathbf{E} \left[\left( \frac{1}{\mathcal{N}_{k,x}} k^{2/(\gamma+1)} \right) \wedge \frac{k}{2} \right] \right).
\end{equation}
Note that whenever $\mathcal{N}_{k,x} \ge 2k^{\frac{1-\gamma}{1+\gamma}}$ we have that $\frac{1}{\mathcal{N}_{k,x}} k^{2/(\gamma+1)} \le k/2$, and therefore
\begin{equation}
\label{eq:Expected-eff-res-LB2}
\begin{split}
       \mathbf{E}\left[ R_{\mathrm{eff}}((x,0), \mathcal{R}^x_{k}(k^{2/(\gamma+1)})^c) \right] \ge \frac{1}{2}  \left( \mathbf{E} \left[\left( \frac{1}{\mathcal{N}_{k,x} + 2k^{\frac{1-\gamma}{1+\gamma}}} k^{2/(\gamma+1)} \right)\right] \right).
\end{split}
\end{equation}
We use that $\textbf{E}[\mathcal{N}_{k,x}] \leq C(2k+1) k^{-2\gamma/(\gamma+1)}$. Thus, applying Jensen's inequality,
\begin{equation}
\label{eq:Expected-eff-res-LB-aux}
    \mathbf{E}\left[R_{\mathrm{eff}}(x, \partial_u \mathcal{R}^x_{k}(k^{2/(\gamma+1)}))\right] \ge \frac{k^{\frac{2}{\gamma + 1}}}{\mathbf{E}[\mathcal{N}_{k,x}] + 2k^{\frac{1-\gamma}{1+\gamma}}} \ge \frac{1}{4C + 2} k^{\frac{2\gamma}{\gamma+1} - 1} k^{\frac{2}{\gamma + 1}} \geq \frac{1}{8C} k.
\end{equation}
Inserting~\eqref{eq:Expected-eff-res-LB-aux} into~\eqref{eq:Expected-eff-res-LB2} yields the lower bound in~\eqref{eq:Effective-res-exp}.

We now turn to the proof of~\eqref{eq:Resistance-concentration-eq}. The upper bound is again immediate by~\eqref{eq:Effective-resistance-graph-distance-bound}, in combination with the lower bound on the expected effective resistance in~\eqref{eq:Effective-res-exp}. On the other hand, we note that $|\{ y \in \llbracket x-k,x+k \rrbracket \colon H_y \geq 2^{u(k)} \}| \leq  \mathcal{N}_{k,x} \leq |\{ y \in \llbracket x-k,x+k \rrbracket \colon H_y \geq 2^{u(k)-1} \}| $, so by Lemma~\ref{lem:GoodBoxes}, we obtain that for $\textbf{P}$-a.e.~$\omega \in \Omega$, there exists $N_0(\omega) \in \mathbb{N}$ such that for all $N \geq N_0(\omega)$, $k \ge N^{\delta_1}$ and all disks $D_k^x \subseteq D_N$
\begin{equation}
    \mathcal{N}_{k,x} \le 2^{2+\gamma} \mathbf{E}[\mathcal{N}_{k,x}].
\end{equation}
Inserting this estimate into \eqref{eqn:LowerRes}, we obtain the claim.
\end{proof}

In the next lemma we obtain a consequence from Corollary~\ref{cor:GoodRes} for the random walk on $\mathrm{Comb}(\mathbb{Z},H^\omega)$ concerning entrance probabilities for sites in the rectangle $ \mathcal{R}_{N}(N^{2/(\gamma+1)})$. This estimate will be fundamental for proving a lower bound on the heat kernel of the random walk killed upon leaving a rectangle, see Lemma~\ref{lem:LBHKOffDiag}.

\begin{lem} \label{lem:ProbandRes}
    There exists $\varrho : (0,\infty) \rightarrow (0,\infty)$ with $\lim_{b \to \infty} \varrho(b) = \infty$ such that for $\mathbf{P}$-a.e.\ $\omega \in \Omega$, with $N_0(\omega)$ as in the statement of Lemma~\ref{lem:GoodBoxes}, for all $N \ge N_0(\omega)$,
	\begin{equation}
	   \max_{v \in \mathcal{R}_N(N^{2/(\gamma+1)})} \max_{w \in \mathcal{R}_N(N)} P_{v}^\omega \left( \tau_w > \tau_{\mathcal{R}_{bN}((bN)^{2/(\gamma+1)})^c} \right) \le \frac{1}{\varrho(b)}.
	\end{equation}
\end{lem}

\begin{proof}
    Let $v = (x,\ell) \in \mathcal{R}_N(N^{2/(\gamma+1)})$ and $w = (x',\ell') \in \mathcal{R}_N(N)$. If $x = x'$ and $\ell > \ell'$, a gambler's ruin estimate shows that \begin{equation}
    \label{eq:Gambler-ruin-estimate}
        P_v^\omega\left(\tau_w > \tau_{\mathcal{R}_{bN}((bN)^{2/(\gamma+1)})^c} \right) \leq  P_v^\omega\left(\tau_w > \tau_{(x,\lfloor (bN)^{2/(\gamma+1)}\rfloor + 1 } \right) \leq \frac{1}{b^{2/(\gamma+1)}}
    \end{equation}
    Therefore, we can assume without loss of generality that $x \neq x'$, or $x = x'$ and $\ell \leq \ell'$. Upon a decomposition into the events $\{\tau_{(x,0)} \leq \tau_{(x, \lfloor (bN)^{2/(\gamma+1) }\rfloor + 1)} \}$ and $\{ \tau_{(x,0)} > \tau_{(x, \lfloor (bN)^{2/(\gamma+1) }\rfloor + 1)}  \}$ 
    and using the strong Markov property for the former event 
    one has 
    \begin{equation}
    \label{eq:Decomposition-Exit-time-proof}
    \begin{split}
        P^\omega_{v} \left(\tau_w > \tau_{\mathcal{R}_{bN}((bN)^{2/(\gamma+1)})^c}\right) & \le  P_{(x, 0)}^\omega \left( \tau_{w} > \tau_{\mathcal{R}_{bN}((bN)^{2/(\gamma+1)})^c} \right)  \\
        & + P_v^\omega \left( \tau_{(x, 0)} > \tau_{(x, \lfloor (bN)^{2/(\gamma+1) }\rfloor + 1)} \right).
            \end{split}
    \end{equation}
    The second summand on the right-hand side of~\eqref{eq:Decomposition-Exit-time-proof} is bounded above by $\frac{1}{b^{2/(\gamma+1)}}$, by a gambler's ruin estimate analogous to~\eqref{eq:Gambler-ruin-estimate}.
    We now estimate the first summand on the right-hand side of~\eqref{eq:Decomposition-Exit-time-proof}. By~\eqref{eq:Escape-Prob-bound}, we have
    \begin{equation}
    \label{eq:Bound-via-resistances}
        P_{(x, 0)}^\omega \left( \tau_{w} > \tau_{\mathcal{R}_{bN}((bN)^{2/(\gamma+1)})^c} \right) \le \frac{R_{\mathrm{eff}}((x, 0), w)}{R_{\mathrm{eff}}((x, 0), \mathcal{R}_{bN}((bN)^{2/(\gamma+1)})^c)}.
    \end{equation}
    By~\eqref{eq:Effective-resistance-graph-distance-bound} and using that $|x| \leq N$, $|\ell'| \leq N$, we first bound
    \begin{equation}\label{eq:Res2N}
        R_{\mathrm{eff}}((x, 0), w) \le d((x, 0), w) \le 2N.
    \end{equation}
    On the other hand, we have for $b > 1$ by~\eqref{eq:Monotonicity} that
    \begin{equation}
    \label{eq:Upper-bound-resistance-on-comb-to-outside}
        R_{\mathrm{eff}}((x, 0), \mathcal{R}_{bN}((bN)^{2/(\gamma+1)})^c) \ge R_{\mathrm{eff}}((x, 0), \mathcal{R}^x_{(b-1)N}(((b-1)N)^{2/(\gamma+1)})^c).
    \end{equation}
    By Corollary~\ref{cor:GoodRes}, we see that for every $N \geq N_0(\omega)$, we have 
    \begin{equation}
    \label{eq:Lower-bound-a.s.-on-comb-resistance}
        R_{\mathrm{eff}}((x, 0), \mathcal{R}^x_{(b-1)N}(((b-1)N)^{2/(\gamma+1)})^c) \ge \frac{1}{c_{\mathrm{eff}}\widetilde{c}_{\mathrm{eff}}} (b-1)N.
    \end{equation}
    Setting $\varrho(b) =  \frac{ 2 c_{\mathrm{eff}}\widetilde{c}_{\mathrm{eff}}}{b-1} + \frac{1}{b^{2/(1+\gamma)}}$ yields the required result upon inserting~\eqref{eq:Res2N} and~\eqref{eq:Lower-bound-a.s.-on-comb-resistance} into~\eqref{eq:Bound-via-resistances}.
    \end{proof}

\subsection{Local time estimates}\label{sect:LocalTime}
In this Subsection, we develop bounds on the local times of a random walk on $\mathbb{Z}$, which will be useful in the subsequent Subsection for the proof of certain exit time estimates for the random walk on a typical realization of $\mathrm{Comb}(\mathbb{Z},H^\omega)$. We start with some further notation. Let $(X_n)_{n \in \mathbb{N}_0}$ be the simple random walk on $\mathbb{Z}$, starting from $x \in \mathbb{Z}$, governed by $P_x^{\mathbb{Z}}$. We denote by $\{L_x(t)\}_{x \in \mathbb{Z}}$, the field of local times, namely
\begin{equation}
\label{eq:Local-time-field-SRW}
    L_x(t) \coloneqq \sum_{n = 0}^{\lfloor t \rfloor} \mathds{1}_{\{X_n = x \}}, \qquad t \geq 0.
\end{equation}
We write $T_{L, x}^{\mathrm{ex}} = \tau_{\llbracket x-L, x+L \rrbracket^c}$ for the exit time of $\llbracket x-L, x+L \rrbracket$, $x \in \mathbb{Z}$, $L \in \mathbb{N}_0$, and further abbreviate $T_{L}^{\mathrm{ex}} = T_{L, 0}^{\mathrm{ex}}$. 
We now introduce the local time spent at the roots of deep teeth.
We introduce for $k \in \mathbb{N}_0$, $s, t \in \mathbb{R}_{\geq 0}$, the random variable (defined on $(\Omega \times \mathbb{Z}, \mathcal{G} \otimes \mathcal{F}, \textbf{P} \otimes P_0^{\mathbb{Z}})$)
\begin{equation}
\label{eq:L-deep-def}
	L^{\mathrm{deep}}_k(t, s) \coloneqq \sum_{|x|\le k} L_{x}(t) \mathds{1}_{\{H_x \ge s\}}.
\end{equation}
Informally, $L^{\mathrm{deep}}_k(t, s)$ will be related to the time that a simple random walk spends at the bases of teeth in $\llbracket -k,k \rrbracket$ that are longer than $s$. The next large deviation-type estimate on this quantity will be instrumental in the proof of the main result, Theorem~\ref{theo:MainRandomComb}.

\begin{prop}\label{prop:KeyLTEstimate}
	For all $\delta, \alpha>0$ small enough, there exists $\beta>0$, $C_6, c_6 > 0 $ such that for $\mathbf{P}$-a.e.\ $\omega \in \Omega$ there is $N_1(\omega) \in \mathbb{N}$ such that for $N \geq N_1(\omega)$, $n \in \llbracket N^\delta, N \rrbracket$, for all $t \in (0,1]$, one has
	\begin{equation}
		\sup_{x \in \llbracket -N, N \rrbracket} P_x^{\mathbb{Z}} \left( L^{\mathrm{deep}}_N(T^{\mathrm{ex}}_{n, x}, \sqrt{t}n^{2/(\gamma + 1)}) \le t^{\frac{1}{2} - \alpha}  n^{2/(\gamma + 1)} \right) \le C_6\exp\left\{- c_6 \left( \frac{1}{t}\right)^\beta\right\}.
	\end{equation}
\end{prop}
\begin{proof}
We split the proof in two cases: $t \in [n^{-\varepsilon}, 1]$, and $t < n^{-\varepsilon}$, for some $\varepsilon \in (0,1)$ which will be fixed later. 

\noindent \textbf{Case 1:} $t \in [n^{-\varepsilon}, 1]$. We will employ a martingale concentration bound using the product measure $\textbf{P} \otimes P^{\mathbb{Z}}_0$ and Markov's inequality to obtain a quenched bound. We fix $\eta\in (0, 1/2 -\alpha)$, and obtain for $x \in \llbracket -N,N \rrbracket$, $n \geq N^{\delta}$, and any $\omega \in \Omega$, the bound
\begin{equation}
\label{eq:Prop-4.6-step1}
	\begin{split}
			P_x^{\mathbb{Z}} &\left( L^{\mathrm{deep}}_N(T^{\mathrm{ex}}_{n, x}, \sqrt{t}n^{2/(\gamma + 1)}) \le t^{\frac{1}{2} - \alpha} n^{2/(\gamma + 1)} \right) \\ 
			& \le P_x^{\mathbb{Z}} \left(T^{\mathrm{ex}}_{n, x} \le t^{\eta} n^2 \right) + P_x^{\mathbb{Z}} \left( L^{\mathrm{deep}}_N(T^{\mathrm{ex}}_{n, x}, \sqrt{t}n^{2/(\gamma + 1)}) \le t^{\frac{1}{2} - \alpha} n^{2/(\gamma + 1)}, T^{\mathrm{ex}}_{n, x} \ge t^{\eta} n^2 \right)\\
			& \le Ce^{-ct^{-\eta}} + P_x^{\mathbb{Z}} \left( L^{\mathrm{deep}}_N(T^{\mathrm{ex}}_{n, x}, \sqrt{t}n^{2/(\gamma + 1)}) \le t^{\frac{1}{2} - \alpha} n^{2/(\gamma + 1)}, T^{\mathrm{ex}}_{n, x} \ge t^{\eta} n^2 \right),
	\end{split}
\end{equation}
having used a standard bound on exit times for the simple random walk on $\mathbb{Z}$ in the second step (see, e.g.,~\cite[Proposition 4.33]{Barbook}). We further decompose the second summand in the last line of~\eqref{eq:Prop-4.6-step1} as 
\begin{equation}
\label{eq:Prop-4.6-further-split}
	\begin{split}
		P_x^{\mathbb{Z}} &\left( L^{\mathrm{deep}}_N(T^{\mathrm{ex}}_{n, x}, \sqrt{t}n^{2/(\gamma + 1)}) \le t^{\frac{1}{2} - \alpha} n^{2/(\gamma + 1)}, T^{\mathrm{ex}}_{n, x} \ge t^{\eta} n^2 \right) \\ &\le P_x^{\mathbb{Z}} \left( L^{\mathrm{deep}}_N(T^{\mathrm{ex}}_{n, x}, \sqrt{t}n^{2/(\gamma + 1)}) \le t^{\frac{1}{2} - \alpha} n^{2/(\gamma + 1)}, T^{\mathrm{ex}}_{n, x} \ge t^{\eta} n^2, \sup_{y \in \llbracket x-n, x+n\rrbracket} L_y(T^{\mathrm{ex}}_{n, x}) \le n^{1+\varepsilon} \right) \\ &+ P_x^{\mathbb{Z}} \left(\sup_{y \in \llbracket x-n, x+n\rrbracket} L_y(T^{\mathrm{ex}}_{n, x}) > n^{1+\varepsilon} \right). 
	\end{split}
\end{equation}
We first control the second summand in~\eqref{eq:Prop-4.6-further-split}, and claim that
\begin{equation}
\label{eq:Claim-Prop-4.6-second-term}
	P_x^{\mathbb{Z}} \left(\sup_{y \in \llbracket x-n, x+n\rrbracket } L_y(T^{\mathrm{ex}}_{n, x}) > n^{1+\varepsilon} \right) \le e^{-cn^{\varepsilon}}.
\end{equation} 
To prove~\eqref{eq:Claim-Prop-4.6-second-term}, we first observe that $L_y(T^{\mathrm{ex}}_{n, x}) \preceq L_x(T^{\mathrm{ex}}_{n, x})$ for all $y \in \llbracket x-n,x+n \rrbracket$ (where the laws of both random variables are taken under $P_x^{\mathbb{Z}}$). 
By a standard gambler's ruin calculation, one has that 
\begin{equation}
\label{eq:Local-time-is-geom}
	L_x(T^{\mathrm{ex}}_{n, x}) \sim \mathrm{Geom}(1/n)
\end{equation}
(with $\mathrm{Geom}(p)$ denoting a geometric distribution on $\mathbb{N}$ with probability mass function $p(1-p)^{k-1}$, $k \in \mathbb{N}$),
and therefore (with $c' > 0$)
\begin{equation}
	P_x^{\mathbb{Z}}\left(L_x(T^{\mathrm{ex}}_{n, x}) > n^{1+\varepsilon}\right) \le e^{- \frac{c n^{1+\varepsilon}}{n}}. 
\end{equation}
A union bound over the $2n+1$ sites $y \in \llbracket x-n, x+n \rrbracket$ gives \eqref{eq:Claim-Prop-4.6-second-term}.

We proceed with the first summand in~\eqref{eq:Prop-4.6-further-split}.
We introduce an abbreviation for the product measure $\mathsf{P} = \textbf{P} \otimes P^{\mathbb{Z}}_0 $, and denote the corresponding expectation by $\mathsf{E}$. Furthermore, we define $b_z = \mathds{1}_{\{H_z \ge \sqrt{t} n^{2/(\gamma + 1)} \}}$, and note under $\mathsf{P}$, $\{b_z \}_{z \in \llbracket -n,n\rrbracket}$ are Bernoulli random variables with success parameter $\bar{b} \in [Ct^{-\gamma/2} n^{-2\gamma/(\gamma + 1)}/2, 2Ct^{-\gamma/2} n^{-2\gamma/(\gamma + 1)}]$, independent of $\{L_z\}_{z \in \llbracket -n,n\rrbracket}$.
We rewrite~\eqref{eq:L-deep-def} as
\begin{equation}
	\label{eq:L-deep-special-case}
    L^{\mathrm{deep}}_N(T^{\mathrm{ex}}_{n}, \sqrt{t}n^{2/(\gamma + 1)}) = \sum_{z \in \llbracket n, n\rrbracket} L_z(T^{\mathrm{ex}}_{n}) b_z.
\end{equation}
We define the event (depending only on $(X_n)_{n \in \mathbb{N}_0}$)
\begin{equation}
\label{eq:A-n-t-event}
    \mathcal{A}(n,t) \coloneqq \left\{T^{\mathrm{ex}}_{n} \ge t^{\eta} n^2, \sup_{y \in \llbracket -n, n\rrbracket} L_y(T^{\mathrm{ex}}_{n}) \le n^{1+\varepsilon}\right\}.
\end{equation}
The conditional expectation of~\eqref{eq:L-deep-special-case} given the occupation-time field of the walk up to $T^{\mathrm{ex}}_n$ is
\begin{equation}
\label{eq:For-Bernstein-1}
\begin{split}
	\mathsf{E}\Big[\sum_{z \in \llbracket -n, n\rrbracket} L_z(T^{\mathrm{ex}}_{n}) b_z \, & \Big|\, \{L_z(T^{\mathrm{ex}}_{n})\}_{z \in \llbracket -n, n \rrbracket} \Big] \ge \sum_{z \in \llbracket -n, n\rrbracket} L_z(T^{\mathrm{ex}}_{n}) \frac{C}{2}t^{-\gamma/2} n^{-\frac{2\gamma}{\gamma + 1}}  \\
    & \stackrel{\eqref{eq:A-n-t-event}}{\ge} \frac{C}{2}n^{2}t^{\eta} n^{-\frac{2\gamma}{\gamma + 1}} = \frac{C}{2}n^{\frac{2}{\gamma + 1}}t^{\eta}, \ \mathsf{P}\text{-a.s.~on $\mathcal{A}(n,t)$.}
    \end{split}
\end{equation}
We also have 
\begin{equation}
\label{eq:For-Bernstein-2}
\begin{split}
	\mathsf{E}\Big[\sum_{z \in \llbracket n, n\rrbracket} L_z(T^{\mathrm{ex}}_{n})^2 b_z^2 \, & \Big|\, \{L_z(T^{\mathrm{ex}}_{n})\}_{z \in \llbracket -n, n \rrbracket}\Big] \le 2C\sum_{z \in \llbracket -n, n\rrbracket} L_z(T^{\mathrm{ex}}_{n})^2 t^{-\gamma/2} n^{-\frac{2\gamma}{\gamma + 1}} \\
    & \stackrel{\eqref{eq:A-n-t-event}, \, n^{-\varepsilon} \leq t}{\le} 2Cn^{3+2\varepsilon}t^{-\gamma/2} n^{-\frac{2\gamma}{\gamma + 1}}= 2Cn^{3-\frac{2\gamma}{\gamma + 1}+3\varepsilon}, \ \mathsf{P}\text{-a.s.~on $\mathcal{A}(n,t)$.}
    \end{split}
\end{equation}
We further observe that for $\gamma \in (0,1)$
\begin{equation}
\label{eq:gamma-observation}
	3-\frac{2\gamma}{\gamma + 1} < 4 - \frac{4\gamma}{\gamma + 1} = \frac{4}{\gamma+1}.
\end{equation}
We are now in a position to apply the Bernstein inequality (see, e.g.~\cite[Theorem~2.9.1]{HighDim}) to the sum of independent centered increments 
\begin{equation}
	S_n \coloneqq \sum_{z \in \llbracket -n, n\rrbracket} L_z(T^{\mathrm{ex}}_{n}) \left( b_z - \mathsf{E}[b_z] \right), \qquad n \in \mathbb{N}.
\end{equation}
Using~\eqref{eq:For-Bernstein-1} and~\eqref{eq:For-Bernstein-2}, 
we can choose $\eta > 0$ appropriately such that
\begin{equation}
\begin{split}
	\mathsf{P}\Big( |S_n| > \frac{1}{2} \sum_{z \in \llbracket -n, n\rrbracket} & L_z(T^{\mathrm{ex}}_{n}) \mathsf{E}[b_z]  \Big\vert \{L_z(T^{\mathrm{ex}}_{n})\}_{z \in \llbracket -n, n\rrbracket} \Big) \\
    & \le \exp\left( - c\frac{n^{\frac{4}{\gamma+1} - 2\varepsilon}}{n^{3 - \frac{2\gamma}{\gamma+1} + 3\varepsilon} + n^{3 - \frac{2\gamma}{\gamma+1}+ \varepsilon}}\right), \ \mathsf{P}\text{-a.s.~on $\mathcal{A}(n,t)$.}
\end{split}
\end{equation}
By~\eqref{eq:gamma-observation}, we can choose $\varepsilon \in (0,1)$ small enough in the previous display so that for $\varepsilon_1>0$,
\begin{equation}
\label{eq:After-Bernstein}
\begin{split}
	\mathsf{P}\Big( |S_n| > \frac{1}{2} \sum_{z \in \llbracket -n, n\rrbracket} L_z(T^{\mathrm{ex}}_{n}) \mathsf{E}[b_z] & \mid \{L_z(T^{\mathrm{ex}}_{n})\}_{z \in \llbracket -n, n\rrbracket} \Big)  \\
    & \le \exp\left( - cn^{\varepsilon_1}\right), \ \mathsf{P}\text{-a.s.~on $\mathcal{A}(n,t)$.}
    \end{split}
\end{equation}
Taking the expectation with respect to $\mathsf{E}$ in~\eqref{eq:After-Bernstein}, and observing that $t^{\eta} > t^{\frac{1}{2} - \alpha}$, we obtain for $t \in [n^\varepsilon, 1]$
\begin{equation}
	\mathsf{P}\left( L^{\mathrm{deep}}_N(T^{\mathrm{ex}}_{n}, \sqrt{t}n^{2/(\gamma + 1)}) \le t^{\frac{1}{2} - \alpha} n^{2/(\gamma + 1)},\mathcal{A}(n,t) \right) \le \exp\left( - cn^{\varepsilon_1}\right).
\end{equation}
Finally, by Markov inequality we obtain (possibly adjusting the constant $c>0$) that 
\begin{equation}
	\textbf{P}\left( P_0^{\mathbb{Z}} \left( L^{\mathrm{deep}}_N(T^{\mathrm{ex}}_{n}, \sqrt{t}n^{2/(\gamma + 1)}) \le t^{\frac{1}{2} - \alpha} n^{2/(\gamma + 1)},\mathcal{A}(n,t) \right) \le \exp\left( - cn^{\varepsilon_1/2}\right)  \right) \geq 1 - \exp(-cn^{\varepsilon_1/2}).
\end{equation}
 Taking a union bound over all intervals $\llbracket x-n, x+n \rrbracket$ inside $\llbracket -N, N \rrbracket$ and all $n \ge N^{\delta}$ (the number of those being polynomially large) one obtains the required result by applying the first Borel-Cantelli lemma.

\vspace{8pt}

\noindent \textbf{Case 2: }$t < n^{-\varepsilon}$.
We use the choice of the parameter $\varepsilon \in (0, 1)$ that was fixed above~\eqref{eq:After-Bernstein}. For our bounds, $t = n^{-\varepsilon}$ will be shown to be the worst case scenario. We obtain that in that case the probability of interest is bounded  above by $\exp(-cn^{c \varepsilon})$ which can be written as $\exp(-ct^{-\beta})$ for all $t < n^{-\varepsilon}$ simultaneously. Hence, we can restrict our proof to the case $t = n^{-\varepsilon}$.

\noindent \textbf{First estimate: Empty intervals.} We define the (random) set
\begin{equation}
\label{eq:Deep-n}
    \mathrm{Deep}(n) \coloneqq \left\{z \in \llbracket -N, N \rrbracket \colon H_z \ge n^{\frac{2}{\gamma+1} - \frac{\varepsilon}{6}}\right\}.
\end{equation}
We observe that the typical distance (in $\llbracket -N, N\rrbracket$) between teeth of height at least $n^{\frac{2}{\gamma+1} - \frac{\varepsilon}{6}}$ is of the order $n^{\frac{2\gamma}{\gamma+1} - \frac{\varepsilon\gamma}{6}}$. For any $x \in \mathbb{Z}$ with $\llbracket x -n,x+ n \rrbracket \subseteq \llbracket  -N, N \rrbracket$, we decompose the interval $\llbracket x-n, x+n\rrbracket$ into sub-intervals of length $\ell(n) = \lfloor \frac{1}{3}n^{\frac{2\gamma}{\gamma+1} - \frac{\varepsilon\gamma}{12}}\rfloor$. For $k \in \{- \lfloor n/\ell(n)\rfloor, \dots, \lfloor n/\ell(n)\rfloor\}$, and $\llbracket x-n,x+n \rrbracket \subseteq \llbracket -N,N\rrbracket$ we set
\begin{equation}
    I_{k,x} \coloneqq \llbracket x+(k-1) \ell(n),x+ k\ell(n) -1 \rrbracket.
\end{equation}
Then, for any $n \in \mathbb{N}$
\begin{equation}
	\mathbf{P}\left(I_{1,0} \cap \mathrm{Deep}(n) = \varnothing\right) \le \left(1 - Cn^{-\frac{2\gamma}{\gamma+1} + \frac{\varepsilon\gamma}{6}}\right)^{\lfloor \frac{1}{3}n^{\frac{2\gamma}{\gamma+1} - \frac{\varepsilon\gamma}{12}}\rfloor} \le \exp\left( - \frac{C}{6}  n^{\frac{\varepsilon\gamma}{12}}\right).
\end{equation}
Moreover, for any fixed $n \in \llbracket N^{\delta},N\rrbracket $, the random variables $( \mathds{1}_{I_{k,x} \cap \mathrm{Deep}(n) \neq \varnothing} )_{k,x}$ for $k \in \{- \lfloor n/\ell(n)\rfloor, \dots, \lfloor n/\ell(n)\rfloor\}$, and $\llbracket x-n,x+n \rrbracket \subseteq \llbracket -N,N\rrbracket$ are equal in distribution. We define the events
\begin{equation}
\label{eq:O-n-def}
    O_n(x) \coloneqq \bigcap_{k = - \lfloor n/\ell(n)\rfloor}^{ \lfloor n/\ell(n)\rfloor} \{I_{k,x} \cap \mathrm{Deep}(n) \neq \varnothing \}, \qquad n \in \llbracket N^{\delta}, N\rrbracket, \llbracket x-n,x+n\rrbracket \subseteq \llbracket-N,N\rrbracket.
\end{equation}
By taking a union bound over all $N^{\delta} \le n \le N$ and all intervals $\llbracket x-n, x+n\rrbracket \subseteq \llbracket -N, N\rrbracket$ we obtain that each interval $I_{k,x}, k \in \{- \lfloor n/\ell(n)\rfloor, \dots, \lfloor n/\ell(n)\rfloor\}$, contains at least one element of $\mathrm{Deep}(n)$ for all $N^\delta \leq n \leq N$ with probability
\begin{equation}
\label{eq:O-n-bound}
    \mathbf{P}\Big( \bigcap_{n = \lfloor N^{\delta}\rfloor + 1} ^{N} \bigcap_{\llbracket x-n,x+n \rrbracket \subseteq \llbracket -N, N\rrbracket} O_n(x) \Big) \ge 1 - \exp\left( - \frac{C}{12} N^{\frac{\delta\varepsilon\gamma}{12}}\right).
\end{equation}
For $t \in \mathbb{N}$ and $A\subseteq \mathbb{Z}$, let $L(t, A) \coloneqq \sum_{x \in A} L_x(t)$ (recall~\eqref{eq:Local-time-field-SRW}). We observe that for any $t \le n^{-\varepsilon}$, and therefore for any $x \in \mathrm{Deep}(n)$ we have $H_x > \sqrt{t}n^{2/(\gamma+1)}$. Hence, we obtain
\begin{equation}
    P_x^{\mathbb{Z}}\left( L^{\mathrm{deep}}_N(T^{\mathrm{ex}}_{n, x}, \sqrt{t}n^{2/(\gamma + 1)}) \le t^{\frac{1}{2} - \alpha} n^{2/(\gamma + 1)}\right)\le  P_x^{\mathbb{Z}}\left( L(T^{\mathrm{ex}}_{n, x}, \mathrm{Deep}(n)) \le t^{\frac{1}{2} - \alpha} n^{2/(\gamma + 1)}\right).
\end{equation}
We now provide an upper bound on the right-hand side of the previous display, assuming that $O_n(x)$ holds for all $\llbracket x-n,x+n\rrbracket \subseteq \llbracket -N,N\rrbracket$, $N^\delta \leq n \leq N$. Since the bound will not depend on $x$, we now assume $x = 0$ without loss of generality. On the event $O_n \coloneqq O_n(0)$, we define the set $\mathcal{D} = \{v_k\}_{k \in \{- \lfloor n/\ell(n)\rfloor, \dots, \lfloor n/\ell(n)\rfloor\}}$ where $v_k$ is the leftmost point in $I_k \cap \mathrm{Deep}(n)$ (any deterministic choice would suffice here). 

\noindent \textbf{Second estimate: Distinct visits to traps.} We consider the set
\begin{equation}
     F \coloneqq 3\ell(n) \mathbb{Z} = \left\{l 3\lfloor \frac{1}{3} n^{\frac{2\gamma}{\gamma+1} - \frac{\gamma \varepsilon}{12}}\rfloor \, \colon \, l \in \mathbb{Z}\right\} (\subseteq \mathbb{Z})
\end{equation}
and define a `coarse-grained' random walk $(Z_k)_{k \in \mathbb{Z}}$ as by setting
\begin{equation}
    Z_k = X_{W_k},
\end{equation}
where $(W_k)_{k \in \mathbb{N}_0}$ are the successive hitting times of the set $F$ by the random walk $(X_n)_{n \in \mathbb{N}_0}$, i.e.
\begin{equation}
\begin{cases}
    W_0 \coloneqq 0, \\
    W_{k+1} \coloneqq  \inf\{j > W_k \colon X_j \in F\}, \qquad k \in \mathbb{N}.
\end{cases}
\end{equation}
By the strong Markov property, under $P_0^{\mathbb{Z}}$, $(Z_k)_{k \in \mathbb{N}_0}$ is a nearest-neighbor symmetric simple random walk on $F$ and therefore the rescaled walk $(\frac{Z_k}{3\ell(n)})_{k \in \mathbb{N}_0}$ has again the law $P^{\mathbb{Z}}_0$. Setting $k(n) = \lfloor n^{1 - \frac{2\gamma}{\gamma+1} + \frac{\gamma \varepsilon}{12}}\rfloor$, we therefore see that
\begin{equation}
\label{eq:Escape-from-F-bound}
\begin{split}
    P_0^{\mathbb{Z}}\left(\inf\{k \geq 0 \colon Z_k \notin F \cap \llbracket -n,n \rrbracket \}  \le n^{2 (1 - \frac{2\gamma}{\gamma+1} + \frac{\gamma \varepsilon}{18})}\right) & \le  P_0^{\mathbb{Z}}(T^{\mathrm{ex}}_{k(n)/2} \leq n^{2 (1 - \frac{2\gamma}{\gamma+1} + \frac{\gamma \varepsilon}{18})} ) \\
    &\le \exp(-n^{\varepsilon\gamma/36}).
    \end{split}
\end{equation}
We define the subset
\begin{equation}
    \widetilde{\mathcal{D}} \coloneqq \{z \in \mathcal{D} \colon |x-z| \geq \ell(n) \text{ for all }x \in F \} (\subseteq \mathcal{D}),
\end{equation}
which corresponds to all points in $\mathcal{D}$ with a distance greater than or equal to $\ell(n)$ to all points in $F$ (see Figure~\ref{fig:FigureInterval} for a visualization). Recall the definition of the event $O_n$ in~\eqref{eq:O-n-def}. If $\omega \in O_n$, then between any two consecutive steps of $Z$, the walk $X$ must enter $\widetilde{\mathcal{D}}$ at some point $x \in \widetilde{\mathcal{D}}$ and leave $\widetilde{I}_x \coloneqq \llbracket x-\ell(n)/2, x + \ell(n)/2\rrbracket$ after. 
More formally, we set $\mathcal{I} \coloneqq \bigcup_{x \in \widetilde{\mathcal{D}}} \widetilde{I}_x$ and iteratively define the following random times:

\begin{equation}
\begin{cases}
    & \rho_0 \coloneqq 0, \\
    & \sigma_0 \coloneqq \inf\{k \ge 0 \colon X_k \not \in \mathcal{I}\}, \\
    & \rho_1 \coloneqq \inf\{k \ge 0 \colon X_k  \in  \widetilde{\mathcal{D}}\}, \\
    & \sigma_{j+1} \coloneqq \sigma_0 \circ \vartheta_{\rho_j} + \rho_j, \qquad j \geq 1, \\
    & \rho_{j+1} \coloneqq \rho_1 \circ \vartheta_{\sigma_j} + \sigma_j, \qquad j \geq 1, \\
\end{cases}
\end{equation}

(again using the convention $\inf \varnothing = \infty)$. On the event $O_n$ (and using recurrence of the simple random walk on $\mathbb{Z}$) we see that all random times defined above are $P_0^{\mathbb{Z}}$-a.s.~finite. 
We now define the random variable $\widetilde{L}(T^{\mathrm{ex}}_{n}, \widetilde{\mathcal{D}}) \coloneqq \sup\{k \ge 0 \colon \sigma_k \le  T^{\mathrm{ex}}_{n}\}$, which is the number of excursions of $X$ from $\widetilde{\mathcal{D}}$ to $\mathcal{I}^c$ before $T_n^{\mathrm{ex}}$. 
The observations above imply that 
\begin{equation}
\begin{split}
    	P_0^{\mathbb{Z}}\bigg(\widetilde{L}(T^{\mathrm{ex}}_{n}, \widetilde{\mathcal{D}}) &\le n^{2 (1 - \frac{2\gamma}{\gamma+1} + \frac{\gamma \varepsilon}{18})}\bigg)  \mathds{1}_{O_n}
    \\ &\le  P_0^{\mathbb{Z}}\left(\inf\{k \geq 0 \colon Z_k \notin F \cap \llbracket -n,n \rrbracket \}  \le n^{2 (1 - \frac{2\gamma}{\gamma+1} + \frac{\gamma \varepsilon}{18})}\right) \mathds{1}_{O_n}\\
    & \stackrel{\eqref{eq:Escape-from-F-bound}}{\le} 2 \exp(-n^{\varepsilon\gamma/36}).
    \end{split}
\end{equation}
\begin{figure}[H]
    \centering
    \includegraphics[width=0.5\linewidth]{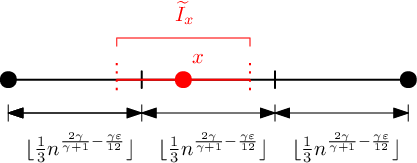}
\caption{Visualization of a the base point of a `deep' tooth in the middle interval. Upon hitting the red point within an excursion between the two black points, the random walk exits the red interval around it.}
\label{fig:FigureInterval}
\end{figure}
\noindent \textbf{Third estimate: Local times between excursions.}
 
 Consider a simple symmetric random walk $(X_n)_{n \geq 0}$ under $P_0^{\mathbb{Z}}$ and recall the notation~\eqref{eq:Local-time-field-SRW} for its local time field. For $\nu > 0$, we have $ L_0(T^{\mathrm{ex}}_{\lfloor n^\nu \rfloor}) \sim \mathrm{Geom}( \lfloor n^\nu \rfloor)$ by~\eqref{eq:Local-time-is-geom}. It follows that for any $n \geq c(\nu)$, 
 \begin{equation}
     P_0^{\mathbb{Z}}\big( L_0(T^{\mathrm{ex}}_{\lfloor n^\nu \rfloor}) \geq n^{\nu - \frac{\gamma\varepsilon}{36}} \big) \geq \frac{1}{2}.
 \end{equation}

 On the event $\widetilde{O}_n \coloneqq \{\widetilde{L}(T^{\mathrm{ex}}_{n}, \widetilde{\mathcal{D}}) > n^{2 (1 - \frac{2\gamma}{\gamma+1} + \frac{\gamma \varepsilon}{18})}\}$ there are $M_n \coloneqq \lfloor n^{2 (1 - \frac{2\gamma}{\gamma+1} + \frac{\gamma \varepsilon}{18})}\rfloor$ times $\{\rho_j\}_{j \ge 1}$ in which $X$ is in some $x \in \widetilde{\mathcal{D}}\subseteq \mathrm{Deep}(n)$ and then exits the interval $\widetilde{I}_x$ of width $\lfloor
 \frac{1}{3}n^{\frac{2\gamma}{\gamma+1}-\frac{\gamma\varepsilon}{12}}\rfloor$ around $x$. 

 \begin{equation}
     A_0 \coloneqq \left\{ L_{X_0}(T^{\mathrm{ex}}_{\lfloor \ell(n)/2 \rfloor, X_0}) \ge \frac{1}{3}n^{\frac{2\gamma}{\gamma+1} - \frac{\varepsilon\gamma}{12}-\frac{\gamma \varepsilon}{36}}\right\},
 \end{equation}
 and for $j \in \mathbb{N}$, we define the events $(A_j)_{j \in \mathbb{N}}$ by setting $\mathds{1}_{A_j} \coloneqq \mathds{1}_{A_0} \circ \vartheta_{\rho_j}$. We now set $\nu = \frac{2\gamma}{\gamma+1} - \frac{\varepsilon\gamma}{12}$. By definition, on $\widetilde{O}_n \cap O_n$, one has
 \begin{equation}
      \frac{1}{3}n^{\frac{2\gamma}{\gamma+1} - \frac{\varepsilon\gamma}{12}-\frac{\gamma \varepsilon}{36}} \sum_{j = 1}^{M_n} \mathds{1}_{A_j} \le L(T^{\mathrm{ex}}_{n, x}, \mathrm{Deep}(n)) .
 \end{equation}
We observe that on the event that 
\begin{equation}
    \sum_{j = 1}^{M_n} \mathds{1}_{A_j} \ge \frac{1}{4} M_n,
\end{equation}
we obtain that the total local time spent in teeth deeper than $n^{\frac{2\gamma}{\gamma+1} - \frac{\varepsilon}{6}}$ is bounded from below (up to a multiplicative constant) by
\begin{equation}
	n^{2 (1 - \frac{2\gamma}{\gamma+1} + \frac{\varepsilon\gamma}{18}) + \frac{2\gamma}{\gamma+1} - \frac{\gamma \varepsilon}{36} - \frac{\gamma \varepsilon}{12}} \ge n^{2 - \frac{2 \gamma}{\gamma+1}} = n^{ \frac{2}{\gamma+1}} \ge t^{\frac{1}{2} - \alpha} n^{ \frac{2}{\gamma+1}},
\end{equation}
where in the last inequality we used that $t^{\frac{1}{2} - \alpha} \le 1$.
 
Using the strong Markov property at the times $\rho_j$, we obtain
 \begin{equation}
     \begin{split} P_0^{\mathbb{Z}}\Bigg(&L^{\mathrm{deep}}_N(T^{\mathrm{ex}}_{n}, \sqrt{t}n^{2/(\gamma + 1)}) \le t^{\frac{1}{2} - \alpha} n^{2/(\gamma + 1)}\Bigg) \mathds{1}_{O_n}\\ &\le  P_0^{\mathbb{Z}}\left(L(T^{\mathrm{ex}}_{n, x}, \mathrm{Deep}(n)) \le t^{\frac{1}{2} - \alpha} n^{2/(\gamma + 1)}, \widetilde{O}_n\right) \mathds{1}_{O_n} + P_0^{\mathbb{Z}}\left(\widetilde{O}_n^c\right)\mathds{1}_{O_n}\\
         &\le P_0^{\mathbb{Z}}\left( \sum_{j = 1}^{M_n} \mathds{1}_{A_j} \le \frac{1}{4}M_n\right) +2 \exp(-n^{\varepsilon\gamma/36})\\
         & \le \mathbf{Q}\left(\mathrm{Binom}\left(M_n, \frac{1}{2}\right) \le \frac{1}{4}M_n\right) +2 \exp(-n^{\varepsilon\gamma/36}) \\
         &\le 3\exp(-n^{\varepsilon\gamma/36}).
     \end{split}
 \end{equation}

The proof is concluded as $O_n(x)$ holds almost surely for all $N \ge N_1(\omega)$ and $n \in \llbracket  N^\delta , N \rrbracket$ and $\llbracket x-n, x+n \rrbracket \subseteq \llbracket -N, N \rrbracket$.
\end{proof}

\subsection{Exit time estimates}\label{sect:Exit}

The key result of this Subsection is an exponential bound on the probability of a fast exit from a disk for the random walk on a typical realization of $\mathrm{Comb}(\mathbb{Z},H^\omega)$. We recall the notation $\tau_{x_1,N} = \tau_{(D^{x_1}_N)^c}$ for $x_1 \in \mathbb{Z}$, $N \in \mathbb{R}_{\geq 0}$.

\begin{lem}\label{lemma:ExitCombExp}
	For all $\delta>0$ small enough, there exist $\widetilde{\beta} >0$, $C_7, c_7 > 0 $ such that for $\mathbf{P}$-a.e.\ $\omega \in \Omega$, $N_1(\omega)$ as in the statement of Proposition~\ref{prop:KeyLTEstimate}, for $N \geq N_1(\omega)$, for all $n \in \llbracket N^\delta, N \rrbracket$, for all $t \in (0,1]$, for any $x = (x_1,\ell) \in D_N$, one has
	\begin{equation}
    \label{eq:Main-exit-bound}
	 P_x^\omega \left( \tau_{x_1,n} < t n^{4/(1+\gamma)} \right) \le C_7\exp\left\{- c_7 \left( \frac{1}{t} \right)^{\widetilde{\beta}} \right\}.
	\end{equation}
\end{lem}

\begin{proof} 
    For a fixed $\omega\in \Omega$, $n \in \mathbb{N}$, and $t \in (0,1]$ we have the stochastic domination, 
    \begin{equation}
    \label{eq:Stochastic-domination}
          tn^{4/(1+\gamma)} \sum_{k = 1}^{ L_N^{\mathrm{deep}}(T_{n,x}^{\mathrm{ex}}, \sqrt{t}n^{2/(\gamma+2)})} \xi_{k} \text{ (under $P^{\mathbb{Z}}_x \otimes \textbf{Q}$) } \preceq \tau_{x_1,n} \text{ (under $P^\omega_x$) }
    \end{equation}
    where $(\xi_k)_{k \in \mathbb{N}}$ is a sequence of Bernoulli random variables (defined on some auxiliary space $(S,\mathcal{S},\textbf{Q})$) with success parameter $\frac{c}{\sqrt{t} n^{2/(\gamma+1)}}$, for some $c \in (0,1)$, as we now explain. We construct an explicit coupling: Each time the random walk $(X_n)_{n \in \mathbb{N}_0}$ on $\mathrm{Comb}(\mathbb{Z},H^\omega)$ (under $P_x^\omega$) changes its horizontal position to a point $z \in \llbracket x-n,x+n\rrbracket$ with $H_z^\omega \geq \sqrt{t}n^{2/(\gamma+1)}$, we perform an `experiment' that succeeds when the walk remains in the same horizontal position for the following $\lfloor t n^{4/(1+\gamma)}\rfloor + 1$ time steps. The latter occurs with a probability at least $ct^{-1/2}n^{-2/(1+\gamma)}$ (for some $c \in (0,1)$) by a gambler's ruin calculation and a standard diffusive bound for a simple random walk on $\mathbb{Z}$, and independently of previous experiments. Since the horizontal steps of $(X_n)_{n \in \mathbb{N}_0}$ under $P_x^\omega$ have the law $P_x^{\mathbb{Z}}$ (by the strong Markov property),~\eqref{eq:Stochastic-domination} follows. Therefore, we obtain, for a fixed $\alpha \in (0,1/2)$,
	\begin{equation}
    \label{eq:After-the-domination-argument}
		\begin{split}
			P_x^\omega \left( \tau_{x_1,n} < t n^{4/(1+\gamma)} \right)
			& \stackrel{\eqref{eq:Stochastic-domination}}{\le} P_x^{\mathbb{Z}} \left( L^{\mathrm{deep}}_N(T^{\mathrm{ex}}_{n, x}, \sqrt{t}n^{2/(\gamma + 1)}) \le t^{\frac{1}{2}- \alpha}  n^{2/(\gamma + 1)} \right) \\
            &+ \textbf{Q}\left(\mathrm{Binom}\left( t^{1/2 - \alpha} n^{2/(\gamma + 1)}, \frac{c}{t^{1/2} n^{2/(\gamma + 1)}}\right) = 0 \right). 
		\end{split}
	\end{equation}
The first quantity on the right-hand side of~\eqref{eq:After-the-domination-argument} is bounded by $C_6 \exp\left(- c_6 \left( \frac{1}{t^{1/2-\alpha}} \right)^\beta\right)$ by Proposition~\ref{prop:KeyLTEstimate}, for all $n \geq N_1(\omega)$, $n \in \llbracket N^\delta,N\rrbracket$ on a set of $\textbf{P}$-probability one.
For the second term on the right-hand side of~\eqref{eq:After-the-domination-argument} one obtains (using that $1+x\leq e^x$ for any $x \in \mathbb{R}$),
	\begin{equation}
    \begin{split}
		\textbf{Q}\left( \mathrm{Binom}\left( t^{1/2 - \alpha} n^{2/(\gamma + 1)}, \frac{c}{t^{1/2} n^{2/(\gamma + 1)}}\right) = 0 \right) & \le \left(1  - \frac{c}{t^{1/2} n^{2/(\gamma + 1)}}\right)^{\lfloor t^{1/2 - \alpha} n^{2/(\gamma + 1)}\rfloor}
        \\ &\le C' e^{-c't^{-\alpha}}.
    \end{split}
	\end{equation}
Combining the two bounds and choosing $\widetilde{\beta} \coloneqq \beta(1/2-\alpha) \vee \alpha $, we obtain~\eqref{eq:Main-exit-bound}.
\end{proof}

\subsection{Heat-kernel estimates}\label{sect:HKRandom}

In this Subsection, we derive three heat-kernel estimates that will be instrumental for the application of a second moment method, for some $\varepsilon > 0$ to be chosen later:
\begin{enumerate}
	\item a lower bound for the heat kernel inside $D_{N}$.
	\item an upper bound  for the heat kernel inside $D^x_{n} \subseteq D_{N}$ for all $x \in \llbracket -N, N \rrbracket$ and all $n \ge N^{\varepsilon}$.
	\item an upper bound  for the heat kernel outside $D^x_{n} \subseteq D_{N}$ for all $x \in \llbracket -N, N \rrbracket$ and all $n \ge N^{\varepsilon}$.
\end{enumerate}
These heat kernel bounds will be valid for $\textbf{P}$-a.e.~environment $\omega \in \Omega$, for large enough $N$. The first lemma gives a lower bound on the on-diagonal heat kernel.

\begin{lem}\label{lem:LowerBoundDiagonalHK}
	For $\delta_2 > 0$ small enough, for all $\widetilde{b} >1$ large enough, there exists a constant $c_8 > 0$ such that for $\mathbf{P}$-a.e.~$\omega \in \Omega$, there exists $N_2(\omega) \in \mathbb{N}$ such that for all $N \geq N_2(\omega)$, and $n \ge N^{\delta_2}$, we have uniformly over $x = (x_1, x_2) \in \mathcal{R}_N(N)$ that
	\begin{equation}
		p_{2\lfloor n/2 \rfloor}^{\omega, \mathcal{R}^{x_1}_{\widetilde{b}N}((\widetilde{b}N)^{2/(\gamma+1)})} (x, x) \ge c_8 n^{-\frac{3-\gamma}{4}}.
	\end{equation}
\end{lem}
\begin{proof}
	We begin by showing the result for $x = (x_1, 0), x_1 \in \llbracket -N, N \rrbracket$ on the horizontal axis. We follow the a similar strategy as in the proof of Proposition~\ref{prop:LBHKReg}. Let $a \in \mathbb{N}$ be an integer constant whose value will be specified below. By~\eqref{eq:Exit-time-estimate} the quenched heat kernel fulfills
    \begin{equation}
    \label{eq:First-step-on-diagonal-random-HK}
        P^\omega_x\left(\tau_{\partial \mathcal{R}^{x_1}_{k}((ak)^{2/(\gamma+1)})} > n \right)^2 \le 3p^{\omega, \mathcal{R}^{x_1}_{k}((ak)^{2/(\gamma+1)})}_{2n}(x,x) |\mathcal{R}^{x_1}_{k}((ak)^{2/(\gamma+1)})|.
    \end{equation}
    We now bound the volume of the set $\mathcal{R}^{x_1}_{k}((ak)^{2/(\gamma+1)})$ from above. Firstly, we observe that $\mathcal{R}^{x_1}_{ak}((ak)^{2/(\gamma+1)}) \supseteq \mathcal{R}^{x_1}_{k}((ak)^{2/(\gamma+1)})$, and we provide a bound on the former. To that end, we note that for any $k \in \mathbb{N}$, 
    \begin{equation}
    \label{eq:Volume-bound-random-1}
        |\mathcal{R}_{ak}^{x_1}((ak)^{2/(\gamma+1)})| \le \sum_{\ell=0}^{u(ak)} \big|\mathbb{V}_{ak}^{x_1}(2^{\ell}) \setminus \mathbb{V}_{ak}^{x_1}(2^{\ell + 1})\big|.
    \end{equation}
    By taking expectations, we see that for any $k \geq \widetilde{c} (> 0)$, 
    \begin{equation}
        \label{eq:Volume-bound-random-2}
    \begin{split}
        \mathbf{E}\left[\sum_{\ell=0}^{u(ak)} \big|\mathbb{V}_{ak}^{x_1}(2^{\ell}) \setminus \mathbb{V}_{ak}^{x_1}(2^{\ell + 1})\big|\right] & \le 2 ak\int_1^{(ak)^{2/(\gamma+1)} + 1} Cs^{-\gamma} \mathrm{d} s \\
        & \le C_* k k^{\frac{2(1 - \gamma)}{\gamma+1}} \le C_* k^{(3-\gamma)/(\gamma+1)},
    \end{split}
    \end{equation}
    for some $C_* > 0$ depending on $a, \gamma$. By Corollary~\ref{cor:GoodVolumes} (recall $a \geq 1)$, one has $|\mathbb{V}_{ak}^{x_1}(2^{\ell}) \setminus \mathbb{V}_{ak}^{x_1}(2^{\ell + 1})|\leq  c_{\mathrm{vol}} \mathbf{E}[|\mathbb{V}_{ak}^{x_1}(2^{\ell}) \setminus \mathbb{V}_{ak}^{x_1}(2^{\ell + 1})|]$ whenever $k \ge N^{\delta_1}$, $N \ge N_0(\omega) \vee \widetilde{c}^{1/\delta_1}$, and $\ell \le u(k)$, and therefore (under the same conditions), 
    \begin{equation}
        |\mathcal{R}^{x_1}_{k}((ak)^{2/(\gamma+1)})|  \stackrel{\eqref{eq:Volume-bound-random-1}, \eqref{eq:Volume-bound-random-2}}{\le} c_{\mathrm{vol}} C_* k^{(3-\gamma)/(\gamma+1)}.
    \end{equation}
    
    Next, we provide an upper bound on the term on the left-hand side of~\eqref{eq:First-step-on-diagonal-random-HK}. To that end, we observe that for every $n \in \mathbb{N}$,
    \begin{equation}\label{eqn:twoEstimatesRes}
        P^\omega_x\left(\tau_{\partial \mathcal{R}^{x_1}_{k}((ak)^{2/(\gamma+1)})} \le n \right)  \le P^\omega_x\left(\tau_{\partial_u \mathcal{R}^{x_1}_{k}((ak)^{2/(\gamma+1)})} \le \tau_{\partial_s \mathcal{R}^{x_1}_{k}((ak)^{2/(\gamma+1)})} \right) + P^\omega_x\left(\tau_{x_1, k} \le n \right).
    \end{equation}
    To control the first term on the right-hand side of \eqref{eqn:twoEstimatesRes}, we use that by Corollary~\ref{cor:GoodRes}, for $\textbf{P}$-a.e.~$\omega \in \Omega$, $N \geq N_0(\omega)$, $k \geq N^{\delta_1}$, by \eqref{eq:Escape-Prob-bound} and the monotonicity~\eqref{eq:Monotonicity}, one has 
    \begin{equation}
    \begin{split}
      P^\omega_x\left(\tau_{\partial_u \mathcal{R}^{x_1}_{k}((ak)^{2/(\gamma+1)})} \le \tau_{\partial_s \mathcal{R}^{x_1}_{k}((ak)^{2/(\gamma+1)})} \right) & \stackrel{\eqref{eq:Monotonicity},\eqref{eq:Escape-Prob-bound}}{\le} \frac{R_{\mathrm{eff}}(x, \partial_s \mathcal{R}^{x_1}_{k}(k^{2/(\gamma+1)}))}{R_{\mathrm{eff}}(x, \mathcal{R}^{x_1}_{ak}((ak)^{2/(\gamma+1)})^c)} \\
      &\le \frac{k \widetilde{c}_{\mathrm{eff}}^2 c_{\mathrm{eff}}^2}{2ak} \le \frac{1}{4},
          \end{split}
    \end{equation}
    for an appropriately chosen constant $a \geq 1$. On the other hand, Lemma~\ref{lemma:ExitCombExp} implies that for $\textbf{P}$-a.e.\ $\omega \in \Omega$, $N \geq N_1(\omega)$, $k \geq N^\delta$, we have
    \begin{equation}
        P^\omega_x\left(\tau_{x_1, k} \le n  \right) \le \frac{1}{4},
    \end{equation}
    for all $n \le c k^{4/(\gamma+1)}$ for an appropriately chosen $c >0$.
    
    Hence, we obtain
    \begin{equation}
        P^\omega_x\left(\tau_{\partial \mathcal{R}^{x_1}_{k}((ak)^{2/(\gamma+1)})} > n \right) \ge \frac{1}{2}>0,
    \end{equation}
    for all $n \le c k^{4/(\gamma + 1)}$. Putting together these estimates one obtains for $N \geq N_0(\omega) \vee N_1(\omega)$, $k \geq N^\delta$,
    \begin{equation}
        p^{\omega, \mathcal{R}^{x_1}_{k}((ak)^{2/(\gamma+1)})}_{2\lfloor n/2 \rfloor}(x,x) \ge \frac{c_{**}}{c_*k^{(3-\gamma)/(\gamma+1)}},
    \end{equation}
    and for $\widetilde{b}>a$ we deduce 
    \begin{equation}
        p^{\omega, \mathcal{R}^{x_1}_{\widetilde{b}k}((\widetilde{b}k)^{2/(\gamma+1)})}_{2\lfloor n/2 \rfloor}(x,x) \ge \frac{c_{**}}{c_*k^{(3-\gamma)/(\gamma+1)}}.
    \end{equation}
    To obtain the result we can set $k = \lfloor c_{***}n^{(\gamma+1)/4} \rfloor$, for an appropriately chosen constant $c_{***}>0$. If $k \ge N^{\max\{\delta, \delta_1\}}$, we set $\delta_2 = 4\max\{\delta, \delta_1\}/(\gamma+1)$. We also remark that $\delta_1, \delta$ can be chosen arbitrarily small in Corollary~\ref{cor:GoodRes} in Lemma~\ref{lemma:ExitCombExp}.

    For $x = (x_1, x_2) \in \mathcal{R}_N(N)$, we first observe that for all $n \le x_2^2, n \in 2\mathbb{N}_0$ there exists a constant $c>0$ such that $p_n^\omega(x, x) \ge cn^{-1/2}$, hence we assume $n > x_2^2$. Note that, as a consequence, we can assume $c_{***}^{-2/(\gamma+1)}k^{2/(\gamma+1)} > x_2$.
    
    Notice that by union bound
    \begin{equation}
        P^\omega_x\left(\tau_{\partial \mathcal{R}^{x_1}_{k}((ak)^{2/(\gamma+1)})} \le n \right) \le P^\omega_{(x_1, 0)}\left(\tau_{\partial \mathcal{R}^{x_1}_{k}((ak)^{2/(\gamma+1)})} < n \right) +  P^I_{x_2}\left(\tau_{\lfloor (ak)^{2/(\gamma+1)} \rfloor}< n \right),
    \end{equation}
    where $I = \llbracket 0, (ak)^{2/(\gamma+1)}\rrbracket $ and $P^I$ denotes the law time of a simple symmetric random walk inside the interval $I$ and reflected at the boundary. The second summand on the right-hand side is bounded from above by standard estimates on $\mathbb{Z}$ by 
    \begin{equation}
        P^I_{x_2}\left(\tau_{(ak)^{2/(\gamma+1)}} < n \right) \le \exp\left(- c \frac{(ak)^{4/(\gamma+1)}}{n}\right) \le \frac{1}{8},
    \end{equation}
    for $a \geq 1$ large enough recalling that $k = \lfloor c_{***}n^{(\gamma+1)/4} \rfloor$.
    The first summand on the right hand side is dealt as in the case $x = (x_1, 0)$ as above applying Lemma~\ref{lemma:ExitCombExp}. Thus, we obtain the key exit time estimate
    \begin{equation}
        P^\omega_x\left(\tau_{\partial \mathcal{R}^{x_1}_{k}((ak)^{2/(\gamma+1)})} > n \right) \ge \frac{1}{2},
    \end{equation}
    and we can apply again the strategy we used for $x = (x_1, 0)$.
\end{proof}

We now apply the previous result to obtain an off-diagonal lower bound on the heat kernel.
 We define 
\begin{equation}
    \bar{\mathcal{R}}_N(N^{2/(\gamma+1)})\coloneqq \mathcal{R}_N(N^{2/(\gamma+1)})\cup \partial \mathcal{R}_N(N^{2/(\gamma+1)}).
\end{equation}
\begin{lem}\label{lem:LBHKOffDiag}
	There exists $c_9, c_l, c_u > 0$, and  $\widetilde{b}>1$ such that $\mathbf{P}$-a.s.\ for all $N \ge N_2(\omega)$ and all $n \in \llbracket c_l N^{4/(\gamma+1)}, c_u N^{4/(\gamma+1)}\rrbracket$ uniformly over $v = (v_1, v_2) \in \bar{\mathcal{R}}_N(N^{2/(\gamma+1)})$ and $w \in \mathcal{R}_N(N)$ with $n - d(v, w) \in 2 \mathbb{N}_0$
	\begin{equation}
		p_n^{\omega, \mathcal{R}_{\widetilde{b}N}((\widetilde{b}N)^{2/(\gamma+1)})} (v, w) \ge c_9 n^{-\frac{3-\gamma}{4}}.
	\end{equation}
\end{lem}
\begin{proof}
    We let $B = \mathcal{R}^{v_1}_{(\widetilde{b}-2)N}(((\widetilde{b}-2)N))^{2/(\gamma+1)})$ stand for the rectangle of side length $(\widetilde{b}-2)N$ and height $((\widetilde{b}-2)N)^{2/(\gamma+1)}$ about $(v_1, 0)$ and recall that $v$ is contained in $\mathcal{R}^{v_1}_{N}(N^{2/(\gamma+1)})$. Note that $B$ is contained in $\bar{\mathcal{R}}_{\widetilde{b}N}((\widetilde{b}N)^{2/(\gamma+1)})$.
    
	We employ the strategy already used in Proposition~\ref{prop:LBHKReg}. In particular, recall that for any $B \subseteq \mathrm{Comb}(\mathbb{Z},H^\omega)$ and vertices $v,w \in B$,
    \begin{equation}
        p_{n}^{\omega, B}(v, w) \ge P_v^\omega\left( \tau_w \le n \wedge \tau_{B^c}\right)p_{2\lfloor \frac{n}{2}\rfloor }^{\omega, B}(w,w).
    \end{equation}
    The second member of the product on the right-hand side of the last display is controlled by Lemma~\ref{lem:LowerBoundDiagonalHK}. We turn to the first member, and use the bound
    \begin{equation}
        P_v^\omega\left( \tau_w > n \wedge \tau_{B^c}\right) \le P_v^\omega\left( \tau_w \wedge \tau_{B^c} > n \right) + P_v^\omega\left( \tau_w > \tau_{B^c}\right).
    \end{equation} 

    Set $L_{v_1, N} \coloneqq \{((v_1, \lfloor (\widetilde{b}N)^{2/(\gamma+1)}\rfloor + 1), (v_1, 0)\}$, by union bound and Markov's inequality we obtain
    \begin{equation}
        P_v^\omega\left( \tau_w \wedge \tau_{B^c} > n \right) \le P_v^\omega\left( \tau_{L_{v_1, N}} > \frac{n}{2} \right) + P_{(v_1, 0)}^\omega\left( \tau_{B^c} > \frac{n}{2} \right).
    \end{equation}
    By the commute time identity \eqref{eq:Commute-time} and Markov's inequality we obtain that for $n \ge c_{l}N^{\frac{4}{\gamma+1}}$ 
    \begin{equation}
        P_v^\omega\left( \tau_{L_{v_1, N}} > \frac{n}{2} \right) \le \frac{2N^{\frac{4}{\gamma+1}} \widetilde{b}^{2/(\gamma+1)}}{c_{l}N^{\frac{4}{\gamma+1}}} \le \frac{2\widetilde{b}^{2/(\gamma+1)}}{c_l}.
    \end{equation}
    On the other hand by Corollary~\ref{cor:GoodRes} we find that
    \begin{equation}
        R_{\mathrm{eff}}\left( (v_1, 0), B^c \right) \le 16 c_{\mathrm{eff}} \widetilde{c}_{\mathrm{eff}} \widetilde{b}N.
    \end{equation}
    Moreover, the volume of $B$ is upper bounded by $c_{\mathrm{vol}}N^{(3-\gamma)/(\gamma+1)}$ for $c_{\mathrm{vol}}>0$ by Corollary~\ref{cor:GoodVolumes}. Hence, applying again \eqref{eq:Commute-time} and Markov's inequality, we obtain
    \begin{equation}
        P_{(v_1, 0)}^\omega\left( \tau_{B^c} > \frac{n}{2} \right) \le \frac{32 c_{\mathrm{eff}} \widetilde{c}_{\mathrm{eff}}\widetilde{b}c_{\mathrm{vol}}N^{\frac{3-\gamma}{\gamma+1}} N}{c_{l}N^{\frac{4}{\gamma+1}}} \le \frac{32c_{\mathrm{vol}} c_{\mathrm{eff}} \widetilde{c}_{\mathrm{eff}} \widetilde{b}}{c_{l}}.
    \end{equation}
    On the other hand, by Lemma~\ref{lem:ProbandRes} we have
    \begin{equation}
        P_v^\omega\left( \tau_w > \tau_{B^c}\right) \le \frac{1}{\varrho(\widetilde{b})}.
    \end{equation}
    Hence, one can choose first $\widetilde{b}> 0$ and hence $\varrho(\widetilde{b})$ and then $c_l$ (depending on $\widetilde{b}$) such that 
    \begin{equation}
        P_v^\omega\left( \tau_w > n \wedge \tau_{B^c}\right) \ge \frac{1}{2},
    \end{equation}
    the proof is concluded.
\end{proof}

We now move to the corresponding upper bounds on the heat kernel, we begin by giving a on-diagonal statement.

\begin{lem}\label{lem:OnDiagonalHKUB}
	For $\delta_3>0$, there exists a constant $c_{10}>0$ such that for $\mathbf{P}$-a.e.~$\omega \in \Omega$, for all $n \ge N^{\delta_3}$ uniformly over $x = (x_1, 0), x_1 \in \llbracket -N, N \rrbracket$ and $N \ge N_2(\omega)$,
	\begin{equation}
		p_n^\omega (x, x) \le c_{10} n^{-\frac{3-\gamma}{4}}.
	\end{equation}
\end{lem}
\begin{proof}
   We will apply~\eqref{eq:Bound-HK-via-Green}. 
   Note that the Green's function of a simple random walk can be expressed in terms of the effective resistance. If $x \in \llbracket-N, N\rrbracket$, the effective resistance to the boundary of a ball of side length $\ell$ is $\ell/2$ can be calculated by the series and parallel laws, see~\cite{LPBook}. Furthermore, set $\ell = n^{(\gamma+1)/4}$, choose $\ell \ge N^{\delta_1}$ so that the hypotheses of Lemma~\ref{lemma:ExitCombExp} are in place, setting simultaneously $\delta_3\ge 4\delta_1/(\gamma+1)$. Hence, one can choose $\widetilde{b}>0$ such that
   \begin{equation}
   	P_x^\omega \left( \tau_{x, \widetilde{b} t^{\frac{\gamma+1}{4}}} \ge n \right) \ge \frac{1}{2}.
   \end{equation}
    Inserting these observations into the formula \eqref{eq:Bound-HK-via-Green} we obtain
    \begin{equation}
    	p_{n}^\omega (x, x) \le 4 \widetilde{b} n^{\frac{\gamma+1}{4} - 1} = 4\widetilde{b} n^{-\frac{3-\gamma}{4}},
    \end{equation}
    as required.
\end{proof}

Using the on-diagonal estimate, we now deduce a crucial off-diagonal upper bound for sites that are `far' from the starting point.

\begin{prop}\label{prop:HKUBOutsideBox}
	For $\delta_3>0$, there exists a constant $c_{11}>0$ such that for $\mathbf{P}$-a.e.~$\omega \in \Omega$, for all $N \ge N_2(\omega)$ we have that for all $n \ge N^{\delta_3}$ uniformly over $x = (x_1, 0), x_1 \in \llbracket-N, N\rrbracket$ and $y = (y_1, 0), y_1 \in \llbracket-N, N\rrbracket \text{ such that } |y_1-x_1| = k \ge n^{\frac{\gamma+1}{4}}$
	\begin{equation}
		p_n^\omega (x, y) \le c_{11} k^{-\frac{3-\gamma}{\gamma + 1}}.
	\end{equation}
\end{prop}
\begin{proof}
	We follow the lines of the proof of \cite[Lemma~4.10]{BPS}. One can write
    \begin{equation}\label{eq:TwoProbsExit}
        p_n(x, y) = P_x^\omega\left(X_t = y, \tau_{x_1, \lfloor k/2 \rfloor}\le n/2\right) + P_x^\omega\left(X_t = y, \tau_{x_1, \lfloor k/2 \rfloor} > n/2\right).
    \end{equation}
    By reversibility and the fact that $\mathrm{deg}(z) \leq 3$ for all $z \in \mathrm{Comb}(\mathbb{Z},H^\omega)$, one observes that 
    \begin{equation}
        P_x^\omega\left(X_n = y, \tau_{x_1, \lfloor k/2 \rfloor} > n/2\right) \le \frac{1}{9}P_y^\omega\left(X_n = y, \tau_{y_1, \lfloor k/2 \rfloor + 1} \le n/2\right).
    \end{equation}
    The two probabilities on the right-hand side of~\eqref{eq:TwoProbsExit} can therefore be bounded in a similar manner, and we will do so for the first term. Observe that 
    \begin{equation}
        P_x^\omega\left(X_n = y, \tau_{x_1, \lfloor k/2 \rfloor}< n/2\right) \le P_x^\omega\left(\tau_{x_1, \lfloor k/2 \rfloor}< n/2\right) \sup_{0 \le s \le k/2} p_{n-s}^\omega (\bar{x}, y).
    \end{equation}
    Here we denote $\bar{x}$ the unique point at distance $\lfloor k/2 \rfloor$ from $x$  on the shortest path between $x$ and $y$. By Lemma~\ref{lem:OnDiagonalHKUB} we have that for all $\bar{x}, y$ and $r \ge N^{\delta_3}$, 
    \begin{equation}
        p_{r}^\omega (z, z)\le c_{10} r^{-\frac{3-\gamma}{4}} \qquad z \in \{ \bar{x}, y\}.
    \end{equation}
    By a standard application of the Cauchy-Schwarz inequality (see \cite[Corollary~4.8]{BPS}) we obtain that 
    \begin{equation}\label{eqn:CauchySchwarz}
        p_{r}^\omega(\bar{x}, y) \le \sqrt{p_{r}^\omega(\bar{x}, \bar{x}) p_{r}^\omega(y, y)}\le c_{10} r^{-\frac{3-\gamma}{4}}.
    \end{equation}
    Therefore, by considering $r= n-s = \lfloor n/2 \rfloor, \dots, n $ we observe that the minimal value for the upper bound of $\sup_{0 \le s \le \lceil n/2 \rceil} p_{n-s}^\omega (\bar{x}, y)$ is achieved at $s =\lceil n/2 \rceil$.
    
    On the other hand one can apply Lemma~\ref{lemma:ExitCombExp} to the exit probability. Setting $n = k^{4/(\gamma+1)}/\eta$ one obtains, for $n \geq N^{\delta_3}$, $n \geq N_2(\omega)$,
    \begin{equation}
    \begin{split}
        P_x^\omega\left(X_n = y, \tau_{x_1, \lfloor k/2 \rfloor}< n/2\right) &\le c_{10} n^{-\frac{3-\gamma}{4}} C_7\exp\Big(- c_{7}\Big(\frac{k^{\frac{4}{\gamma+1}}}{n} \Big)^\beta\Big)\\
        & \le  C_7c_{10} k^{-\frac{3-\gamma}{\gamma+1}} \sup_{\eta>1} \eta^{3} \exp\left(- c_{7} \eta^\beta \right)\\
        & \le c_{11}k^{-\frac{3-\gamma}{\gamma+1}},
    \end{split}
    \end{equation}
    as required.
\end{proof}

Finally, we obtain an upper bound for the off-diagonal heat kernel for sites that are close to the starting point.

\begin{lem}\label{lem:UBHKInsideAllPoints}
    	For $\delta_4>0$, there exists $c_{12} >0$ such that for $\mathbf{P}$-a.e.~$\omega \in \Omega$, for all $N\ge N_2(\omega)$ we have that for all $n \ge N^{\delta_3}$ uniformly over $x = (x_1, x_2) \in \bar{\mathcal{R}}_N(N^{2/(\gamma+1)})$ and $y = (y_1, 0), y_1 \in \llbracket -N, N\rrbracket$,
	\begin{equation}
		p_n^\omega (x, y) \le c_{12} n^{-\frac{3-\gamma}{4}}.
	\end{equation}
\end{lem}

\begin{proof}
    We show the stronger result that for some $c_{12} >0$ and $x_2 = h, h \ge 0$
    \begin{equation}
		p_n^\omega (x, y) \le c_{12} n^{-\frac{3-\gamma}{4}} e^{-\frac{h^2}{c_{12} n}},
	\end{equation}
    uniformly over the rectangle $\mathcal{R}_N(N^{2/(\gamma+1)})$ and $n \ge N^{\delta_3}$.
    The proof follows the lines of \cite[Lemma~4.9]{BPS}, we give the main steps for completeness but refer to the latter for more details.

    We write $h(x_1) = \lfloor H_{x_1}^\omega  \rfloor$ for the height of the tooth based at $x_1$. Let $T_a^b$ be the hitting time of $a$ by a simple random walk on an interval $\llbracket a, b \rrbracket \cap \mathbb{Z}$. By \cite[(4.9)]{BPS}, employing a standard ballot theorem on an interval of $\mathbb{Z}$, one obtains that for $c>0$
    \begin{equation}\label{eq:Ballot}
        P^{\mathbb{Z}}_h\left( T_0^{h(x_1)} = s \right) \le c \frac{h}{s^{3/2}} e^{-\frac{h^2}{cs}} + c\frac{2h(x_1) - h}{s^{3/2}}e^{-\frac{(2h(x_1) - h)^2}{cs}}, \qquad s \in \mathbb{N}.
    \end{equation}
    By reversibility and \eqref{eqn:CauchySchwarz} we obtain that 
    \begin{equation}\label{eq:SplitFurther}
        p_n^\omega (x, y) = \sum_{s = 1}^{n - 1} P^{\mathbb{Z}}_h\left( T_0^{h(x_1)} = s \right) p_{n-s}^\omega (x, y).
    \end{equation}
    Observe that when $(n-s) \le N^{\delta_3} \le n^{1/2}$, where we set $\delta_4 = 2 \delta_3$, one has $s \ge n/2$ and hence $s^{-3/2} \le 4 n^{-3/2}$. We can split \eqref{eq:SplitFurther} into
    \begin{equation}
        \sum_{s = 1}^{n -2\lfloor n^{1/2} \rfloor} P^{\mathbb{Z}}_h\left( T_0^{h(x_1)} = s \right) c_{10} (n-s)^{-\frac{3-\gamma}{4}} +  \sum_{s = n -2\lfloor n^{1/2} \rfloor+1}^{n} P^{\mathbb{Z}}_h\left( T_0^{h(x_1)} = s \right) c_{10} (n-s)^{-1/2}.
    \end{equation}
    Inserting \eqref{eq:Ballot} into the right-hand side of the last display we obtain
    \begin{equation}
    \begin{split}
        & \sum_{s = n -2\lfloor n^{1/2} \rfloor+1}^{n} P^{\mathbb{Z}}_h\left( T_0^{h(x_1)} = s \right) c_{10} (n-s)^{-1/2}  \\
        & \qquad \qquad \le c h e^{-\frac{h^2}{ct}} \frac{1}{n^{5/4}} + c (2h(x) - h) e^{-\frac{(2h(x) - h)^2}{cn}} \frac{1}{n^{5/4}}.
    \end{split}
    \end{equation}
    
    Following the steps of \cite[Lemma~4.9]{BPS} we also have the bound
    \begin{equation}\label{eq:Dominate}
    \begin{split}
        &\sum_{s = 1}^{n - 1} P^{\mathbb{Z}}_h\left( T_0^{h(x_1)} = s \right) c_{10} (n-s)^{-\frac{3-\gamma}{4}}  \\
        & \qquad \qquad \le c h e^{-\frac{h^2}{ct}} \frac{1}{\sqrt{n} n^{\frac{3-\gamma}{4}}} + c (2h(x) - h) e^{-\frac{(2h(x) - h)^2}{cn}} \frac{1}{\sqrt{n} n^{\frac{3-\gamma}{4}}}. 
    \end{split}
    \end{equation}
    Observe that for all $\gamma>0$, we one has $\sqrt{n} n^{\frac{3-\gamma}{4}} \le n^{5/4}$. Hence, we only need to bound the right-hand side of \eqref{eq:Dominate}. To conclude one observes that $2h(x) - h \ge h$ and inserts in the last display the inequality $x e^{-\frac{x^2}{ct}} \le c' \sqrt{n}e^{-\frac{x^2}{cn}}$.
\end{proof}

\subsection{Second moment method}\label{sect:SecondMomentRandom}

We define a random variable counting the number of collisions inside an appropriately scaled rectangle. Let the two independent copies of the walk be $X^1 = (U^1, V^1)$, $X^2 = (U^2, V^2)$ defined as in Section~\ref{sec:Regularly-growing-comb}. Recall from \eqref{eq:Sets-on-Comb} that $\mathcal{R}^y_{N}(h) = \{v = (x, \ell) \colon |x-y|< N, \ell \le h \}$ furthermore, letting $\theta_N \coloneqq \theta_N^{X^1} \wedge \theta_N^{X^2}$ with $\theta_N^{X^j} \coloneqq \inf\{k \ge 0 \colon X^j_k \not \in \mathcal{R}_N(N^{2/(\gamma+1)})\}, j = 1, 2$, we set
\begin{equation}
	\mathcal{C}_n \coloneqq \left\{X_n^1 = X_n^2, n < \theta_{\widetilde{b}N}, U_n^1 \in \llbracket -N, N \rrbracket, V_n^1 = 0 \right\}.
\end{equation}
We now define
\begin{equation}
	\mathcal{H}_1^{N} \coloneqq \sum_{k =1}^T \mathds{1}_{\mathcal{C}_k},
\end{equation}
with $T\coloneqq \lfloor c_{13} N^{4/(1+\gamma)} \rfloor$ for some $c_{13} > 0$. Next, we show a lower bound for the first moment of the random variable $\mathcal{H}_1^{N}$.

\begin{lem}\label{lem:LBCollExp}
	There exists a constant $c_{14}>0$ such that for $\mathbf{P}$-a.e.~$\omega \in \Omega$, for all $N \ge N_1(\omega)$ we have that uniformly over all $v, w \in \bar{\mathcal{R}}_N(N^{2/(\gamma+1)})$  with $d(v, w)\in 2 \mathbb{N}_0$ for $\ell \le N/2$ 
	\begin{equation}
		E^\omega_{v, w}[\mathcal{H}_1^{N}] \ge c_{14}  N^{\frac{3\gamma-1}{\gamma +1}}.
	\end{equation}
\end{lem}
\begin{proof}
	We apply Lemma~\ref{lem:LBHKOffDiag}. Indeed, by summing the heat kernel between times $\lfloor c_l N^{\frac{4}{\gamma+1}}\rfloor$ and $\lfloor c_u N^{\frac{4}{\gamma+1}}\rfloor$ we obtain, under the hypotheses of Lemma~\ref{lem:LBHKOffDiag} that
    \begin{equation}
        \begin{split}
        E^\omega_{v, w}[\mathcal{H}_1^{N}] &\ge \sum_{n = \lfloor c_l N^{\frac{4}{\gamma+1}}\rfloor}^{\lfloor c_u N^{\frac{4}{\gamma+1}}\rfloor} \sum_{x \in \llbracket -N, N\rrbracket} p_n^\omega(v, (x, 0))p_n^\omega(w, (x, 0)) \ge \sum_{n = \lfloor c_l N^{\frac{4}{\gamma+1}}\rfloor}^{\lfloor c_u N^{\frac{4}{\gamma+1}}\rfloor}  c_9^2 \sum_{x \in \llbracket -N, N\rrbracket} N^{-2\frac{3 - \gamma}{\gamma+1}}\\
        & \ge c_{14} N^{\frac{4}{\gamma+1} + 1 -\frac{6- 2\gamma}{\gamma+1}} = c_{14} N^{\frac{3\gamma-1}{\gamma +1}}.
        \end{split}
    \end{equation}
    The proof is concluded.
\end{proof}

We now give the corresponding upper bound on the second moment for $\mathcal{H}_1^{N}$.

\begin{prop}\label{prop:UBCollSecondMoment}
	For all $\delta_5>0$, there exists a constant $C_{15}$ such that for $\mathbf{P}$-a.e.~$\omega \in \Omega$, for all $N \ge N_3(\omega)$ large enough we have, for all $v, w \in \bar{\mathcal{R}}_N(N^{2/(\gamma+1)})$,
	\begin{equation}\label{eq:SMUB}
		E^\omega_{v, w}\left[(\mathcal{H}_1^{N})^2\right] \le C_{15} \left( N^{\delta_5} + E_{v, w}[\mathcal{H}_1^{N}]^2 \right).
	\end{equation}
\end{prop}
\begin{proof}
    By standard manipulations one obtains
    \begin{equation}
        E_{v, w}^\omega\left[(\mathcal{H}_1^{N})^2\right] = E^\omega_{v, w}\left[\mathcal{H}_1^{N}\right] + 2\sum_{n = 1}^{T} \sum_{k = n+1}^T P^\omega_{v, w} \left( \mathcal{C}_n \cap \mathcal{C}_{k} \right).
    \end{equation}
    Since the first term $E^\omega_{v, w}\left[\mathcal{H}_1^{N}\right]$ is always smaller than the right-hand side of \eqref{eq:SMUB}, we focus on the second term. 

    We observe that by the Markov property
    \begin{equation}
        P^\omega_{v, w} \left( \mathcal{C}_n \cap \mathcal{C}_{k} \right) \le P^\omega_{v, w} \left( \mathcal{C}_n\right) \sup_{z \in \llbracket -N, N \rrbracket} P^\omega_{z, z} \left(\mathcal{C}_{k-n} \right).
    \end{equation} 
    We further notice that 
    \begin{equation}
        E^\omega_{v, w}\left[\mathcal{H}_1^{N}\right] = \sum_{n = 1}^{T} P^\omega_{v, w} \left( \mathcal{C}_n\right),
    \end{equation}
    so it will be sufficient to show that 
    \begin{equation}
        \sum_{k = n+1}^T \sup_{z \in \llbracket -N, N \rrbracket} P^\omega_{z, z} \left(\mathcal{C}_{k-n} \right) \le \frac{c_{16}}{2}E^\omega_{v, w}\left[\mathcal{H}_1^{N}\right] \le \frac{c_{16}}{2} N^{\frac{3\gamma-1}{\gamma +1}},
    \end{equation}
    to conclude. We will now show that this estimate holds.

    Firstly, we observe that
    \begin{equation}
        \sum_{k = n+1}^T \sup_{z \in \llbracket -N, N \rrbracket} P^\omega_{z, z} \left(\mathcal{C}_{k-n} \right) \le \sum_{k = 1}^T \sup_{z \in \llbracket -N, N \rrbracket} P^\omega_{z, z} \left(\mathcal{C}_{k} \right) \le \sup_{z \in \llbracket -N, N \rrbracket}  \sum_{k = 1}^T \sum_{y \in \llbracket -N, N \rrbracket} p_{k}^\omega(z, (y, 0))^2.
    \end{equation}
    We split the sum according to Lemma \ref{lem:UBHKInsideAllPoints} into
    \begin{equation}
         \sum_{k = 1}^T \sum_{y \in  \llbracket -N, N \rrbracket} p_{k}^\omega(z, (y, 0))^2 \le  \sum_{k =1}^{\lfloor N^{4\delta_4}\rfloor} \sum_{y \in  \llbracket -N, N \rrbracket} p_{k}^\omega(z, (y, 0))^2 + \sum_{k = \lfloor N^{4\delta_4}\rfloor+1}^T \sum_{y \in  \llbracket -N, N \rrbracket} p_{k}^\omega(z, (y, 0))^2.
    \end{equation}
    We control the first summand. To that end, note that for any $\omega \in \Omega$, $k \in \mathbb{N}$, $p_{k}^\omega(z, (y, 0))^2 \le c_{17}k^{-1}$ (see, e.g.,~\cite[Corollary 4.4 (b)]{Barbook}), which yields 
    \begin{equation}
        \sum_{k =1}^{\lfloor N^{4\delta_4}\rfloor} \sum_{y \in \llbracket-N^{4\delta_4}, N^{4\delta_4}\rrbracket} p_{k}^\omega(z, (y, 0))^2 \le c_{17}N^{8\delta_4}.
    \end{equation}
    We will fix $\delta_4>0$ so that $c_{17}N^{8\delta_4} \le c_{16} N^{\frac{3\gamma-1}{\gamma +1}}/2$, which is always possible the parameter range $\gamma>1/3$. Note that the bound is uniform over all $z \in  \llbracket -N, N \rrbracket$.

    Hence, we turn to the remaining quantity
    \begin{equation}
        \sum_{k = \lfloor N^{4\delta_4}\rfloor}^T \sum_{y \in \llbracket -N, N \rrbracket} p_{k}^\omega(z, (y, 0))^2.
    \end{equation}
    Since $k \ge \lfloor N^{4\delta_4}\rfloor$, we have that $k^{(\gamma+1)/4} \ge N^\delta_4$. Hence, setting 
    \[I_z^k \coloneqq \llbracket-z - k^{(\gamma+1)/4}, z + k^{(\gamma+1)/4}\rrbracket,\]
    we further split the sum
    \begin{equation}
        \sum_{k = \lfloor N^{4\delta_4}\rfloor}^T \sum_{y \in \llbracket-N, N\rrbracket} p_{k}^\omega(z, (y, 0))^2 = \sum_{k = \lfloor N^{4\delta_4}\rfloor}^T \sum_{y \in I_z^k} p_{k}^\omega(z, (y, 0))^2 + \sum_{k = N^{4\varepsilon}}^T \sum_{y \in \llbracket-N, N\rrbracket \setminus I_z^k} p_{k}^\omega(z, (y, 0))^2.
    \end{equation}
    For the first term, we apply the upper bound at Lemma~\ref{lem:UBHKInsideAllPoints} and obtain
    \begin{equation}
    \begin{split}
        \sum_{k = \lfloor N^{4\delta_4}\rfloor}^T \sum_{y \in I_z^k} p_{k}^\omega(z, (y, 0))^2 &\le \sum_{k = \lfloor N^{4\delta_4}\rfloor}^T \sum_{y \in I_z^k} c_{12}^2 k^{-\frac{6 - 2\gamma}{4}}\\
        &\le 2 c_{12}^2\sum_{k = \lfloor N^{4\delta_4}\rfloor}^T k^{\frac{\gamma+1}{4}}  k^{-\frac{6 - 2\gamma}{4}} \\
        & \le 2 c_{12}^2 \sum_{k = \lfloor N^{4\delta_4}\rfloor}^T  k^{-\frac{5 - 3\gamma}{4}} \le \frac{c_{16}}{16} N^{\frac{3\gamma-1}{\gamma+1}},
            \end{split}
    \end{equation}
    where we adjust $c_{16}$ depending on all preceding constants but independently of $N$.

    Finally, we deal with the term 
    \begin{equation}
        \sum_{k = \lfloor N^{4\delta_4}\rfloor}^T \sum_{y \in \llbracket-N, N\rrbracket \setminus I_z^k} p_{k}^\omega(z, (y, 0))^2.
    \end{equation}
    We aim to apply Proposition~\ref{prop:HKUBOutsideBox}. We re-parametrize the points in terms of their distance $\ell$ from $z$, noting that for each $\ell > k^{\frac{\gamma+1}{4}}$ there are exactly two associated points. Hence, we obtain
    \begin{equation}
    \begin{split}
        \sum_{y \in \llbracket-N, N\rrbracket \setminus I_z^k} p_{k}^\omega(z, (y, 0))^2 &\le 2c_{11}^2\sum_{\ell = \lfloor k^{\frac{\gamma+1}{4}}\rfloor }^N \ell^{-2\frac{3-\gamma}{\gamma+1}} \\
        &\le \frac{c_{16}}{16} \left[ \ell^{\frac{-6+2\gamma}{\gamma+1} + 1}\right]_{k^{\frac{\gamma+1}{4}}}^N \le \frac{c_{16}}{16} k^{\frac{-5+3\gamma}{4}},
            \end{split}
    \end{equation}
    where the last bound uses that $\frac{-6+2\gamma}{\gamma+1} + 1 < 0$ for $\gamma  \in (1/3,1)$. We are left with
    \begin{equation}
        \sum_{k = \lfloor N^{4\delta_4}\rfloor}^T \frac{c_{16}}{16} k^{\frac{-5+3\gamma}{4}} \le \frac{c_{16}}{8} N^{\frac{3\gamma-1}{\gamma+1}},
    \end{equation}
    with the same bound we already employed and possibly adjusting $c_{16}$ again.

\end{proof}

\subsection{Conclusion for $\gamma \in (1/3, 1)$}
\label{subsec:gamma-greater-1/3}

We now prove the first part of our main result. 

\begin{proof}[Proof of Theorem~\ref{theo:MainRandomComb}-\emph{(1)}, $\gamma \in (1/3,1)$.]
We observe that by definition of $\mathcal{H}_1^N $ and the Paley-Zygmund inequality, uniformly over all $x, y \in \bar{\mathcal{R}}_N(N^{\frac{2}{\gamma+1}})$ with $d(x, y) \in 2 \mathbb{N}_0$, there exists $\rho>0$ such that for $\mathbf{P}$-a.e.~$\omega \in \Omega$, for all $N \geq N_1(\omega) \vee N_2(\omega)$
\begin{equation}
\begin{split}
\label{eq:Unif-lower-bound-random-case}
        P^\omega_{x, y} \left(\left\{ \exists n < \theta_{\widetilde{b}N} \colon X_n^1 = X_n^2\right\}\right) &\ge \frac{E_{x, y}^\omega[\mathcal{H}_1^N]^2}{E_{x, y}^\omega[(\mathcal{H}_1^N)^2]}\\
    &\ge \frac{c_{14}^2}{C_{15}} \ge \rho > 0.
    \end{split}
\end{equation}
We now define the sequence of events 
\begin{equation}
    \mathrm{Col}_m \coloneqq \left\{ \exists n \in [\theta_{\widetilde{b}^m}, \theta_{\widetilde{b}^{m+1}}) \cap \mathbb{N}_0 \colon X_n^1 = X_n^2\right\}.
\end{equation}
By the strong Markov property and the uniformity over the starting points in~\eqref{eq:Unif-lower-bound-random-case}, we obtain
\begin{equation}
    P^\omega_{0, 0} \left(\mathrm{Col}_m \mid \mathcal{F}_{\theta_{\widetilde{b}^m}}\right) \ge \min_{x \in \partial \mathcal{R}_{\widetilde{b}^m}(\widetilde{b}^{m2/(1+\gamma)})} \min_{\substack{y \in \bar{\mathcal{R}}_{\widetilde{b}^m}(\widetilde{b}^{2m/(1+\gamma)})  \\
         d(x, y) \in 2 \mathbb{N}_0 }} P^\omega_{x, y}\left(\mathcal{H}_1^{\widetilde{b}^{m}} \ge 1\right)\ge \rho,
\end{equation}
where $\mathcal{F}^2_{\theta_{\widetilde{b}^m}}$ is the sigma field generated by the two random walks up to the random time $\theta_{\widetilde{b}^m}$. By a standard conditional Borel-Cantelli argument this yields
\begin{equation}
    P^\omega_{0, 0} \left(\mathrm{Col}_m \text{ infinitely often}\right) = 1.
\end{equation}
We have thus established the infinite collision property.
\end{proof}

\subsection{Conclusion for $\gamma = 1$}\label{Sect:Gamma1}

This case is covered much more directly using the following lemma due to \cite{Klass1984}. We provide a proof for completeness.

\begin{lem}\label{lem:MaxSmall}
    Let $(X_n)_{n \in \mathbb{N}_0}$ be an~i.i.d.\ sequence of random variables on some probability space $(S,\mathcal{S},\mathbf{Q})$ such that for some $C>0$
    \begin{equation}
        \mathbf{Q}(X_1 \ge t) \sim \frac{C}{t}.
    \end{equation}
    For $M_n \coloneqq \max\{X_1, \dots, X_n\}$, one has 
    \begin{equation}
        \mathbf{Q}\left(\liminf_{n \to \infty} \frac{M_n}{n} \le 1 \right) = 1.
    \end{equation}
\end{lem}
\begin{proof}
    We choose a sequence $(a_n)_{n \ge 0} \subseteq (0,\infty)$ such that (by symmetry)
    \begin{equation}
        \mathbf{Q}\left( \max_{i \le a_{n-1}} X_i \ge  \max_{i \in \llbracket a_{n-1} + 1, a_n\rrbracket} X_i \right) = \frac{a_{n-1}}{a_n}
    \end{equation}
    is summable. By the Borel-Cantelli lemma, the event inside the probability $\mathbf{Q}$-a.s.~occurs finitely many times.
    For the same sequence $(a_n)_{n \ge 0}$ we have that
    \begin{equation}
        \mathbf{Q}\left( \max_{i \in \llbracket a_{n-1} + 1, a_n\rrbracket} X_i \le a_n \right) \ge \left( 1 - \frac{C}{2a_n}\right)^{a_n - a_{n-1}} \to e^{-C/2}, \qquad \text{ as }n \to \infty.
    \end{equation}
    The second Borel-Cantelli lemma implies that
    \begin{equation}
        \mathbf{Q}\left(\max_{i \in \llbracket a_{n-1} + 1, a_n\rrbracket} X_i \le a_n \text{ infinitely often}\right) = 1,
    \end{equation}
    and the claim follows. 
\end{proof}
\begin{proof}[Proof of Theorem~\ref{theo:MainRandomComb}-\emph{(1)}, $\gamma=1$]
    The result follows from Lemma~\ref{lem:MaxSmall} and the Green's kernel criterion \eqref{eq:GKC} applied to the random subsequence $(n_k(\omega))_{k \in \mathbb{N}_0}$ (which exists on an event of $\textbf{P}$-probability one) along which $\max\{H_1^\omega,...,H_{n_k}^\omega\} \le n_k(\omega)$. Indeed, on the subsequence $(n_k(\omega))_{k \in \mathbb{N}_0}$, for any point $y \in D_{n_k}$ by \cite[Theorem 2.6]{Barbook} we have
\begin{equation}
   g_{D_r}(y,y) = R_{\mathrm{eff}}(y, D_{n_k}^c) \le d(y, D_{n_k}^c) \le 2n_k \le 4 g_{D_r}(0,0). \qedhere
\end{equation}
\end{proof}

\subsection{The finite collision regime}\label{sect:FiniteCollision}

Throughout this Subsection, we assume that $(K_x)_{x \in \mathbb{Z}}$ is a sequence of i.i.d.~random variables defined on some 
$(\Omega,\mathcal{G},\textbf{P})$ with
\begin{equation}
\label{eq:Tail-LB-Assumption}
    \mathbf{P}\left( K_0 > t \right) \ge Ct^{-\gamma},  \qquad t \in (0,\infty),
\end{equation}
for some $\gamma \in (0, 1/3)$. The main result of this Subsection is as follows. 

\begin{thm}\label{theo:FinitePrecise}
    For $\textnormal{\textbf{P}}$-a.e.~$\omega \in \Omega$, the comb graph $\mathrm{Comb}(\mathbb{Z},K^\omega)$ has the finite collision property.
\end{thm}

 Theorem~\ref{theo:FinitePrecise} implies Theorem~\ref{theo:MainRandomComb}-\emph{(2)}. A variant of the former (that already implies Theorem~\ref{theo:MainRandomComb}-\emph{(2)}) was proved in \cite{Koops} under the additional assumption that $\mathbf{P}\left( H_0 > t \right) \le C't^{-1/M}$ for arbitrary large but fixed $M>0$, $t \in \mathbb{R}_{\geq 0}$. To remove this condition, we count the collisions in space-time sets (as opposed to considering sets indexed by space as in~\cite{BPS, Koops}). \medskip 

We begin by deriving an exit time estimate for the random walk. Although we do not believe it are sharp (as in the previous Sections), it suffices for our purpose. 

\begin{lem}\label{lemma:ExitCombExpHeavy}
	There exist constants $C_{\mathrm{tail}}, c_{\mathrm{tail}},C_{18}, c_{18}, \beta'>0$ such that for all $\delta>0$ small enough, there exists a family of random variables $\{\rho_x\}_{x \in \mathbb{Z}}$ with the property that
    \begin{equation}
        \mathbf{P}\left(\rho_x > s\right) \le C_{\mathrm{tail}}e^{- c_{\mathrm{tail}} s^{\beta'}},
    \end{equation}
    such that the following holds: For $t<1$ and $n \ge \rho_x$,
	\begin{equation}
		P_{0}^\omega \left( \tau_{x, n} \le t n^{3 + \delta} \right) \le C_{18}\exp\left\{- c_{18} \left( \frac{1}{t} \right)^{\beta'}\right\}.
	\end{equation}
\end{lem}
\begin{proof}
    This bound follows by the same argument given in Proposition~\ref{prop:KeyLTEstimate} in the case $t \le n^{-\varepsilon}$ as $\frac{2}{\gamma + 1} > \frac{2}{1/3 + 1} = \frac{3}{2}$ for all $\gamma \in (0, 1/3)$. We only explain the necessary adaptations. 
    
    Pick $\varepsilon>0$ small enough so that $\frac{3}{2} + \varepsilon < \frac{2}{\gamma + 1} $. Then the set $\mathrm{Deep}(n)$ (defined as in~\eqref{eq:Deep-n}, but replacing $N$ by $n$ and $H$ by $K$) contains the indices $x \in \llbracket-n, n\rrbracket$ with  $K_x \geq n^{\frac{3}{2} + \frac{\varepsilon}{12}}$. 

    Then, we split the interval $\llbracket-n, n\rrbracket$ in sub-intervals of length $\tilde{\ell}(n) \coloneqq \lfloor\frac{1}{3}n^{\frac{3\gamma}{2} + \frac{\varepsilon\gamma}{6}}\rfloor$. Similarly as in the proof of Proposition~\ref{prop:KeyLTEstimate}, we can compute
    \begin{equation}
        \mathbf{P}\left(I_{0, 1} \cap \mathrm{Deep}(n) = \varnothing\right) \le \exp\left( - \frac{1}{6} n^{-\frac{\varepsilon\gamma}{12}}\right).
    \end{equation}
    We can define the random variable $\rho_x$ as
    \begin{equation}
        \rho_x \coloneqq \sup\{k \ge 0 \colon \mathds{1}_{\hat{O}_k(x)} = 1\},
    \end{equation}
    with
    \begin{equation}
\label{eq:O-hat-n-def}
    \hat{O}_n(x) \coloneqq \bigcap_{k = - \lfloor n/\tilde{\ell}(n)\rfloor}^{ \lfloor n/\tilde{\ell}(n)\rfloor} \{I_{k,x} \cap \mathrm{Deep}(n) \neq \varnothing \}.
\end{equation}
    The rest of the proof follows as that of Proposition~\ref{prop:KeyLTEstimate} and Lemma~\ref{lemma:ExitCombExp} with the necessary adaptations in the exponents, substituting $\frac{2}{\gamma+1}$ by $3/2$.
\end{proof}

We now state an upper bound for the heat kernel. 

\begin{prop}\label{eqn:HKUBFromOrigin}
    For all $\alpha_1 > 0$ small enough, there exists $c_{19},c_{20} >0$ such that $\mathbf{P}$-a.e.~$\omega \in \Omega$ there exists $n_1 = n_1(\omega) >0$ such that for all $n \ge n_1(\omega)$ and all $x = (k, h) \in \mathrm{Comb}(\mathbb{Z},K^\omega)$ we have  
    \begin{equation}
        p_n(0, (k, h)) \le c_{19} n^{-(\frac{2}{3} + \alpha_1)} e^{-c_{20}\frac{h^2}{n}}.
\end{equation}
\end{prop}
\begin{proof}
    The proof follows using the arguments in Lemma~\ref{lem:OnDiagonalHKUB} and Lemma~\ref{lem:UBHKInsideAllPoints} (which originate from \cite[Lemma~4.9]{BPS}), we omit the details.
\end{proof}

\begin{cor}\label{cor:HKUBFromOrigin}
    For all $\alpha_2>0$ small enough there exists $c_{21}>0$ such that $\mathbf{P}$-a.s.\ there exist $k_0 = k_0(\omega)$ such that for all $n \ge n_1$, all $k \ge k_0$ and all $x = (k, h) \in \mathrm{Comb}(\mathbb{Z},K^\omega)$ we have that
    \begin{equation}
        p_n(0, (k, h)) \le \begin{cases}
            c_{21} n^{-(\frac{2}{3} + \alpha_1)} & n \ge |k|^{3 + \alpha_2}\\
            c_{21} |k|^{-(2 + \alpha_1)} & n \le |k|^{3 + \alpha_2}.
        \end{cases}
    \end{equation}
\end{cor}
\begin{proof}
    The first part is a direct consequence of Proposition~\ref{eqn:HKUBFromOrigin}. The second one, follows from the same proposition and Lemma~\ref{lemma:ExitCombExpHeavy} with the argument given in the proof of Proposition~\ref{prop:HKUBOutsideBox}, we omit repeating it.
\end{proof}
Similarly as in~\cite{KP}, we introduce the quantities
\begin{equation}
    \begin{split}
        O_{n, \ell} &\coloneqq |\{k \in \{n, 2n\} \colon X_n^1 = X_n^2 = v = (v_1, v_2)\colon v_2 \in \llbracket\ell, 2\ell\rrbracket \}|,\\
        A_{n, \ell}&\coloneqq \{O_{n, \ell} > 0\},\\
        W_{n, \ell} &\coloneqq \sum_{k \in \{ \frac{\ell}{2}, \ell, 2\ell \}} O_{n, k} + \sum_{k \in \{ \frac{\ell}{2}, \ell, 2\ell \}} O_{2n, k}.
    \end{split}
\end{equation}
We will utilize these quantities, and in particular the events $A_{n, \ell}$ to show the finite collision property on the graph $\mathrm{Comb}(\mathbb{Z},K^\omega)$. To that end, we give bounds on expectations and conditional expectations involving the event $A_{n, \ell}$ and the quantity $W_{n, \ell}$. We remark that the terms $O_{j, k}, j =n, 2n, k \in \{ \ell/2, \ell\}$ are present in the definition of $W_{n, \ell}$ because on one side they do not affect the upper bound on the expectation while giving room (in time and space) for the lower bound of $E^\omega_{0, 0}\left[W_{n, \ell} \mid A_{n, \ell}\right]$.
\begin{prop}\label{cor:ExpectationUBFromOrigin}
        For all $\alpha_3>0$ small enough there exist constants $C_{22}, c_{23}, C_{24}>0$ such that $\mathbf{P}$-a.s.\ for all $\eta<1$ and all $n \ge n_1$ 
        \begin{equation}\label{eqn:ExpUB}
            E^\omega_{0, 0}\left[W_{n, \ell}\right] \le C_{22}\frac{\ell}{n^{\alpha_3}},
        \end{equation}
        and 
        \begin{equation}\label{eqn:CondExpLB}
            E^\omega_{0, 0}\left[W_{n, \ell} \mid A_{n, \ell}\right] \ge c_{23}\ell^{\eta} \quad \text{for } \ell \le 2(2n)^{\frac{1}{2\eta}}.
        \end{equation}
        As a direct consequence of these two facts
        \begin{equation}\label{eqn:ProbUBFromOrigin}
            P^\omega_{0, 0}(A_{n, \ell}) \le C_{24}\ell^{1 - \eta}n^{-\alpha_3}\quad \text{for } \ell \le 2(2n)^{\frac{1}{2\eta}}.
        \end{equation}
\end{prop}

\begin{proof}
    Fix any $t \in \llbracket n, 2n\rrbracket$ and $h \in \llbracket \ell, 2\ell\rrbracket$. We write, comparing sum and integral in a standard way
    \begin{equation}
    \begin{split}
        E^\omega_{0, 0}\left[W_{n, \ell}\right] &\le n \ell \Big( \sum_{|k| \le n^{\frac{1}{3+ \alpha_2}} } p_n(0, (k, h)) + \sum_{|k| > n^{\frac{1}{3+ \alpha_2}} } p_n(0, (k, h)) \Big)\\
        &\le n \ell \left( n^{\frac{1}{3+ \alpha_2}} n^{-\frac{4}{3} - 2 \alpha_1} + 2\int_{n^{\frac{1}{3+ \alpha_2}}}^\infty k^{-(4 + 2\alpha_1)} dk \right)\\
        & \le \ell n^{-\alpha_3} + C\ell n^{-\frac{3 + 2 \alpha_1}{3+\alpha_2} + 1} \le C_{22} \ell n^{-\alpha_3}.
            \end{split}
    \end{equation}
    Note that in the last step we choose $\alpha_2$ small enough compared to $2\alpha_1$, which is justified by Corollary~\ref{cor:HKUBFromOrigin}.

    For the second statement we follow the lines of \cite[Lemma~2.2]{KP}. We observe that on $A_{n, \ell}$ the two walks meet on some segment at height $\llbracket\ell, 2\ell\rrbracket$ before time $n$. Hence, we can lower bound $W_{n, \ell}$ by the number of collisions that the two walks make before exiting the segment $\llbracket \ell/2, 4\ell\rrbracket$. Then by the resistance estimate given in \cite[Lemma~3.2]{BPS} (stated there in terms of killed Green's function), we have that 
    \begin{equation}
        E^\omega_{0, 0}\left[W_{n, \ell} \mid A_{n, \ell}\right] \ge \frac{1}{2} \frac{\ell}{2},
    \end{equation}
    as the resistance on a segment is linear in its length. We remark that in a random tooth the site of the meeting $h \in \llbracket\ell, 2\ell\rrbracket$ can be the tip of said tooth. However, the segment $\llbracket\ell/2, \ell\rrbracket$ is always non-empty if a meeting happened at height $h \in \llbracket\ell, 2\ell\rrbracket$.

    The third fact follows from the first two and the inequality
    \begin{equation}\label{eqn:Conditioning}
        P^\omega_{0, 0}(A_{n, \ell}) \le \frac{E^\omega_{0, 0}\left[W_{n, \ell}\right]}{E^\omega_{0, 0}\left[W_{n, \ell} \mid A_{n, \ell}\right]}.
    \end{equation}
    We observe that for sites at height $0, 1$ (where the lower bound does not work as smoothly), we always have probability at least $1/9$ that the next step after a collision is a collision, this is sufficient for our purpose.
\end{proof}

We are now in position to conclude the proof of Theorem~\ref{theo:FinitePrecise}.

\begin{proof}[Proof of Theorem~\ref{theo:FinitePrecise}.]

We follow~\cite{KP}. Note that almost surely there are finitely many collisions before the finite random time $n_1$, when our bounds become effective. Let $r, \ell$ satisfy $\ell \le 2(2n)^{\frac{1}{2\eta}}$ and $2^r \ge n_1$, then summing over dyadic scales one obtains
\begin{equation}
\begin{split}
    \sum_{r = 0}^\infty \sum_{k = 0}^{\lceil 1 + \frac{r+1}{2\eta}\rceil} P^\omega_{0, 0}(A_{2^r, 2^k}) &\le \sum_{r = 0}^\infty \sum_{k = 0}^{\lceil 1 + \frac{r+1}{2\eta}\rceil} C_{24}\frac{2^{k(1 - \eta)}}{2^{r\alpha_3}} \le \sum_{r = 0}^\infty C\frac{2^{r\frac{(1 - \eta)}{2\eta}}}{2^{r\alpha_3}} < \infty,
\end{split}    
\end{equation}
provided we choose $\eta$ such that
\begin{equation}
    \frac{(1 - \eta)}{2\eta} < \alpha_3 \iff \eta(1 + 2\alpha_3) > 1 \iff \eta > \frac{1}{1 + 2\alpha_3}.
\end{equation}
Since $\frac{1}{1 + 2\alpha_3} < 1$, we can choose $\eta \in (0,1)$ as required. Furthermore, we observe that $\frac{1}{2\eta}>\frac{1}{2}$, hence we need to exclude the collisions at heights higher than $2(2n)^{\frac{1}{2\eta}}$. We claim that the set
\begin{equation}
    \{n \colon V^{1}_n > 2(2n)^{\frac{1}{2\eta}} \text{ or }V^{2}_n > 2(2n)^{\frac{1}{2\eta}}\},
\end{equation}
contains finitely many elements almost surely. Indeed, recall the result of Proposition~\ref{eqn:HKUBFromOrigin}, then for all $(k, h) \colon h \ge 2(2n)^{\frac{1}{2\eta}}$
\begin{equation}
        p_n(0, (k, h)) \le c_{19} n^{-(\frac{2}{3} + \alpha_1)} e^{-c_{20}\frac{h^2}{n}}.
\end{equation}
However, if we pick $C, C'>0$ large enough, we obtain that 
\begin{equation}\label{eqn:EndDouble}
    P_{0}^\omega\left( V_{n}^1 \ge n^{1/2}C \log(n) \right) \le c_{19} n^2 e^{-c_{20}\frac{(Cn^{1/2} \log(n)))^2}{n}} \le C' n^{-3},
\end{equation}
and the same holds for $V^2$. Indeed, this is easy to obtain for a fixed $(k, h)$ with $h \geq 2(2n)^{\frac{1}{2\eta}}$ and then a union bound suffices as in $n$ steps there is at most a polynomial number of sites that are reachable by the random walk (deterministically less than $n^2$). Finally, integrating \eqref{eqn:EndDouble} over $n$ yields the event $\{V^j_{n} \ge n^{1/2}C \log(n)\}, j=1,2$ happens only finitely many times almost surely by applying the first Borel-Cantelli lemma. We also observe that for all $n$ large enough $Cn^{1/2} \log(n) \le n^{\frac{1}{2\eta}}$ so that eventually
\begin{equation}
    \{n \colon V^{1}_n > 2(2n)^{\frac{1}{2\eta}} \text{ or }V^{2}_n > 2(2n)^{\frac{1}{2\eta}}\} \subseteq \{n \colon V^{1}_n > Cn^{1/2}\log(n) \text{ or }V^{2}_n > Cn^{1/2}\log(n)\}.
\end{equation}
This finishes the proof.
\end{proof}

\section{Triple collisions in the random comb}
\label{sec:Triple-collisions-random}

In this Section we analyze triple collisions on the random comb model, with teeth generated by a distribution as in \eqref{eq:Heavy-tailed-def}. Throughout the Section, we use the same convention on notation and abbreviations as in Section~\ref{sec:Double-collisions-random} (see above Subsection~\ref{sect:EnvEstimates}).

We recall the result~\eqref{eq:Chen-Chen-triple-coll} from \cite{Chenchen} that a typical realization $\mathrm{Comb}(\mathbb{Z},H^\omega)$ of the comb with an i.i.d.\ tooth profile whose random length has finite first moment has the infinite triple collision property. The first result of this Section is essentially derived from \cite{CroydonDeAmbroggio}. Together with our heat-kernel estimates, it immediately yields that for all $\gamma \in (0,1)$ the random comb has the finite triple collision property almost surely. The bound is proved as the penultimate step in the proof of~\cite[Lemma~2.1]{CroydonDeAmbroggio}.

\begin{lem}\label{lem:Umbi}
    On any connected graph $G = (V,E)$ of with $\sup_{x \in V} \mathrm{deg}(x) \leq 3$ and $o \in V$, we have
    \begin{equation}
        P^G_{o, o, o}\left( X_n^1 = X_n^2 = X_n^3 \right) \le 9 \sup_{x \in B(o, n)} \sqrt{p_{2\lfloor n/2 \rfloor}^G(o, o)p^G_{2\lceil n/2 \rceil}(x, x) } p^G_{2n}(o, o).
    \end{equation}
\end{lem}
We state the result concerning the triple collision property of random comb graphs with $\gamma \in (0, 1)$.
\begin{cor}
\label{cor:Triple-coll-gamma-smaller-1}
    For all $\gamma \in (0,1)$ and $\mathbf{P}$-a.e. $\omega \in \Omega$, the graph $\mathrm{Comb}(\mathbb{Z},H^\omega)$ has the finite triple collision property.
\end{cor}
\begin{proof}
    A minor modification of Lemma~\ref{lem:OnDiagonalHKUB} yields that for $\textbf{P}$-a.e.~$\omega \in \Omega$ there exists $c> 0$ and $n_0(\omega) \in \mathbb{N}$  such that for $n \ge n_0(\omega)$, 
    \begin{equation}
        p^\omega_n(0, 0) \le c n^{-\frac{3 - \gamma}{4}}.
    \end{equation}
    Since $\frac{3 - \gamma}{4} > \frac{1}{2}$, we can use Lemma~\ref{lem:Umbi} to see that there exists $\delta>0$ such that for $n \geq n_0(\omega)$,
    \begin{equation}
        P^\omega_{0, 0, 0}\left( X_n^1 = X_n^2 = X_n^3 \right) \le C n^{-1 - \delta}.
    \end{equation}
    It follows that the expected number of triple collisions under $P_{0,0,0}^\omega$ is finite, whence the number of triple collisions is $P_{0,0,0}^\omega$-a.s.\ finite.
\end{proof}

We now focus on the more challenging case $\gamma = 1$. Since $|B(0, n)|$ is the sum of $n$ i.i.d.\ random variables with tail $\mathbf{P}(H>t) \sim Ct^{-1}$, one expects that $|B(0, n)| \approx n \log(n)$. This in turn suggests a heat-kernel decay at $0$ of the form $n^{-1/2}\log(n)^{-1/2}$. 
This corresponds to the `boundary case' between the finite and infinite triple collisions in \cite{CroydonDeAmbroggio}. There, the underlying graph is a wedge comb with logarithmic growth profile, and three independent random walks collide infinitely many times almost surely. Despite the fact that the heat kernels for the wedge comb $\mathrm{Comb}(\mathbb{Z},J^{(1)})$ with growth profile $J^{(1)}(k) = \log(|k| \vee 1)$ and a typical realization of the random comb $\mathrm{Comb}(\mathbb{Z},H^\omega)$ with tail parameter $\gamma = 1$ turn out to be comparable, we will show that finitely many triple collisions occur for this choice almost surely. The main reason causing this contrasting behavior is that the volumes are achieved with very different geometries in the two cases. This highlights that collisions properties are extremely sensitive to the geometry of graphs, and fine information on the heat kernel does not appear to be sufficient for a qualitative classification. We may also compare our set-up to~\cite{croydon2026collision}, where the infinite triple collision property is established for a typical realization of the trace of the four-dimensional random walk. Their result is obtained through a very fine analysis of the random graph, its geometry and corresponding heat-kernel estimates. The heat-kernel estimate of order $n^{-1/2}\log(n)^{-1/2}$ also appears there (whenever a set is good in a suitable sense). We again believe that the different behavior is a consequence of the very different way in which the volume is distributed in the comb graph in comparison to the trace of a random walk.

Here is the main result of the present Section.

\begin{thm}\label{theo:MainRandomCombTrip}
	Let $\gamma = 1$. For $\mathbf{P}$-a.e.\ $\omega \in \Omega$,  $\mathrm{Comb}(\mathbb{Z},H^\omega)$ has the finite triple collision property.
\end{thm}
The remainder of this Section is concerned with the proof of Theorem~\ref{theo:MainRandomCombTrip}, which will require several intermediate steps. Recall that $B(0, n)$ denotes the graph metric ball around $0$ and the set $D_k$ is defined in \eqref{eq:Sets-on-Comb}.

Let us start by introducing some notation. Let $h_n \coloneqq \lfloor 2^{\log_2^{1/3}(n)}\rfloor$, and define
\begin{equation}
    B^{\mathrm{up}}(0, n) \coloneqq \{ v = (x, y) \in B(0, n) \colon y \ge h_n\}.
\end{equation}
We also write $B^{\mathrm{down}}(0, n) \coloneqq B(0, n) \setminus B^{\mathrm{up}}(0, n)$. For $\dagger \in \{\mathrm{up}, \mathrm{down}\}$ we introduce the events
\begin{equation}
    \begin{split}
        O_{n}^{\dagger} &\coloneqq |\{k \in \{n, 2n\} \colon X_k^1 = X_k^2 = X_k^3 = v, v \in B^{\dagger}(0, 2n) \}|,\\
        A_n^{\dagger}&\coloneqq \{O_{n}^{\dagger} > 0\},\\
        W_n^{\dagger} &\coloneqq O_{n}^{\dagger} + O_{2n}^{\dagger}.
    \end{split}
\end{equation}
We now state the main auxiliary result before proving Theorem~\ref{theo:MainRandomCombTrip}.
\begin{prop}\label{prop:KeyTriple}
    There exist constants $C, \delta>0$ such that for $\mathbf{P}$-a.e.\ $\omega \in \Omega$ and $n$ large enough, for $\dagger \in \{\mathrm{up}, \mathrm{down}\}$
    \begin{equation}
        P_{0, 0, 0}^\omega\left( A_n^{\dagger} \right) \le C \log(n)^{-(1+\delta)}.
    \end{equation}
\end{prop}

Before proceeding with the proof of this result we show how to deduce Theorem~\ref{theo:MainRandomCombTrip} from it. 

\begin{proof}[Proof of Theorem~\ref{theo:MainRandomCombTrip}]
    Observe that if there are finitely many triple collisions almost surely in $B^{\mathrm{up}}(0, 2n)$ and $B^{\mathrm{down}}(0, 2n)$, then there are finitely many triple collisions altogether as they exhaust the possibilities. Then, by assuming the result of Proposition~\ref{prop:KeyTriple}, we see that $\mathbf{P}$-a.s.\ for $\dagger \in \{\mathrm{up}, \mathrm{down}\}$, 
    \begin{equation}
        \sum_{k = 1}^\infty P_{0, 0, 0}^\omega\left( A_{2^k}^{\dagger} \right) \le C \sum_{k = 1}^\infty k^{-(1+\delta)}<\infty.
    \end{equation}
    This shows that 
    \begin{equation}
         P_{0, 0, 0}^\omega\left(X_n^1 = X_n^2 = X_n^3 = v, v \in B^{\dagger}(0, n), \text{ infinitely often}\right) \le \sum_{k = 1}^\infty P_{0, 0, 0}^\omega\left( A_{2^k}^{\dagger} \right) < \infty,
    \end{equation}
    finishing the proof.
\end{proof}

We split the proof of Proposition~\ref{prop:KeyTriple} into two cases, each dealing with to one of the events $A_n^{\mathrm{up}}$ and $A_n^{\mathrm{down}}$.

\subsection{Environment estimates and heat kernel}

\begin{lem}\label{lem:EnvGamma1}
    There exist constants $c_{\mathrm{v}}, C_{\mathrm{v}}$ and, for every $w = (w_1, 0)$, there exists a $\mathcal{G}$-measurable random variable $\eta_{w_1}$ with the property that
    \begin{equation}\label{eq:StrechedExp}
        \mathbf{P}\left(\eta_{w_1} \ge n \right) \le Ce^{-cn^{1/3}},
    \end{equation}
    such that the following holds. For all $n \ge \eta_{w_1}$ and for all $\ell \in \llbracket 0,  \tfrac{1}{2}\log_2(n)  \rrbracket$ we have that 
    \begin{equation}\label{eq:GoodVolLog}
        |B(w, n)| \ge c_{\mathrm{v}} n \log(n), \qquad \text{and} \qquad |\{v = (x, y) \colon v \in B(0, n), y \in \llbracket 2^\ell, 2^{\ell+1}\rrbracket\}| \le C_{\mathrm{v}} n.
    \end{equation}
\end{lem}
\begin{proof}
    The proof follows the same strategy as the those of Lemma~\ref{lem:GoodBoxes} and Lemma~\ref{cor:GoodVolumes}. Let $w_1=0$ throughout the proof. We observe that, for any $k \ge 1$,
    \begin{equation}
      \mathcal{R}^0_{k/2}(k/2) \subseteq   B(0, k) \subseteq \mathcal{R}^0_k(k).
    \end{equation}
    Therefore, proving the lemma with $\mathcal{R}^0_{n}(n)$ automatically gives the result for $B(0, n)$.

    We have that 
    \begin{equation}
		|\{x \in \llbracket -n, n \rrbracket \colon H_x \in [2^{\ell}, 2^{\ell+1})\}| \sim \mathrm{Binom}\left( 2n+1, C2^{-(\ell+1) } \right).
	\end{equation}
    Thus, there exists a constant $\mu>0$ such that
    \begin{equation}
    	\mathbf{P}\left(|\{x \in \llbracket -n, n \rrbracket \colon H_x \in [2^{\ell}, 2^{\ell+1})\}| > 4 \mathbf{E}\left[|\{x \in [-k, k] \colon H_x \in [2^{\ell}, 2^{\ell+1})\}|\right] \right) \le e^{-\mu n 2^{-\ell}}.
    \end{equation}
    The same bound holds for the lower tail. We observe that over the range considered,
    \begin{equation}
        \inf_{\ell \le \lfloor\frac{1}{2}\log_2(n) \rfloor} \mu n 2^{-\ell} \ge \mu n^{1/3}.
    \end{equation}
    We define
    \begin{equation}
        A_k \coloneqq \left\{ \frac{|\{x \in \llbracket -k, k \rrbracket \colon H_x \in [2^{\ell}, 2^{\ell+1})\}|}{\mathbf{E}\left[|\{x \in \llbracket -k, k \rrbracket\colon H_x \in [2^{\ell}, 2^{\ell+1})\}|\right] } \in [c_{\mathrm{vol}}^{-1}, c_{\mathrm{vol}}], \text{ for all } \in \llbracket 0,  \tfrac{1}{2}\log_2(n) \rrbracket  \right\}.
    \end{equation}
    We now set $\eta_0 \coloneqq \sup\{k \ge 0 \colon  \mathds{1}_{A_k} = 0\}$ and observe that
    \begin{equation}
        \mathbf{P}\left(\eta_0 \ge n\right) \le \sum_{k = n}^\infty \mathbf{P}\left(A_k^c\right) \le \sum_{k = n}^\infty e^{-\mu k^{1/3}} \le C e^{-\mu' n^{1/3}},
    \end{equation}
    for some constants $C, \mu'>0$.

    A Borel-Cantelli argument implies the second display in \eqref{eq:GoodVolLog} and also that there exists $c'_{\mathrm{v}}>0$ such that almost surely 
    \begin{equation}
        |\{x \in \llbracket -n, n \rrbracket \colon H_x \in [2^{\ell}, 2^{\ell+1})]\}| \ge c'_{\mathrm{v}} 2^{-\ell} n.
    \end{equation}
    The volume estimate follows from 
    \begin{equation}
    \begin{split}
        |B(0, n)| &\ge \sum_{\ell = 1}^{\lfloor \tfrac{1}{2}\log_2(n)\rfloor}|\{v = (x, y) \colon v \in B(0, n), y \in \llbracket 2^\ell, 2^{\ell+1}-1\rrbracket\}| 2^{\ell}\\
        &\ge  c'_{\mathrm{v}} \sum_{\ell = 1}^{\lfloor \tfrac{1}{2}\log(n)\rfloor} n \ge \frac{c'_{\mathrm{v}}}{4} \log(n).
        \end{split}
    \end{equation}
    This concludes the proof for $w_1= 0$. The translation invariance of the environment implies the statement for all other values of $w_1 \in \mathbb{Z}$.
\end{proof}

We are now in position to deduce a heat-kernel estimate.

\begin{prop}\label{prop:HKTriple}
    There exists a constant $C_{\mathrm{up}}^*>0$ such that for $\mathbf{P}$-a.e. $\omega \in \Omega$, there exists $n_0 = n_0(\omega) > 0$ such that for all $n \ge n_0$
    \begin{equation}
        p_{2\lfloor n/2\rfloor}^\omega(0, x) \le C_{\mathrm{up}}^* n^{-1/2}\log^{-1/2}(n) \qquad \text{for all } x = (x_1, x_2).
    \end{equation}
\end{prop}
\begin{proof}
    We will proceed by establishing the on-diagonal bound on the event $\eta_0 \le n_0^{1/2}\log^{-1/2}(n_0)$ of Lemma~\ref{lem:EnvGamma1} and then deduce the off-diagonal bound using Cauchy-Schwarz. 

    Recall from~\cite[(2.2)]{CroydonDeAmbroggio} that for all $r \ge 0$,
    \begin{equation}
        p_{2\lfloor n/2\rfloor}^\omega(0, 0) \le \frac{4r}{\lfloor n/2 \rfloor} + \frac{2}{|B(0, r)|}.
    \end{equation}
    If we choose $r = n^{1/2}\log^{-1/2}(n)$, by applying Lemma~\ref{lem:EnvGamma1} on the event $\{\eta_0 \le r\}$ and the fact that $\log(n^{1/2}\log^{-1/2}(n)) \ge \log(n)/4$ we obtain, for a suitably chosen $C_{\mathrm{up}}^{**}>0$,
    \begin{equation}\label{eqn:HKUBGamma1}
        p_{2\lfloor n/2\rfloor}^\omega(0, 0) \le C_{\mathrm{up}}^{**} n^{-1/2} \log^{-1/2}(n).
    \end{equation}
    Furthermore, by a standard Borel-Cantelli argument and \eqref{eq:StrechedExp} there exists $K_0 = K_0(\omega)>0$ such that, for all $k \ge K_0$, we have $\{\eta_{\pm k} \le k/2 \}$. Thus, on these events, for all $n \ge \eta_k$ (and symmetrically for $\eta_{-k}$) we obtain
    \begin{equation}\label{eqn:HKUBGamma1Bis}
        p_{2\lfloor n/2\rfloor}^\omega((k, 0), (k, 0)) \le C_{\mathrm{up}}^{**} n^{-1/2} \log^{-1/2}(n).
    \end{equation}
    For the off-diagonal bound, we observe that $p_{2\lfloor n/2\rfloor}(0, x) = 0$ if $n \le k$ and $x= (\pm k, 0)$. For all $n \ge k/2 \vee \eta_0$ we obtain, by applying the Cauchy-Schwarz inequality as in \eqref{eqn:CauchySchwarz},
    \begin{equation}
        p_{2\lfloor n/2\rfloor}^\omega(0, x) \le C_{\mathrm{up}}^{**} n^{-1/2} \log^{-1/2}(n).
    \end{equation}
    For all points $x =(x_1, x_2)$ one can obtain the bound 
    \begin{equation}
        p_{2\lfloor n/2\rfloor}^\omega(0, x) \le C_{\mathrm{up}}^{*} n^{-1/2} \log^{-1/2}(n) e^{- c\frac{x_2^2}{n}},
    \end{equation}
    which directly implies the statement of the lemma. The argument follows analogously to the one given in Lemma~\ref{lem:UBHKInsideAllPoints}, we omit repeating it.
\end{proof}

We now deduce a heat-kernel bound for sites that are far from the starting point.

\begin{lem}\label{lem:HKUBFromOriginGamma1}
    There exists $C_{\mathrm{out}}>0$ such that, for all $n \ge \eta_0$ and all $x = (k, h) \in \mathrm{Comb}(\mathbb{Z},H^\omega)$ for which $|k|\ge K_0 \vee \sqrt{n}$, we have 
    \begin{equation}
        p_n(0, (k, h)) \le 
            C_{\mathrm{out}} |k|^{-1}\log(|k|)^{-1/2}.
    \end{equation}
\end{lem}
\begin{proof}

    The argument follows from the heat kernel bound given in \eqref{eqn:HKUBGamma1} and goes along the same lines as Proposition~\ref{prop:HKUBOutsideBox}. The key exit time estimate
    \begin{equation}
        P_0^\omega\left(\tau_{D_k^c} \le n \right) \le C\exp\left(- c \frac{k^2}{n}\right),
    \end{equation}
    follows in a straightforward manner by coupling the random walk on the comb and simple random walk on an interval.
\end{proof}

\subsection{The events $A_n^{\mathrm{up}}$}

The following lemma implies immediately Proposition~\ref{prop:KeyTriple} for the event $A_n^{\mathrm{up}}$.

\begin{lem}
    There exists two constants $c_{\mathrm{up}}, C_{\mathrm{up}} >0$ such that $\mathbf{P}$-a.s.\ for all $n$ large enough
    \begin{equation}
        \begin{split}
            E_{0, 0, 0}^\omega[W_n^{\mathrm{up}}] &\le  C_{\mathrm{up}}\log^{-3/4}(n), \\
            E_{0, 0, 0}^\omega[W_n^{\mathrm{up}} \mid A_n^{\mathrm{up}} ] &\ge c_{\mathrm{up}}\log^{1/3}(n).
        \end{split}
    \end{equation}
\end{lem}

\begin{proof}
    To show the first statement, we aim to apply Proposition~\ref{prop:HKTriple}. Inserting the on-diagonal estimate into the bound stated in Lemma~\ref{lem:Umbi} and recalling that $p_n^\omega(x, x) \le C_{\mathrm{up}}^{**} n^{-1/2}$ for any $\omega \in \Omega$ (see, e.g.,~\cite[Corollary 4.4 (b)]{Barbook}), we obtain
    \begin{equation}
        E_{0, 0, 0}^\omega[W_n^{\mathrm{up}}] \le 4nC_{\mathrm{up}} n^{-1/4} n^{-1/4} \log^{-1/4}(n) n^{-1/2}\log^{-1/2}(n) \le C_{\mathrm{up}}\log^{-3/4}(n).
    \end{equation}
    The second statement follows from the fact that, after colliding in $B^{\mathrm{up}}(0, n)$, the three random walks are at height larger than $2^{\lfloor \log_2^{1/3}(n) \rfloor}$ and hence (using \cite[Lemma~A.3]{CroydonDeAmbroggio}) the collisions before exiting the tooth can be lower bounded by
    \begin{equation}
    c_{\mathrm{up}}'\log_2(2^{\lfloor \log_2^{1/3}(n) \rfloor}) = c_{\mathrm{up}}\log^{1/3}(n),
        \end{equation}
    for some $c_{\mathrm{up}}', c_{\mathrm{up}} > 0$.
\end{proof}

\subsection{The events $A_n^{\mathrm{down}}$}

The following lemma immediately implies Proposition~\ref{prop:KeyTriple} for the event $A_n^{\mathrm{down}}$.

\begin{lem}
    There exists a constant $C_{\mathrm{down}} >0$ such that $\mathbf{P}$-a.s.\, for all $n$ large enough
    \begin{equation}
            E_{0, 0, 0}^\omega[O_n^{\mathrm{down}}] \le  C_{\mathrm{down}}\log^{-1-\delta}(n).
    \end{equation}
\end{lem}

\begin{proof}
    We write $\tilde{B}^{\mathrm{down}}(0, 2n) \coloneqq \{v = (x, y) \in B^{\mathrm{down}}(0, 2n) \colon |x|\le 2n^{1/2}\}$. 
    We begin by splitting
    \begin{equation}
    \begin{split}
        E_{0, 0, 0}^\omega[O_n^{\mathrm{down}}] &= \sum_{k = n}^{2n} \sum_{z \in B^{\mathrm{down}}(0, 2n)} p^\omega_k(0, z)^3\\
        & \le \sum_{k = n}^{2n} \sum_{z \in \tilde{B}^{\mathrm{down}}(0, 2n)} p^\omega_k(0, z)^3 + \sum_{k = n}^{2n} \sum_{z \in B^{\mathrm{down}}(0, 2n) \setminus \tilde{B}^{\mathrm{down}}(0, 2n)} p^\omega_k(0, z)^3.
        \end{split}
    \end{equation}
    We will treat the two terms separately, and apply repeatedly Proposition~\ref{prop:HKTriple}. Let us begin with the first one. We have
    \begin{equation}\label{eq:DownComps}
    \begin{split}
        \sum_{k = n}^{2n} \sum_{z \in \tilde{B}^{\mathrm{down}}(0, 2n)} p^\omega_k(0, z)^3  &\le C_{\mathrm{up}}^* \sum_{k = n}^{2n} \sum_{z \in \tilde{B}^{\mathrm{down}}(0, 2n)} n^{-3/2} \log^{-3/2}(n)\\
        & \le C_{\mathrm{up}}^* C_{\mathrm{v}} n n^{1/2} \log^{1/3}(n) n^{-3/2} \log^{-3/2}(n) \\
        & \le \frac{C_{\mathrm{down}}}{2}\log^{-1-\delta}(n),
            \end{split}
    \end{equation}
    for some $\delta>0$. To obtain \eqref{eq:DownComps} we used the fact that \eqref{eq:GoodVolLog} implies the volume estimate 
    \begin{equation}
        \begin{split}|\tilde{B}^{\mathrm{down}}(0, 2n)| & \le \sum_{\ell = 1}^{\lceil \log_2(h_n)\rceil} |\{v = (x, y) \colon v \in B(0, n), |x|\le n^{1/2}, y \in \llbracket 2^\ell, 2^{\ell+1}\rrbracket\}| \\
        & \le 4C_{\mathrm{v}} n^{1/2}\log^{1/3}(n).    \end{split}
    \end{equation}
    We now focus on the second term. By Lemma~\ref{lem:HKUBFromOriginGamma1}, 
    \begin{equation}
    \begin{split}
        \sum_{k = n}^{2n} \sum_{z \in B^{\mathrm{down}}(0, 2n) \setminus \tilde{B}^{\mathrm{down}}(0, 2n)} p^\omega_k(0, z)^3
        &\le \sum_{k = n}^{2n} \sum_{j = \lfloor 2n^{1/2} \rfloor}^{2n} \sum_{h = 0}^{\lfloor H_j\rfloor}  p^\omega_k(0, (j, h))^3\\
        &\le 2n  C_{\mathrm{out}} \sum_{j = \lfloor 2n^{1/2} \rfloor}^{2n} \sum_{h = 0}^{\lfloor H_j\rfloor \wedge h_n}|j|^{-3}\log(|j|)^{-3/2}.
            \end{split}
    \end{equation}
    Thus, we only need to show that 
    \begin{equation}
        \sum_{j = \lfloor 2n^{1/2} \rfloor}^{2n} \sum_{h = 0}^{\lfloor H_j\rfloor \wedge h_n}|j|^{-3}\log(|j|)^{-3/2} \le \frac{C_{\mathrm{down}}}{2} n^{-1} \log^{-1 - \delta}(n).
    \end{equation}
    To proceed in this direction we split the sites in $B^{\mathrm{down}}(0, 2n) \setminus \tilde{B}^{\mathrm{down}}(0, 2n)$ according to their horizontal coordinate along a dyadic scale. Then the number of sites in $B^{\mathrm{down}}(0, 2n) \setminus \tilde{B}^{\mathrm{down}}(0, 2n)$ with horizontal coordinate in $\llbracket 2^\ell + 1, 2^{\ell+1}\rrbracket$ for $\ell \in \llbracket\log_2(n^{1/2}), \lceil\log_2(2n)\rceil\rrbracket$ is at most $2^{\ell} C_{\mathrm{v}} \log_2^{1/3}(2^\ell)$, by applying Lemma~\ref{lem:EnvGamma1}. Thus, we obtain
    \begin{equation}
    \begin{split}
        \sum_{j = \lfloor 2n^{1/2} \rfloor}^{2n} \sum_{h = 0}^{\lfloor H_j\rfloor \wedge h_n}|j|^{-3}\log(|j|)^{-3/2}&\le 2 C_{\mathrm{v}} \sum_{\ell = \lfloor \log(n^{1/2})\rfloor }^{\lceil \log(2n)\rceil} 2^{\ell} 2^{-3\ell} \ell^{-3/2 + 1/3} \\
        & \le 2 C_{\mathrm{v}} \sum_{\ell = \lfloor \log(n^{1/2})\rfloor}^{\lceil\log(2n)\rceil} 2^{-2\ell} \ell^{-3/2 + 1/3}\\
        & \le \frac{C_{\mathrm{down}}}{2} n^{-1} \log^{-1 - \delta}(n),
            \end{split}
    \end{equation}
    with $\delta = 1/6$ and $C_{\mathrm{down}}>0$ chosen appropriately. 
\end{proof}

\noindent \textbf{Acknowledgments.} MN was partially supported by Hong Kong RGC grants ECS 26301824 and GRF 16303825. The authors are grateful to Noam Berger and Manuel Cabezas for helpful discussions.

\appendix

\section{Some integral bounds}

\label{}

We compute basic integral bounds that are useful in our proofs. Let $k > 1, \beta > 1$
\begin{align}
    \int_k^\infty \frac{1}{x^4\log^{2\beta}(x)} \mathrm{d} x^3\log^\beta(x) &= \int_k^\infty \frac{3}{x^2\log^{\beta}(x)} \mathrm{d} x+ \int_k^\infty \frac{\beta}{x^2\log^{\beta + 1}(x)} \mathrm{d} x.
\end{align}
Observe that for large $k$, the first term on the right hand side dominates. To bound it from above, we pull out a factor $\log^\beta(k)$ and obtain for some $C>0$
\begin{equation}\label{Int:squuared}
    \int_k^\infty \frac{1}{x^4\log^{2\beta}(x)} \mathrm{d} x^3\log^\beta(x) \le \frac{C}{k \log^\beta(k)}.
\end{equation}
Let $M>1$ and $\alpha \in (0, 2]$, the second integral we compute is for any $b>1$ and some $C(\alpha, b)>0$
\begin{equation}\label{Int:xlog(x)}
    \int_b^M \frac{1}{x\log^{\alpha}(x)}\mathrm{d}x \stackrel{u=\log(x)}{=} \int_{\log(b)}^{\log(M)} u^{-\alpha} \mathrm{d}u \stackrel{M \to \infty}{\sim} \begin{cases}
        \frac{\log(M)^{1-\alpha}}{1-\alpha} \vee C(\alpha, b) & \alpha \in (0, 1) \cup (1, 2],\\
        \log(\log(M)) & \alpha = 1.
    \end{cases}
\end{equation}
These asymptotics are used throughout Section~\ref{sec:Regularly-growing-comb}.
\bibliographystyle{plain}
\bibliography{biblio}

\end{document}